\documentclass[a4paper,fleqn,11pt]{article}

\usepackage{amsmath}
\usepackage{amsthm}
\usepackage{amssymb}
\usepackage{accents}

\usepackage[a4paper,top=3cm, bottom=3cm, left=3cm, right=3cm]{geometry}
\usepackage[shortlabels,inline]{enumitem}
\usepackage{microtype}
\usepackage[pdftex,ocgcolorlinks,pagebackref=false]{hyperref}
\usepackage[affil-it]{authblk}

\setlist[enumerate,1]{label={(\roman*)}}

\usepackage[capitalise,noabbrev]{cleveref}

\Crefname{property}{Property}{Properties}

\theoremstyle{plain}
\newtheorem{theorem}{Theorem}[section]
\newtheorem{lemma}[theorem]{Lemma}
\newtheorem{proposition}[theorem]{Proposition}

\theoremstyle{definition}
\newtheorem{definition}[theorem]{Definition}
\newtheorem{corollary}[theorem]{Corollary}
\newtheorem{remark}[theorem]{Remark}
\newtheorem{example}[theorem]{Example}

\newcommand{\reals}{\mathbb{R}}
\newcommand{\complexes}{\mathbb{C}}
\newcommand{\integers}[1][]{\mathbb{Z}_{#1}}
\newcommand{\naturals}{\mathbb{N}}
\newcommand{\rationals}{\mathbb{Q}}
\newcommand{\positivereals}{\reals_{>0}}
\newcommand{\positiveintegers}{\mathbb{N}_{>0}}
\newcommand{\nonnegativereals}{\reals_{\ge 0}}
\newcommand{\tropicalreals}{\mathbb{TR}}
\newcommand{\opposite}{^\mathrm{op}}
\newcommand{\dualnumbers}{\mathbb{D}}

\newcommand{\nonzeros}[1]{#1_{\neq 0}}

\newcommand{\distributions}[1][]{\mathcal{P}_{#1}}
\DeclareMathOperator{\support}{supp}

\newcommand{\entropy}{H}
\newcommand{\relativeentropy}[3][]{\mathop{D_{#1}}\mathopen{}\left(#2\middle\|#3\right)\mathclose{}}
\newcommand{\mutualinformation}{I}

\newcommand{\typeclass}[2]{T^{#1}_{#2}}
\newcommand{\typeclassprojector}[2]{\Pi^{#1}_{#2}}

\newcommand{\graphs}{\mathcal{G}}

\newcommand{\typegraph}[3]{#1^{\strongproduct #2}[\typeclass{#2}{#3}]} 
\renewcommand{\complement}[1]{\overline{#1}}

\newcommand{\disjointunion}{\sqcup}
\newcommand{\strongproduct}{\boxtimes}

\DeclareMathOperator{\subrank}{Q}
\DeclareMathOperator{\tensorrank}{R}
\DeclareMathOperator{\asymptoticsubrank}{\undertilde{Q}}
\DeclareMathOperator{\asymptoticrank}{\undertilde{R}}
\newcommand{\unittensor}[1]{\langle{#1}\rangle}

\newcommand{\CW}[1]{\operatorname{CW}_{#1}}

\newcommand{\preorderle}{\preccurlyeq}

\newcommand{\asymptoticle}{\lesssim}
\newcommand{\asymptoticge}{\gtrsim}
\newcommand{\asymptoticeq}{\approx}

\DeclareMathOperator{\rank}{R}

\newcommand{\completion}[1]{\overline{#1}}

\newcommand{\sequencetimes}{\odot}
\newcommand{\sequenceplus}{\oplus}

\DeclareMathOperator{\Tr}{Tr}

\DeclareMathOperator{\Hom}{Hom}
\DeclareMathOperator{\End}{End}

\newcommand{\norm}[2][]{\left\|#2\right\|_{#1}}
\newcommand{\setbuild}[2]{\left\{#1\middle|#2\right\}}
\newcommand{\ubar}[1]{\underaccent{\bar}{#1}}

\author{Péter Vrana}
\title{Asymptotic completions of preordered semirings}
\affil{Department of Algebra and Geometry, Institute of Mathematics, Budapest University of Technology and Economics, Műegyetem rkp. 3., H-1111 Budapest, Hungary.}
\date{September 25, 2026}

\begin{document}
\maketitle

\begin{abstract}
The study of preordered semirings is motivated by applications in computer science, graph theory, and information theory, and provides tools for understanding the asymptotic preorder, which compares large powers of a pair of elements. This paper studies sequences which behave approximately as sequences of powers, but are not necessarily equivalent to geometric sequences. Our main result is that preordered semirings admit completions where such sequences, that we call approximately geometric, become equivalent to geometric sequences, and that existing characterizations of the asymptotic preorder extend to the completion.

We provide several classes of examples of approximately geometric sequences in the semiring of tensors, and in the semiring of graphs. As a concrete application, we determine the strong converse exponent for binary hypothesis testing with composite Markov hypotheses.
\end{abstract}

\section{Introduction}

Motivated by the study of the computational complexity of matrix multiplication, Strassen developed the theory of asymptotic spectra of preordered semirings \cite{strassen1987relative,strassen1988asymptotic,strassen1991degeneration} (see \cite{wigderson2026asymptotic} for a modern exposition). The central result of the theory is the characterization of the asymptotic preorder in terms of monotone homomorphisms from the semiring to the nonegative real numbers (which form the asymptotic spectrum), under some assumptions on the preordered semiring. Here the asymptotic preorder means inequalities of the form $a^n\le 2^{k_n}b^n$ with $n\to\infty$ and $k_n/n\to 0$, i.e., it compares geometric sequences up to a subexponential factor. Recently, the theory of preordered semirings has found applications in graph theory \cite{zuiddam2019asymptotic}, information theory and statistics \cite{jensen2019asymptotic,perry2022semiring,farooq2024matrix,verhagen2025matrix}, representation theory \cite{strassen2000asymptotic,fritz2024asymptotic}, and probability theory \cite{fritz2024characterizing}, in parallel with a considerable extension of the theory to semirings with weaker properties \cite{fritz2020generalization,vrana2021generalization,fritz2023abstract,fritz2021abstract2}. A common theme in these extensions is that the characterization of the asymptotic preorder requires homomorphisms into other ``test'' semirings in addition to the nonnegative reals.

In many of these applications, the significance of the asymptotic preorder is not immediately obvious. In the case of matrix multiplication, the semiring is the semiring of $3$-tensors (which encode bilinear maps) with direct sum and Kronecker product, and the relevant preorder corresponds to bilinear reductions. Bilinear complexity, which can be shown to be asymptotically the same as the arithmetic complexity, can be expressed by comparing with the unit tensors, i.e., the ``unit problem'' of independent scalar multiplications. Crucially, Kronecker products of matrix multiplication tensors are themselves matrix multiplication tensors, and Kronecker products of unit tensors are unit tensors, which explains why the asymptotic preorder between matrix multiplication tensors and unit tensors captures the growth of the complexity with the input size, at least on an exponential scale. More precisely, the asymptotic rank of the $n\times n$ matrix multiplication tensor is $n^\omega$ for some (unknown) constant $\omega$, the exponent of matrix multiplication.

In a similar way, the Shannon capacity of a graph measures the exponential growth of the independence number of the graph powers with respect to the strong product. In this context, Zuiddam considers the semiring of isomorphism classes of finite simple undirected graphs with disjoint union as addition, strong product as multiplication, and the cohomomorphism preorder \cite{zuiddam2019asymptotic}. The multiplicative unit is the graph with one vertex, and the edgeless graphs form a subsemiring that is isomorphic to the natural numbers. Crucially, the independence number can be characterized in terms of the preorder as the largest edgeless graph less than or equal to the graph in question, therefore the asymptotic preorder between edgeless graphs and an arbitrary graph provides information on the exponential growth of the independence number. The Shannon capacity is then identified as the asymptotic subrank in this semiring.

In the applications in information theory and statistics, the elements of the semiring are, roughly speaking, (unnormalized) probability distributions, quantum states, or parametrized families thereof, and product operations correspond to independently drawn random samples. Geometric sequences therefore capture sources of i.i.d.\ random variables or quantum states (which may play the role of hypotheses), which is a basic model studied in asymptotic questions in information theory and hypothesis testing.

Abstractly, the problems above are about understanding the properties of geometric sequences in preordered semirings, and the dual characterizations in terms of appropriate versions of the spectrum provide powerful tools for this. In this paper, we seek to generalize the theory to a larger class of sequences, thereby expanding the scope in ways that are meaningful in the applications. For instance, as explained above, it is a special property of the family of matrix multiplication tensors that it is closed under the tensor product, and there is no reason to expect that other bilinear computational tasks form geometric families (or have such subfamilies). In this direction, sequences of tensors resembling Kronecker powers have been studied in \cite{bjorklund2026kronecker}. Likewise, in information theory, more realistic models allow correlations between different samples or channel uses, which no longer seems to fit in the geometric sequence picture.

In fact, more general sequences are useful even in the study of geometric sequences as intermediate steps. In the case of bilinear complexity, in practice, bounds on the tensor rank of powers are frequently proved by a chain of reductions, where the intermediate steps may not be Kronecker powers. A mild class of examples are sequences of the form $T_1^{\otimes\lfloor an\rfloor}\otimes T_2^{\otimes \lfloor bn\rfloor}$ with $a,b>0$, which form the basis of a mixed method to get an upper bound on the complexity of rectangular matrix multiplication \cite{coppersmith1997rectangular,christandl2025barriers}. When $a/b\in\rationals$, this contains geometric subsequences as $n$ runs in multiples of some $m\in\positiveintegers$, but in general it only behaves similarly to sequences of the form $T^{\otimes n}$. As another example from graph theory, certain non-geometric sequences of graphs have been essential in obtaining convexity of the asymptotic spectrum of graphs \cite{vrana2021probabilistic}. The present paper provides a framework in which these intermediate sequences can be conveniently placed as well.

\subsection{Results}

We aim to prove our abstract results in a generality that not only applies to semirings satisfying Strassen's conditions, but in more general settings that include all the applications above as well. In particular, we do not assume that $0\le 1$ holds in the semiring. In some cases, however, stronger conclusions or a more complete discussion is possible if we make this assumption.

Our starting point is the observation that, under the mapping of elements of a semiring to the geometric sequence that they generate ($s\mapsto(1,s,s^2,s^3,\dots)$), the semiring operations translate to binomial convolution and elementwise multiplication, and the asymptotic preorder translates the elementwise comparison of the sequences up to a sequence of powers with sublinear exponent (for this we assume the existence of a special element $u$, explained in detail in \cref{sec:preliminaries}), while semiring homomorphisms can be applied elementwise. These operations and relation on the sequences satisfy the compatibility required for a preordered semiring \emph{except distributivity}.
\begin{itemize}
\item We identify a class of sequences that we call \emph{approximately geometric}, prove that it is closed under elementwise multiplication, binomial convolution, and asymptotic equivalence, and within the class, distributivity is satisfied up to asymptotic equivalence.
(\cref{sec:approximatelygeometric})

\item We introduce the property of \emph{asymptotic completeness} of a preordered semiring, requiring that approximately geometric sequences are asymptotically equivalent to geometric ones. We prove that every semiring can be ``almost'' embedded (via a map that only identifies asymptotically equivalent elements) into an asymptotically complete preordered semiring. Moreover, there is an essentially unique ``minimal'' way to do this, with asymptotic equivalence classes of approximately geometric sequences as the elements, and operations and preorder induced by the ones above. It satisfies a universal property with respect to monotone homomorphisms into asymptotically complete \emph{closed partially ordered semirings}.
(\cref{sec:asymptoticcompleteness})

\item Assuming that the asymptotic preorder of a preordered semiring $S$ is characterized by homomorphisms into a set $\mathcal{K}$ of test semirings, all of which are asymptotically complete, closed, and partially ordered (as in the existing characterization theorems), we prove that the canonical extensions of the homomorphisms to the asymptotic completion $\completion{S}$ characterize the order on $\completion{S}$, or equivalently, the asymptotic preorder on approximately geometric sequences in $S$. Under Strassen's conditions, we show that a class of functionals with weaker properties (i.e., which are not elements of the spectrum), the regular upper and lower functionals on $S$ extend in a unique way to functionals $\completion{S}\to\nonnegativereals$ with similar properties.
(\cref{sec:spectrum})

\item As the first (and still abstract) examples, we construct approximately geometric sequences that behave like powers of some fixed (suitable) approximately geometric sequence with non-integer exponents. Under Strassen's conditions on the semiring, every approximately geometric sequence is suitable for the construction. Applied to $\naturals\subseteq S$, they provide direct rank-type characterizations of the \emph{asymptotic rank} and \emph{asymptotic subrank}.
(\cref{sec:powers})

\item We find a sufficient condition for a linear recursion in a preordered semiring to generate an approximately geometric sequence, which is, up to asymptotic equivalence, unchanged under suitable changes in the initial conditions. If $0\le 1$, the conditions are satisfied by primitive matrices (a property depending only on the pattern of nonzero elements in the usual way), which gives rise to a generalization of the dominant eigenvalue to primitive matrices in such semirings.
(\cref{sec:linearrecurrence})

\item Our first concrete setting is the semiring of tensors. We show that approximately geometric sequences of tensors satisfy two previously studied properties: they are \emph{almost exponential} sequences \cite{christandl2025barriers}, and have the \emph{Kronecker scaling property} \cite{bjorklund2026kronecker}. We exhibit a broad family of approximately geometric sequences, which includes one of the sequences considered in \cite{bjorklund2026kronecker}, and tensors appearing in intermediate steps of the laser method for obtaining upper bounds on $\omega$.
(\cref{sec:tensors})

\item In the semiring of graphs $\graphs$, we study induced subgraphs of powers of some graph (with respect to the strong product) on type classes of vertices with converging type \cite{vrana2021probabilistic}. We show that such sequcences are approximately geometric. The corresponding elements of the asymptotic completion satisfy equations and inequalities previosly formulated for evaluations of elements of the asymptotic spectrum of graphs. We also construct approximately geometric sequences from Cauchy sequences of graphs with respect to the asymptotic spectrum distance \cite{de2024asymptotic}, realizing their limit points as elements in the completion $\completion{\graphs}$.
(\cref{sec:graphs})

\item For certain versions of the majorization semiring, all the elements of the relevant kind of spectrum are known, and are simple transforms of (multivariate) Rényi divergences, which therefore give concrete sufficient and necessary conditions for the asymptotic preorder. Our results imply that the characterizations extend to approximately geometric sequences, and can be written in terms of Rényi divergence rates, which always exist for such sequences. In a particular semiring preordered by relative submajorization, approximately geometric sequences can be interpreted as composite and correlated hypotheses, and the strong converse exponent for the corresponding binary hypothesis testing problem can be expressed as an optimization over Rényi divergence rates. We show that tuples of Markov chains with strictly positive transition probabilities form approximately geometric sequences, which provides a single-letter expression for the Rényi divergence rates and of the strong converse exponent.
(\cref{sec:majorization})

\end{itemize}

\subsection{Preliminaries}\label{sec:preliminaries}

\paragraph{Preordered semirings.}
We start by recalling some definitions and results on preordered semirings (see \cite{fritz2023abstract,fritz2021abstract2,wigderson2026asymptotic} for more details and examples).
\begin{definition}
A \emph{semiring} is a set $S$ together with binary operations $+,\cdot:S\times S\to S$ (addition and multiplication) and elements $0,1\in S$ such that both operations are associative and commutative, $0$ is a neutral element for $+$, $1$ is a neutral element for $\cdot$, and multiplication distributes over addition.

A \emph{preordered semiring} is a semiring $S$ equipped with a preorder $\le$ (i.e., transitive and reflexive relation), such that $x\le y$ implies $a+x\le a+y$ and $ax\le ay$ for all $a,x,y\in S$. If the preorder is antisymmetric as well, then we say that $S$ is a \emph{partially ordered semiring}.

If $S$ and $T$ are semirings, then a map $\phi:S\to T$ is a (semiring) \emph{homomorphism} if $\phi(0)=0$, $\phi(1)=1$, $\phi(a+b)=\phi(a)+\phi(b)$, and $\phi(ab)=\phi(a)\phi(b)$ for all $a,b\in S$. If $S$ and $T$ are preordered semirings, then a \emph{monotone homomomorphism} is a semiring homomorphism that in addition satisfies $\phi(a)\le\phi(b)$ whenever $a\le b$. The set of monotone semiring homomorphisms from $S$ to $T$ will be denoted by $\Hom(S,T)$.
\end{definition}
\begin{example}\label{ex:preorderedsemirings}\leavevmode
\begin{enumerate}
\item Natural numbers form a partially ordered semiring $\naturals$ with the usual addition, multiplication, and total order.
\item The set of nonnegative real numbers $\nonnegativereals$ with the usual addition, multiplication, and total order, is a partially ordered semiring.
\item The tropical semiring $\tropicalreals$ is, as a set, the same as $\nonnegativereals$, with the same multiplication and order, but with addition $a+b:=\max\{a,b\}$.
\item The set of positive dual numbers $\dualnumbers=\setbuild{r+sX}{r\in\positiveintegers,s\in\reals}\cup\{0\}$ with
\begin{itemize}
\item polynomial addition and multiplication modulo $X^2$, i.e., $(r_1+s_1X)+(r_2+s_2X)=(r_1+r_2)+(s_1+s_2)X$ and $(r_1+s_1X)(r_2+s_2X)=(r_1r_2)+(r_1s_2+r_2s_1)X$, and
\item the partial order $r_1+s_1X\le r_2+s_2X$ if and only if $r_1=r_2$ and $s_1\le s_2$,
\end{itemize}
is a partially ordered semiring.
\item Every semiring $S$ can be turned into a partially ordered semiring $S^=$, where $x\le y$ if and only if $x=y$.
\item If $S$ is a preordered semiring, then the semiring $S\opposite$ with the same underlying set and operations, and the reversed preorder (i.e. $a\le b$ in $S\opposite$ if and only if $a\ge b$ in $S$), is also a preordered semiring.
\end{enumerate}
\end{example}

Given a preorder $\le$ on a set $X$, the relation $x\approx y$ if and only if $x\le y$ and $y\le x$ is an equivalence relation, the largest equivalence relation contained in $\le$. The quotient $X/\approx$ is a partially ordered set. If $S$ is a preordered semiring, then $S/\approx$ is a partially ordered semiring.

The equivalence relation generated by $\le$, i.e.,  smallest equivalence relation containing $\le$ will be denoted by $\sim$. Concretely, $a\sim b$ if there is a sequence $a=x_0,x_1,\dots,x_n=b$ such that $x_{i-1}\le x_i$ or $x_{i-1}\ge x_i$ holds for all $i=1,2,\dots,n$.

\begin{definition}
A nonzero element $u$ in a preordered semiring $S$ is \emph{power universal}, if $1\le u$ and for every nonzero $x,y\in S$ with $x\le y$, there is a $k\in\naturals$ such that $y\le u^kx$. We say that the preordered semiring $S$ is of \emph{polynomial growth}, if $S$ has a power universal element.
\end{definition}
For nonzero elements $x,y\in S$, the weaker assumption $x\sim y$ also implies $y\le u^kx$ for suitable $k$.

We note that in \cite{fritz2023abstract}, polynomial growth is defined in a more general sense, requiring only the existence of a power universal pair $(u_-,u_+)$, i.e., a pair of nonzero elements in $S$ such that $u_-\le u_+$ and for all $x\le y$ the inequality $u_-^k y\le u_+^kx$ for some $k\in\naturals$. We choose to use the more restrictive definition for simplicity, and because it is satisfied in many applications. It is possible to modify our subsequent definitions and arguments to be valid under the more general condition if necessary (by replacing $y\le u^kx$ with $u_-^k y\le u_+^kx$), or alternatively the stronger condition can be enforced by passing to the semifield of fractions.

\begin{example}\leavevmode
\begin{enumerate}
\item In $\naturals$, $\nonnegativereals$, and $\tropicalreals$, the power universal elements $u$ are precisely those that satisfy $u>1$. In $\nonnegativereals\opposite$ and $\tropicalreals\opposite$, the set of power universal elements is $(0,1)$. $\naturals\opposite$ is not of polynomial growth in the sense of our definition (but it has power universal pairs such as $(2,1)$).
\item The preordered semiring $\dualnumbers$ is of polynomial growth. Its power universal elements are $1+cX$ with $c\in\positivereals$.
\item If $S$ is an arbitrary semiring, then $S^=$ is of polynomial growth. Its only power universal element is $u=1$.
\end{enumerate}
\end{example}

Several ways to relax the preorder on a preordered semiring have been studied in abstract or concrete settings, such as the catalytic (requiring $xa\le xb$ for some nonzero $x$), the many-copy (requiring $a^n\le b^n$ for some $n\in\positiveintegers$), and the asymptotic preorders. In this paper, we focus on the asymptotic one.

\begin{remark}\label{rem:improvedsublinear}
Some of the definitions and many proofs below involve sequences $k:\naturals\to\naturals$ satisfying $\lim_{n\to\infty}\frac{k_n}{n}$. By a sublinear sequence we will always mean such a sequence even when not explicitly stating that it is $\naturals$-valued.

In particular, we will require the existence of sublinear sequences with certain properties that remain true if the sequence is elementwise increased. We can use this freedom to simplify notations and arguments by replacing sublinear sequences with an improved sequence (in the sense of satisfying more properties) as follows.
\begin{enumerate}
\item Let $(k^{(1)}_n)_{n\in\naturals},(k^{(2)}_n)_{n\in\naturals},\dots,(k^{(m)}_n)_{n\in\naturals}$ be sublinear sequences. Then $k'_n:=\max_{1\le i\le m}k^{(m)}_n$ is a sublinear sequence $(k'_n)_{n\in\naturals}$ such that $k^{(i)}_n\le k'_n$ for all $i=1,\dots,m$ and $n\in\naturals$.
\item Let $(k_n)_{n\in\naturals}$ be a sublinear sequence. Then $k'_n:=\max_{0\le l\le n}k_l$ is a monotone increasing sublinear sequence $(k'_n)_{n\in\naturals}$ such that $k_n\le k'_n$ for all $n\in\naturals$.
\item Let $(k_n)_{n\in\naturals}$ be a sublinear sequence. Then $k'_n:=\lceil n\sup_{l\ge n}\frac{k_l}{l}\rceil$ is a subadditive sublinear sequence $(k'_n)_{n\in\naturals}$ such that $k_n\le k'_n$ for all $n\in\naturals$. If in addition $(k_n)_{n\in\naturals}$ is monotone increasing, then $(k'_n)_{n\in\naturals}$ is monotone increasing as well.
\end{enumerate}
\end{remark}

\begin{definition}
Let $S$ be a preordered semiring of polynomial growth with preorder denoted by $\le$. We define the \emph{asymptotic preorder} (or, extending \cite[Definition 3.1.]{wigderson2026asymptotic}, the \emph{closure} of $\le$) as $a\asymptoticle b$ if and only if there exists a sublinear sequence $(k_n)_{n\in\naturals}$ such that for all $n\in\naturals$ the inequality $a^n\le u^{k_n}b^n$ holds.

The preorder $\le$ is \emph{closed} if it coincides with its closure. In this case, we will also say that the preordered semiring is closed.
\end{definition}
Note that, while the definition of the closure refers to a particular power universal element, the new relation does not depend on the choice, only the existence is required. This is because if $u'$ is another power universal element, then $1\le u$ implies $u\le {u'}^p$ for some $p\in\naturals$, therefore $a^n\le u^{k_n}b^n$ implies $a^n\le {u'}^{pk_n}b^n$.

The closure of a semiring preorder is closed, as the terminology suggests \cite[Lemma 3.5.]{wigderson2026asymptotic}. Examples of closed semiring preorders include $\naturals$, $\nonnegativereals$, $\tropicalreals$, and $\dualnumbers$ from \cref{ex:preorderedsemirings}. In general, no information concerning the asymptotic preorder is lost by taking the closure of the preorder and quotienting by $\approx$, which results in a closed preordered semiring.

Preordered semirings can be studied using monotone semiring homomorphisms to certain test objects such as $\nonnegativereals$. If $\phi:S\to T$ is a monotone semiring homomorphism, and both $S$ and $T$ are of polynomial growth, then $a\asymptoticle b$ implies $\phi(a)\asymptoticle\phi(b)$ for all $a,b\in S$ (if $u_S\in S$ and $u_T\in T$ are power universal, then $\phi(u_S)\le u_T^p$ for some $p\in\naturals$). In particular, if the preorder of $T$ is closed, then $\phi$ is also monotone with respect to the closure of the preorder in $S$. Thus, monotone homomorphisms from $S$ into closed preordered semirings are well suited for studying the asymptotic preorder of $S$.

In applications, we would like to know what test objects we need for characterizing the asymptotic preorder. There exist several duality theorems or \emph{Vergleichsstellensätze} that provide conditions on $S$ under which the set of all monotone homomorphisms into a small set of test objects characterize asymptotic inequalities. The first one is due to Strassen, which we now state using terminology introduced in \cite{wigderson2026asymptotic}.
\begin{definition}
A preorder on a semiring $S$ is a \emph{Strassen preorder}, if the unique semiring homomorphism $\naturals\to S$ is an order-embedding, and $u=2$ is power universal.

Let $S$ be a semiring with a Strassen preorder and $s\in S$. The \emph{rank} of $s$ is $\rank(s)=\min\setbuild{m\in\naturals}{s\le m}$, and the \emph{subrank} of $s$ is $\subrank(s)=\max\setbuild{m\in\naturals}{m\le s}$. The \emph{asymptotic rank} and \emph{asymptotic subrank} are $\asymptoticrank(s)=\lim_{n\to\infty}\sqrt[n]{\rank(s^n)}$ and $\asymptoticsubrank(s)=\lim_{n\to\infty}\sqrt[n]{\subrank(s^n)}$.

An element $s\in S$ is called \emph{gapped} if $2\le a^k$ for some $k\in\naturals$ or $\phi(1)=1$ for some monotone homomorphism $\phi:S\to\nonnegativereals$.
\end{definition}

\begin{theorem}[Strassen, {\cite[Corollary 2.6.]{strassen1988asymptotic} and {\cite[Theorem 3.42]{wigderson2026asymptotic}}}]
Let $S$ be a semiring with a Strassen preorder, and let $a,b\in S$. Then
\begin{enumerate}
\item $a\asymptoticle b$ if and only if $\phi(a)\le\phi(b)$ for all $\phi\in\Hom(S,\nonnegativereals)$,
\item if $a\ge 1$, then $\asymptoticrank(s)=\max_{\phi\in\Hom(S,\nonnegativereals)}\phi(a)$,
\item if $a$ is gapped, then $\asymptoticsubrank(s)=\min_{\phi\in\Hom(S,\nonnegativereals)}\phi(a)$.
\end{enumerate}
\end{theorem}

In \cite{fritz2023abstract,fritz2021abstract2} Fritz developed generalizations of Strassen's theorem in different directions. We state parts of two particular versions that concern the asymptotic preorder, and refer to his papers for the complete statements and other versions.
\begin{theorem}[Fritz, {\cite[7.15. Theorem.]{fritz2023abstract}}]\label{thm:Vergleichsstellentropical}
Let $S$ be a preordered semiring with $0\le 1$ and a power universal element $u$. For nonzero $a,b\in S$, we have $a\asymptoticle b$ if and only if for all $\phi\in\Hom(S,\nonnegativereals)\cup\Hom(S,\tropicalreals)$ the inequality $\phi(a)\le\phi(b)$ holds.
\end{theorem}
The second one does not assume $0\le 1$. The resulting symmetry between $S$ and $S\opposite$ partially explains why homomorphisms into $\nonnegativereals\opposite$ and $\tropicalreals\opposite$ need to be considered as well. It turns out, in addition, that certain infinitesimal information is necessary, provided by monotone derivations.
\begin{theorem}[Fritz, {\cite[8.6. Theorem]{fritz2021abstract2}}]\label{thm:Vergleichsstellenderivations}
Let $S$ be a preordered semiring with a power universal element $u$. Suppose that for some $d\in\naturals$, there is a surjective homomorphism $\norm{\cdot}:S\to\positivereals^d\cup\{0\}$ with trivial kernel such that $a\le b\implies\norm{a}=\norm{b}\implies a\sim b$.

Let $a,b\in S\setminus\{0\}$ with $\norm{a}=\norm{b}$. Then $a\asymptoticle b$ if and only if for every monotone homomorphism $\phi:S\to\mathbb{K}$ with $\mathbb{K}\in\{\nonnegativereals,\tropicalreals,\nonnegativereals\opposite,\tropicalreals\opposite\}$ the inequality $\phi(a)\le\phi(b)$ holds, and for every monotone derivation $D$ at $\norm[i]{\cdot}$ (i.e., the $i$th component of $\norm{\cdot}$) the inequality $D(a)\le D(b)$ holds.
\end{theorem}
Regarding $\dualnumbers$ as a subset of $\nonnegativereals\times\reals$, any monotone homomorphism from $S$ to $\dualnumbers$ has two components $\phi:S\to\nonnegativereals^=$ and $D:S\to\reals$, where $\phi$ is a degenerate homomorphism, and $D$ is a monotone derivation at $\phi$. Conversely, given a homomorphism $\phi:S\to\nonnegativereals^=$ and a monotone $\phi$-derivation $D$, the map $a\mapsto \phi(a)+D(a)\phi$ gives a monotone homomorphism $S\to\dualnumbers$ (see, e.g., the proof of \cite[6.8. Proposition]{fritz2021abstract2}). For this reason, under the conditions of \cref{thm:Vergleichsstellenderivations}, $a\asymptoticle b$ holds if and only if $\phi(a)\le\phi(b)$ for all $\phi\in\Hom(S,\mathbb{K})$ with $\mathbb{K}\in\{\nonnegativereals,\tropicalreals,\nonnegativereals\opposite,\tropicalreals\opposite,\dualnumbers\}$.

\paragraph{Types and entropy.}
In applications to tensors and graphs, we will make use of the concept of types and type classes \cite[Chapter 2]{csiszar2011information}. We identify probability distributions on a finite set $\mathcal{X}$ with functions $P:\mathcal{X}\to\nonnegativereals$ such that $\sum_{x\in\mathcal{X}}P(x)=1$. The set of probability distributions on $\mathcal{X}$ will be denoted by $\distributions(\mathcal{X})$.

For $n\in\naturals$, the set of $n$-\emph{types} is the subset $\distributions[n](\mathcal{X})=\setbuild{P\in\distributions(\mathcal{X})}{\forall x\in\mathcal{X}:nx\in\naturals}$. The \emph{type class} is defined as
\begin{equation}
\typeclass{n}{P}=\setbuild{(x_1,x_2,\dots,x_n)\in\mathcal{X}^n}{\frac{1}{n}\lvert\setbuild{i\in[n]}{x_i=x}\rvert=P(x)\text{ for all $x\in\mathcal{X}$}}
\end{equation}
These sets form a partition of $\mathcal{X}^n$ into $\left\lvert\distributions[n](\mathcal{X})\right\rvert\le(n+1)^{\lvert\mathcal{X}\rvert}$ parts.

The cardinalities of the type classes can be given exactly in terms of multinomial coefficients
\begin{equation}
\left\lvert\typeclass{n}{P}\right\rvert
 = \frac{n!}{\prod_{x\in\mathcal{X}}(nP(x))!},
\end{equation}
and can be estimated in terms of the \emph{Shannon entropy} (with the convention $\log=\log_2$ and $0\log 0=0$)
\begin{equation}
\entropy(P)=-\sum_{x\in\mathcal{X}}P(x)\log P(x)
\end{equation}
as
\begin{equation}
\frac{1}{(n+1)^{\lvert\mathcal{X}\rvert}}2^{n\entropy(P)}
 \le \left\lvert\typeclass{n}{P}\right\rvert
 \le 2^{n\entropy(P)}.
\end{equation}

The entropy is a continuous function of the distribution. We will use the following continuity bound.
\begin{lemma}[{\cite{audenaert2007sharp}}]\label{lem:entropycontinuity}
If $P,Q\in\distributions(\mathcal{X})$ and $\epsilon=\frac{1}{2}\norm[1]{P-Q}$, then
\begin{equation}
\lvert\entropy(P)-\entropy(Q)\rvert\le \epsilon\log(\lvert\mathcal{X}\rvert-1)+h(\epsilon),
\end{equation}
where $h(p)=-p\log p-(1-p)\log(1-p)$.
\end{lemma}

\section{Approximately geometric sequences}\label{sec:approximatelygeometric}

We consider sequences in preordered semirings. The main examples to have in mind are geometric sequences, and later we will focus on sequences that in some respects behave like geometric ones, but in the following definitions we will not make any restriction. From now on, preordered semirings $S$ and $T$ are assumed to be of polynomial growth, with power universal element $u$. We allow $0\not\le 1$ unless explicitly stated otherwise.
\begin{definition}\label{def:sequenceoperations}
Let $a,b\in S^\naturals$, $a=(a_n)_{n\in\naturals}$ and $b=(b_n)_{n\in\naturals}$.

The \emph{sum} $a\sequenceplus b$ is the binomial convolution of $a$ and $b$, i.e., the sequence
\begin{equation}
(a\sequenceplus b)_n=\sum_{m=0}^n\binom{n}{m}a_{n-m}\cdot b_m.
\end{equation}

The \emph{product} $a\sequencetimes b$ or $ab$ is the elementwise product, i.e., the sequence
\begin{equation}
(a\sequencetimes b)_n=a_n\cdot b_n.
\end{equation}

We define $a\asymptoticle b$ if there exists a sublinear sequence $(k_n)_{n\in\naturals}\in\naturals^\naturals$ such that, for all $n\in\naturals$, the inequality $a_n\le u^{k_n}b_n$ holds. We write $a\asymptoticge b$ if $b\asymptoticle a$, and $a\asymptoticeq b$ if both $a\asymptoticle b$ and $a\asymptoticge b$ hold. In the latter case we say that $a$ and $b$ are \emph{asymptotically equivalent}.

If $\phi:S\to T$ is a monotone semiring homomorphism, then we set $\phi(a)=(\phi(a_n))_{n\in\naturals}\in T^\naturals$.
\end{definition}
Note that $\asymptoticle$ does not depend on the particular choice of the power universal element.

\begin{example}\label{ex:geometric}
For all $x\in S$, let $G(x)=(1,x,x^2,\dots)$ denote the corresponding geometric sequence.

If $a,b\in S$, then
\begin{equation}
(G(a)\sequenceplus G(b))_n=\sum_{m=0}^n\binom{n}{m}a^{n-m}b^m=G(a+b)_n,
\end{equation}
\begin{equation}
(G(a)\sequencetimes G(b))_n=a^nb^n=G(ab)_n,
\end{equation}
and $G(a)\asymptoticle G(b)$ iff $a^n\le u^{k_n}b^n$ for some sublinear sequence $(k_n)_{n\in\naturals}$. Thus the operations and the relation in \cref{def:sequenceoperations}, on geometric sequences correspond to the semiring operations in $S$ and the \emph{asymptotic} preorder of $S$.

If $\phi:S\to T$ is a semiring homomorphism, then $\phi(G(x))=(\phi(1),\phi(x),\phi(x^2),\dots)=(1,\phi(x),\phi(x)^2,\dots)=G(\phi(x))$.
\end{example}

The operations and the relation satisfy the properties that the notations suggest, except distributivity.
\begin{proposition}\label{prop:sequenceoperationsproperties}
Let $a,b,c\in S^\naturals$ be arbitrary sequences and $\phi:S\to T$ a monotone semiring homomorphism. Then the following equalities and implications hold:
\begin{enumerate}
\item\label{it:sequenceplusnull} $a\sequenceplus G(0)=a$
\item\label{it:sequencepluscommutative} $a\sequenceplus b=b\sequenceplus a$
\item\label{it:sequenceplusassociative} $(a\sequenceplus b)\sequenceplus c=a\sequenceplus (b\sequenceplus c)$
\item\label{it:sequencetimesunit} $a\sequencetimes G(1)=a$
\item\label{it:sequencetimescommutative} $a\sequencetimes b=b\sequencetimes a$
\item\label{it:sequencetimesassociative} $(a\sequencetimes b)\sequencetimes c=a\sequencetimes (b\sequencetimes c)$
\item\label{it:sequencelereflexive} $a\asymptoticle a$
\item\label{it:sequenceletransitive} if $a\asymptoticle b$ and $b\asymptoticle c$, then $a\asymptoticle c$
\item\label{it:sequenceleplus} if $a\asymptoticle b$, then $a\sequenceplus c\asymptoticle b\sequenceplus c$
\item\label{it:sequenceletimes} if $a\asymptoticle b$, then $a\sequencetimes c\asymptoticle b\sequencetimes c$
\item\label{it:sequenceplushomomorphism} $\phi(a\sequenceplus b)=\phi(a)\sequenceplus\phi(b)$
\item\label{it:sequencetimeshomomorphism} $\phi(a\sequencetimes b)=\phi(a)\sequencetimes\phi(b)$
\item\label{it:sequencelehomomorphism} if $a\asymptoticle b$, then $\phi(a)\asymptoticle\phi(b)$
\end{enumerate}
\end{proposition}
\begin{proof}
Since $G(0)=(1,0,0,\dots)$, we have $(a\sequenceplus G(0))_n=\sum_{m=0}^n\binom{n}{m}a_{n-m}\cdot G(0)_m=a_n$, which proves \ref{it:sequenceplusnull}. \ref{it:sequencepluscommutative} and \ref{it:sequenceplusassociative} are well-known properties of the binomial convolution.

\Cref{it:sequencetimesunit,it:sequencetimescommutative,it:sequencetimesassociative} follow from the corresponding properties of the semiring multiplication and that $\sequencetimes$ is defined elementwise.

To see \ref{it:sequencelereflexive}, we may take $k_n=0$ in the definition.
\ref{it:sequenceletransitive}: Let $k_n$ be a sublinear sequence such that $a_n\le u^{k_n}b_n$ and $b_n\le u^{k_n}c_n$ hold for all $n\in\naturals$. Then $a_n\le u^{k_n}b_n\le  u^{2k_n}c_n$, and $2k_n$ is sublinear, therefore $a\asymptoticle c$.

\ref{it:sequenceleplus}: Let $k_n$ be a monotone increasing sublinear sequence such that $a_n\le u^{k_n}b_n$ holds for all $n$. Then
\begin{equation}
\begin{split}
(a\sequenceplus c)_n
 & = \sum_{m=0}^n\binom{n}{m}a_{n-m}c_m
 \le \sum_{m=0}^n\binom{n}{m}u^{k_{n-m}}b_{n-m}c_m  \\
 & \le u^{k_n}\sum_{m=0}^n\binom{n}{m}b_{n-m}c_m
 = u^{k_n}(b\sequenceplus c)_n,
\end{split}
\end{equation}
which proves $a\sequenceplus c\asymptoticle b\sequenceplus c$.

\ref{it:sequenceletimes}: Let $k_n$ be a sublinear sequence such that $a_n\le u^{k_n}b_n$ holds for all $n$. Then $(a\sequencetimes c)_n=a_nc_n\le u^{k_n}b_nc_n=u^{k_n}(b\sequencetimes c)_n$, therefore $a\sequencetimes c\preorderle b\sequencetimes c$.

\ref{it:sequenceplushomomorphism} and \ref{it:sequencetimeshomomorphism} are obtained by applying $\phi$ to both sides of the definitions of $\sequenceplus$ and $\sequencetimes$.

\ref{it:sequencelehomomorphism}: Let $u_S$ be a power universal element of $S$ and $u_T$ a power universal element of $T$. Since $1\le u_S$ and $\phi$ is monotone, we have $1\le\phi(u_S)$, which implies that $\phi(u_S)\le u_T^p$ holds for some $p\in\naturals$. Let $k_n$ be a sublinear sequence such that $a_n\le u_S^{k_n}b_n$ holds for all $n$. Then $\phi(a)_n=\phi(a_n)\le\phi(u_S)^{k_n}\phi(b_n)\le u_T^{pk_n}\phi(b)_n$. Since $n\mapsto pk_n$ is sublinear, this implies $\phi(a)\asymptoticle\phi(b)$.
\end{proof}

\begin{remark}\leavevmode
\begin{enumerate}
\item If $a,b,c\in S^\naturals$ are arbitrary sequences, then $(a\sequenceplus b)\sequencetimes c=(a_0b_0c_0,a_1b_0c_1+a_0b_1c_1,\dots)$ and $(a\sequencetimes c)\sequenceplus(b\sequencetimes c)=(a_0b_0c_0^2,a_1b_0c_0c_1+a_0b_1c_0c_1,\dots)$ are in general different, therefore the distributive law is not satisfied by the sum and product of sequences.
\item If $a,b\in S^\naturals$ are arbitrary and $c=G(s)$ is a geometric sequence ($s\in S$), then
\begin{equation}
\begin{split}
((a\sequenceplus b)\sequencetimes c)_n
 & = \sum_{m=0}^n\binom{n}{m}a_{n-m}b_mc_n  \\
 & = \sum_{m=0}^n\binom{n}{m}a_{n-m}b_m s^n  \\
 & = \sum_{m=0}^n\binom{n}{m}(a_{n-m}s^{n-m})(b_m s^m)  \\
 & = \sum_{m=0}^n\binom{n}{m}(a_{n-m}c_{n-m})(b_m c_m)  \\
 & = ((a\sequencetimes c)\sequenceplus(b\sequencetimes c))_n.
\end{split}
\end{equation}
\end{enumerate}
\end{remark}

Our next goal is to restrict to a class of sequences on which distributivity is approximately satisfied. We start by proving the equivalence of certain properties, any of which can serve as a definition for these sequences. Of the properties in \cref{prop:approxgeomcharacterizations}, \ref{it:approxgeompowersr} and \ref{it:approxgeompowers1} are probably intuitively the clearest: they essentially say that $a=(a_n)_{n\in\naturals}$ is sandwiched between geometric sequences (with rescaled index) generated by elements of $a$ itself, and the upper and lower bounds can be made arbitrarily close in an asymptotic sense, a more or less direct translation of the existence of $\lim_{n\to\infty}\frac{1}{n}\ln a_n$ when $a_n\in\positivereals$ (see \cref{ex:approximatelygeometric} below). These are also the ones that we find to be simplest to verify in practice for concrete sequences. In contrast, property \ref{it:approxgeomproducts} seems to be the most convenient for proving properties abstractly, in particular constructing new sequences from old ones.
\begin{proposition}\label{prop:approxgeomcharacterizations}
Let $a\in(\nonzeros{S})^\naturals$. The following statements are equivalent:
\begin{enumerate}
\item\label{it:approxgeomproducts} There exists a sublinear sequence $(k_n)_{n\in\naturals}$ such that for all $t\in\positiveintegers$ and $n_1,n_2,\dots,n_t\in\naturals$ the inequalities
\begin{align}
a_{n_1}a_{n_2}\cdots a_{n_t} & \le u^{k_{n_1}+\dots+k_{n_t}}a_{n_1+n_2+\dots+n_t}
\intertext{and}
a_{n_1+n_2+\dots+n_t} & \le u^{k_{n_1}+\dots+k_{n_t}}a_{n_1}a_{n_2}\cdots a_{n_t}
\end{align}
hold.
\item\label{it:approxgeompowersr} For all $n\in\naturals$ the relation $a_n\sim a_1^n$ holds, and for all $\epsilon>0$ there exist $m\in\positiveintegers$, $c\in\naturals$ such that for all $q\in\naturals$, $r\in\{0,1,\dots,m-1\}$, with $n=qm+r$ the inequalities $a_n\le u^{c+\lfloor\epsilon n\rfloor}a_m^qa_r$ and $a_m^qa_r\le u^{c+\lfloor\epsilon n\rfloor}a_n$ hold.
\item\label{it:approxgeompowers1} For all $n\in\naturals$ the relation $a_n\sim a_1^n$ holds, and for all $\epsilon>0$ there exist $m\in\positiveintegers$, $c\in\naturals$ such that for all $q\in\naturals$, $r\in\{0,1,\dots,m-1\}$, with $n=qm+r$ the inequalities $a_n\le u^{c+\lfloor\epsilon n\rfloor}a_m^qa_1^r$ and $a_m^qa_1^r\le u^{c+\lfloor\epsilon n\rfloor}a_n$ hold.
\item\label{it:approxgeomproductsepsilon} For all $\epsilon>0$ there exists $c\in\naturals$ such that for all $t\in\positiveintegers$ and $n_1,n_2,\dots,n_t\in\naturals$ the inequalities $a_{n_1}a_{n_2}\cdots a_{n_t}\le u^{tc+\lfloor\epsilon(n_1+\dots+n_t)\rfloor}a_{n_1+n_2+\dots+n_t}$ and $a_{n_1+n_2+\dots+n_t}\le u^{tc+\lfloor\epsilon(n_1+\dots+n_t)\rfloor}a_{n_1}a_{n_2}\cdots a_{n_t}$ hold.
\end{enumerate}
\end{proposition}
\begin{proof}
\ref{it:approxgeomproducts}$\implies$\ref{it:approxgeompowersr}:
Let $(k_n)_{n\in\naturals}$ be a monotone increasing sublinear sequence satisfying the inequalities. Then $a_1^n\le u^{nk_1}a_n$ and $a_n\le u^{nk_1}a_1^n$, therefore $a_n\sim a_1^n$.

For all $q,m,r\in\naturals$, $0\le r<m$, and $n=qm+r$ we have
\begin{align}
a_n
 & = a_{qm+r}
   \le u^{qk_m+k_r}a_m^qa_r
\intertext{and}
a_m^qa_r
 & \le u^{qk_m+k_r}a_{qm+r}
   = u^{qk_m+k_r}a_n.
\end{align}
Let $\epsilon>0$. Since $qk_m+k_r\le n\frac{k_m}{m}+k_{m-1}$, the inequalities also hold if the exponent of $u$ is replaced with $\lfloor n\frac{k_m}{m}\rfloor+k_{m-1}$. By sublinearity, we can choose an $m$ such that $\frac{k_m}{m}\le\epsilon$ and $c=k_{m-1}$. With these choices, we have $a_n\le u^{c+\lfloor\epsilon n\rfloor}a_m^qa_r$ and $a_m^qa_r\le u^{c+\lfloor\epsilon n\rfloor}a_n$ for all $q,m,r$.

\ref{it:approxgeompowersr}$\implies$\ref{it:approxgeompowers1}:
The difference between the two pairs of inequalities is that $a_r$ is replaced with $a_1^r$. For $\epsilon>0$ we keep the same $m$ and adjust $c$ as follows. The relation $a_r\sim a_1^r$ implies that we can choose $p\in\naturals$ such that $a_r\le u^pa_1^r$ and $a_1^r\le u^pa_r$. In the range $r\in\{0,1,\dots,m-1\}$, we can find a common $p$ that satisfies the finitely many inequalities. Then $a_n\le u^{c+\lfloor\epsilon n\rfloor}a_m^qa_r\le u^{c+p+\lfloor\epsilon n\rfloor}a_m^qa_1^r$ and $a_m^qa_1^r\le u^pa_m^qa_r\le u^{c+p+\lfloor\epsilon n\rfloor}a_n$, and the extra $p$ term can be absorbed into $c$.

\ref{it:approxgeompowers1}$\implies$\ref{it:approxgeomproductsepsilon}:
Suppose that $\epsilon>0$, $m$, $c$ satisfy the conditions in \ref{it:approxgeompowers1} for all $n=qm+r$. Let $p\in\naturals$ such that $a_m\le u^p a_1^m$ and $a_1^m\le u^p a_m$. Let $t\in\naturals$, $n_1,\dots,n_t\in\naturals$, $n_i=q_im+r_i$ ($q_i\in\naturals$, $r_i\in\{0,1,\dots,m-1\}$) for all $i=1,\dots,t$. Let $q=\lfloor\frac{r_1+\dots+r_t}{m}\rfloor\le t$ and $r=r_1+\dots+r_t-qm$. Then
\begin{equation}
\begin{split}
a_{n_1}a_{n_2}\cdots a_{n_t}
 & \le u^{tc+\lfloor\epsilon n_1\rfloor+\dots+\lfloor\epsilon n_t\rfloor}a_m^{q_1+\dots+q_t}a_1^{r_1+\dots+r_t}  \\
 & \le u^{tc+\lfloor\epsilon n_1\rfloor+\dots+\lfloor\epsilon n_t\rfloor+pq}a_m^{q_1+\dots+q_t+q}a_1^r  \\
 & \le u^{(t+1)c+\lfloor\epsilon n_1\rfloor+\dots+\lfloor\epsilon n_t\rfloor+pq+\lfloor \epsilon(n_1+\dots+n_t)\rfloor}a_{n_1+\dots+n_t}  \\
 & \le u^{(t+1)(c+p)+2\lfloor \epsilon(n_1+\dots+n_t)\rfloor}a_{n_1+\dots+n_t}  \\
 & \le u^{2t(c+p)+\lfloor 2\epsilon(n_1+\dots+n_t)\rfloor}a_{n_1+\dots+n_t},
\end{split}
\end{equation}
and in a similar way we obtain
\begin{equation}
a_{n_1+\dots+n_t}
 \le u^{2t(c+p)+\lfloor 2\epsilon(n_1+\dots+n_t)\rfloor}a_{n_1}a_{n_2}\cdots a_{n_t}.
\end{equation}
Therefore the condition in \ref{it:approxgeomproductsepsilon} is satisfied with $2\epsilon$ and $2(c+p)$.

\ref{it:approxgeomproductsepsilon}$\implies$\ref{it:approxgeomproducts}:
Choose $\epsilon_0>\epsilon_1>\epsilon_2>\cdots$ converging to $0$, $c_0,c_1,\dots$ nondecreasing such that for all $j\in\naturals$, $t\in\naturals$ and $n_1,n_2,\dots,n_t\in\naturals$ the inequalities
\begin{align}
a_{n_1}a_{n_2}\cdots a_{n_t} & \le u^{c_jt+\lfloor\epsilon_j(n_1+\dots+n_t)\rfloor}a_{n_1+\dots+n_t}  \\
a_{n_1+\dots+n_t} & \le u^{c_jt+\lfloor\epsilon_j(n_1+\dots+n_t)\rfloor}a_{n_1}a_{n_2}\cdots a_{n_t}
\end{align}
hold. Let $N_0=0$ and
\begin{equation}
N_j=\left\lceil\frac{c_j(j+1)}{\epsilon_j}\right\rceil
\end{equation}
when $j\ge 1$. Then $N_j\to\infty$ increasingly, and if $N_j\le n$, then $c_j\le c_j(j+1)\le\epsilon_j n$. Let
\begin{equation}
k_n=\begin{cases}
c_0+3\lceil\epsilon_0 n\rceil & \text{if $N_0\le n<N_1$}  \\
4\lceil\epsilon_j n\rceil & \text{if $N_j\le n< N_{j+1}$.}
\end{cases}
\end{equation}
Then $\lim_{n\to\infty}\frac{k_n}{n}=0$, since $N_j\le n$ implies $k_n\le 4\lceil\epsilon_j n\rceil$ when $j\ge 1$.

Let $t\in\naturals$, $n_1,\dots,n_t\in\naturals$. Let $t_n$ be the number of times $n$ appears among $n_1,\dots,n_t$, and $n_{\textnormal{max}}$ the largest index appearing. Let $j_{\textnormal{max}}$ be the number satisfying $N_{j_{\textnormal{max}}}\le n_{\textnormal{max}}<N_{j_\textnormal{max}+1}$. Then
\begin{equation}
\begin{split}
a_{n_1}\cdots a_{n_t}
 & = \prod_{j=0}^{j_{\textnormal{max}}}\left(\prod_{n=N_j}^{N_{j+1}-1}a_n^{t_n}\right)  \\
 & \le \prod_{j=0}^{j_{\textnormal{max}}}\left(u^{c_j\sum_{n=N_j}^{N_{j+1}-1}t_n+\left\lfloor\epsilon_j\sum_{n=N_j}^{N_{j+1}-1}t_nn\right\rfloor}a_{\sum_{n=N_j}^{N_{j+1}-1}t_nn}\right)  \\
 & = u^{\sum_{j=0}^{j_{\textnormal{max}}}\left(c_j\sum_{n=N_j}^{N_{j+1}-1}t_n+\left\lfloor\epsilon_j\sum_{n=N_j}^{N_{j+1}-1}t_nn\right\rfloor\right)}\prod_{j=0}^{j_{\textnormal{max}}}a_{\sum_{n=N_j}^{N_{j+1}-1}t_nn}  \\
 & \le u^{\sum_{j=0}^{j_{\textnormal{max}}}\left(c_j\sum_{n=N_j}^{N_{j+1}-1}t_n+\left\lfloor\epsilon_j\sum_{n=N_j}^{N_{j+1}-1}t_nn\right\rfloor\right)+c_{j_{\textnormal{max}}}(j_{\textnormal{max}}+1)+\lfloor\epsilon_{j_{\textnormal{max}}}\sum_{n=0}^{n_{\textnormal{max}}}t_nn\rfloor}a_{\sum_{n=0}^{n_{\textnormal{max}}}t_nn}  \\
 & = u^ra_{n_1+\dots+n_t},
\end{split}
\end{equation}
where we have introduced the notation
\begin{equation}
r
 = \sum_{j=0}^{j_{\textnormal{max}}}\left(c_j\sum_{n=N_j}^{N_{j+1}-1}t_n+\left\lfloor\epsilon_j\sum_{n=N_j}^{N_{j+1}-1}t_nn\right\rfloor\right)+c_{j_{\textnormal{max}}}(j_{\textnormal{max}}+1)+\lfloor\epsilon_{j_{\textnormal{max}}}\sum_{n=0}^{n_{\textnormal{max}}}t_nn\rfloor.
\end{equation}
With similar steps in the opposite order we obtain $a_{n_1+\dots+n_t}\le u^r a_{n_1}\cdots a_{n_t}$. The exponent can be bounded as
\begin{equation}
\begin{split}
r
 & \le c_0\sum_{n=0}^{N_1-1}t_n+\epsilon_0\sum_{n=0}^{N_1-1}t_nn+\sum_{j=1}^{j_{\textnormal{max}}}2\epsilon_j\sum_{n=N_j}^{N_{j+1}-1}t_nn+\epsilon_{j_{\textnormal{max}}}n_{\textnormal{max}}+\epsilon_{j_{\textnormal{max}}}\sum_{n=0}^{n_{\textnormal{max}}}t_nn  \\
 & \le c_0\sum_{n=0}^{N_1-1}t_n+\epsilon_0\sum_{n=0}^{N_1-1}t_nn+\sum_{j=1}^{j_{\textnormal{max}}}2\epsilon_j\sum_{n=N_j}^{N_{j+1}-1}t_nn+2\epsilon_{j_{\textnormal{max}}}\sum_{n=0}^{n_{\textnormal{max}}}t_nn  \\
 & \le c_0\sum_{n=0}^{N_1-1}t_n+3\epsilon_0\sum_{n=0}^{N_1-1}t_nn+\sum_{j=1}^{j_{\textnormal{max}}}4\epsilon_j\sum_{n=N_j}^{N_{j+1}-1}t_nn  \\
 & = \sum_{n=0}^{N_1-1}(c_0+3\epsilon_0n)t_n+\sum_{j=1}^{j_{\textnormal{max}}}\sum_{n=N_j}^{N_{j+1}-1}(4\epsilon_jn)t_n  \\
 & \le \sum_{n=0}^{N_1-1}k_nt_n+\sum_{j=1}^{j_{\textnormal{max}}}\sum_{n=N_j}^{N_{j+1}-1}k_nt_n 
   = \sum_{n=0}^{n_{\textnormal{max}}}k_nt_n 
   = k_{n_1}+k_{n_2}+\dots+k_{n_t}.
\end{split}
\end{equation}
\end{proof}
We note that if $0\le 1$ in $S$, then the relation $a_n\sim a_1^n$ in \ref{it:approxgeompowersr} and \ref{it:approxgeompowers1} is automatically true. In this case, or under the weaker condition that for all $n\in\naturals$, $a_n\sim 1$ holds, then the $a_r$ or $a_1^r$ factor can be absorbed into $u^c$, yielding the inequalities $a_n\le u^{c+\lfloor\epsilon n\rfloor}a_m^{\left\lfloor\frac{n}{m}\right\rfloor}$ and $a_m^{\left\lfloor\frac{n}{m}\right\rfloor}\le u^{c+\lfloor\epsilon n\rfloor}a_n$. In addition, since $a_n\le u^{nk_1}a_1^n$, $a_1^n\le u^{nk_1}a_n$, and $a_1\sim 1$, there exists $p\in\naturals$ such that $a_n\le u^{pn}$ and $1\le u^{pn}a_n$ for all $n\in\positiveintegers$.

\begin{definition}
A sequence $a\in S^\naturals$ is \emph{approximately geometric} if either $a\in(\nonzeros{S})^\naturals$ and $a$ satisfies the equivalent conditions of \cref{prop:approxgeomcharacterizations}, or $a\asymptoticeq G(0)$.
\end{definition}
We note that the definition only uses the multiplicative structure and the preorder. It follows, e.g., that in $\nonnegativereals$ and $\tropicalreals$ the approximately geometric sequences are the same.

\begin{example}\label{ex:approximatelygeometric}\leavevmode
\begin{enumerate}
\item Let $S=\naturals$ and $u=2$. The sequence $a_n=2^{\lfloor n/2\rfloor}$ is approximately geometric, since
\begin{align}
a_{n_1+n_2+\dots+n_t}
 & = 2^{\lfloor (n_1+\dots+n_t)/2\rfloor}
   \le 2^{\lfloor n_1/2\rfloor+1+\dots+\lfloor n_t/2\rfloor+1}
   = 2^ta_{n_1}a_{n_2}\cdots a_{n_t}
\intertext{and}
a_{n_1}a_{n_2}\cdots a_{n_t}
 & = 2^{\lfloor n_1/2\rfloor}\cdots 2^{\lfloor n_t/2\rfloor}
   \le 2^{\lfloor (n_1+\dots+n_t)/2\rfloor}
   = a_{n_1+n_2+\dots+n_t},
\end{align}
therefore \ref{it:approxgeomproducts} in \cref{prop:approxgeomcharacterizations} is satisfied with $k_n=1$.
\item Let $S=\nonnegativereals$ and $u>1$ and $a\in(\nonzeros{S})^\naturals$. Condition \ref{it:approxgeompowers1} in \cref{prop:approxgeomcharacterizations} can be rewritten as follows. For all $\epsilon>0$ there exist $m\in\positiveintegers$ and $c\in\naturals$ such that for all $n\in\naturals$ the inequalities
\begin{align}
a_n & \le u^{c+\left\lfloor\epsilon n\right\rfloor}a_m^{\left\lfloor\frac{n}{m}\right\rfloor}a_1^{n-m\left\lfloor\frac{n}{m}\right\rfloor}  \\
a_m^{\left\lfloor\frac{n}{m}\right\rfloor}a_1^{n-m\left\lfloor\frac{n}{m}\right\rfloor} & \le u^{c+\left\lfloor\epsilon n\right\rfloor}a_n
\end{align}
hold. By taking logarithm, dividing by $n$ and letting $n\to\infty$, the inequalities become
\begin{equation}
\frac{\ln a_m}{m}-\epsilon\le \liminf_{n\to\infty}\frac{\ln a_n}{n}\le\limsup_{n\to\infty}\frac{\ln a_n}{n}\le\frac{\ln a_m}{m}+\epsilon,
\end{equation}
which are equivalent to the preceding ones for suitably large $c$. It follows that $a$ is approximately geometric if and only if $\frac{\ln a_n}{n}$ converges (equivalently: $\sqrt[n]{a_n}$ converges to a nonzero number).
\item Let $S=\dualnumbers$ and $u=1+X$, and let $(a_n)_{n\in\naturals}$ be an approximately geometric sequence with $a_n=r_n+s_nX$, $r_n>0$ for all $n$. Since
\begin{equation}
r_n+s_nX
 \sim (r_1+s_1X)^n
 = r_1^n+r_1^{n-1}s_1nX,
\end{equation}
the sequence satisfies $r_n=r_1^n$ for all $n\in\naturals$.

For $\epsilon>0$ let $m\in\positiveintegers$ and $c\in\naturals$ such that the inequalities above hold for all $n\in\naturals$. The coefficients of $X$ then satisfy
\begin{gather}
s_n
 \le r_1^n(c+\left\lfloor\epsilon n\right\rfloor)+r_1^{n-m}\left\lfloor\frac{n}{m}\right\rfloor s_m+r_1^{n-1}\left(n-m\left\lfloor\frac{n}{m}\right\rfloor\right)s_1  \\
r_1^{n-m}\left\lfloor\frac{n}{m}\right\rfloor s_m+r_1^{n-1}\left(n-m\left\lfloor\frac{n}{m}\right\rfloor\right)s_1
 \le r_1^n(c+\left\lfloor\epsilon n\right\rfloor)+s_n,
\end{gather}
which can be written as
\begin{equation}
\left\lvert \frac{1}{n}\frac{s_n}{r_n}-\frac{1}{n}\left\lfloor\frac{n}{m}\right\rfloor \frac{s_m}{r_m}+\left(1-\frac{m}{n}\left\lfloor\frac{n}{m}\right\rfloor\right)\frac{s_1}{r_1}\right\rvert\le \frac{c+\left\lfloor\epsilon n\right\rfloor}{n}
\end{equation}
As $n\to\infty$ we arrive at the inequality
\begin{equation}
\frac{1}{m}\frac{s_m}{r_m}-\epsilon\le\liminf_{n\to\infty}\frac{1}{n}\frac{s_n}{r_n}\le\limsup_{n\to\infty}\frac{1}{n}\frac{s_n}{r_n}\le\frac{1}{m}\frac{s_m}{r_m}+\epsilon,
\end{equation}
therefore the condition is that $r_n=r_1^n$ and $\frac{1}{n}\frac{s_n}{r_n}$ converges.
\item If $S$ is an arbitrary semiring, then the approximately geometric sequences of $S^=$ are the geometric ones.
\end{enumerate}
\end{example}

\begin{proposition}\label{prop:approximatelygeometricproperties}
Let $a,b,c\in(\nonzeros{S})^\naturals$ be sequences and $\phi:S\to T$ a monotone semiring homomorphism.
\begin{enumerate}
\item\label{it:approximatelygeometricequivalent} If $a$ is approximately geometric and $a\asymptoticeq b$, then $b$ is approximately geometric.
\item\label{it:approximatelygeometricplus} If $a$ and $b$ are approximately geometric, then $a\sequenceplus b$ is approximately geometric.
\item\label{it:approximatelygeometrictimes} If $a$ and $b$ are approximately geometric, then $a\sequencetimes b$ is approximately geometric. If $a$ is approximately geometric, then $G(0)\sequencetimes a\asymptoticeq G(0)$.
\item\label{it:approximatelygeometricdistributive} If $c$ is approximately geometric, then $(a\sequenceplus b)\sequencetimes c\asymptoticeq (a\sequencetimes c)\sequenceplus(b\sequencetimes c)$ holds.
\item\label{it:approximatelygeometrichomomorphism} If $a$ is approximately geometric, then $\phi(a)$ is approximately geometric.
\end{enumerate}
\end{proposition}
\begin{proof}
\ref{it:approximatelygeometricequivalent}:
If $a\asymptoticeq G(0)$ and $a\asymptoticeq b$, then $b\asymptoticeq G(0)$, therefore $b$ is approximately geometric. Otherwise let $(k_n)_{n\in\naturals}$ be a subadditive sublinear sequence such that for all $t\in\naturals$ and $n_1,\dots,n_t\in\naturals$ the inequalities $a_{n_1}\dots a_{n_t}\le u^{k_{n_1}+\dots+k_{n_t}}a_{n_1+\dots+n_t}$, $a_{n_1+\dots+n_t}\le u^{k_{n_1}+\dots+k_{n_t}}a_{n_1}\dots a_{n_t}$ hold, and for all $n\in\naturals$ we have $a_n\le u^{k_n}b_n$ and $b_n\le u^{k_n}a_n$. Then
\begin{equation}
\begin{split}
b_{n_1}\cdots b_{n_t}
 & \le u^{k_{n_1}+\dots+k_{n_t}}a_{n_1}\cdots a_{n_t}  \\
 & \le u^{2(k_{n_1}+\dots+k_{n_t})}a_{n_1+\dots+n_k}  \\
 & \le u^{2(k_{n_1}+\dots+k_{n_t})+k_{n_1+\dots+n_k}}b_{n_1+\dots+n_k}  \\
 & \le u^{3(k_{n_1}+\dots+k_{n_t})}b_{n_1+\dots+n_k}
\end{split}
\end{equation}
and
\begin{equation}
\begin{split}
b_{n_1+\dots+n_k}
 & \le u^{k_{n_1+\dots+n_t}}a_{n_1+\dots+n_k}  \\
 & \le u^{2(k_{n_1}+\dots+k_{n_t})}a_{n_1}\cdots a_{n_t}  \\
 & \le u^{3(k_{n_1}+\dots+k_{n_t})}b_{n_1}\cdots b_{n_t}.
\end{split}
\end{equation}
Since $n\mapsto 3k_n$ is sublinear, this proves that $b$ is approximately geometric.

\ref{it:approximatelygeometricplus}:
If one of the terms, say $a$ is asymptotically equivalent to $G(0)$, then $a\sequenceplus b\asymptoticeq G(0)\sequenceplus b=b$ is approximately geometric. Otherwise let $k_n$ be a sublinear sequence such that for all $t\in\positiveintegers$ and $n_1,\dots,n_t\in\naturals$ the inequalities $a_{n_1}\dots a_{n_t}\le u^{k_{n_1}+\dots+k_{n_t}}a_{n_1+\dots+n_t}$, $a_{n_1+\dots+n_t}\le u^{k_{n_1}+\dots+k_{n_t}}a_{n_1}\dots a_{n_t}$, $b_{n_1}\dots b_{n_t}\le u^{k_{n_1}+\dots+k_{n_t}}b_{n_1+\dots+n_t}$, and $b_{n_1+\dots+n_t}\le u^{k_{n_1}+\dots+k_{n_t}}b_{n_1}\dots b_{n_t}$ hold. Let $\epsilon>0$ and $c\in\naturals$ such that $k_n\le c+\epsilon n$ for all $n\in\naturals$. Let $s=a\sequenceplus b$.

We verify condition \ref{it:approxgeompowersr} from \cref{prop:approxgeomcharacterizations}. We have
\begin{equation}
s_n
 = \sum_{i=0}^n\binom{n}{i}a_{n-i}b_{n-i}
 \sim \sum_{i=0}^n\binom{n}{i}a_1^{n-i}b_1^{n-i}
 = (a_1+b_1)^n
 = s_1^n.
\end{equation}

For $n=qm+r$ with $q,m,r\in\naturals$, $0\le r<m$, we have
\begin{equation}
\begin{split}
s_m^qs_r
 & = \left(\sum_{i=0}^m\binom{m}{i}a_{m-i}b_i\right)^q\sum_{j=0}^r\binom{r}{j}a_{r-j}b_j  \\
 & = \sum_{i_1=0}^m\sum_{i_2=0}^m\dots\sum_{i_q=0}^m\sum_{j=0}^r\binom{m}{i_1}\cdots\binom{m}{i_q}\binom{r}{j}a_{m-i_1}a_{m-i_2}\cdots a_{m-i_q}a_{r-j}b_{i_1}b_{i_2}\cdots b_{i_q}b_j  \\
 & \le \sum_{i_1=0}^m\sum_{i_2=0}^m\dots\sum_{i_q=0}^m\sum_{j=0}^r\binom{m}{i_1}\cdots\binom{m}{i_q}\binom{r}{j}u^{k_{m-i_1}+\dots+k_{m-i_q}+k_{r-j}}a_{m-i_1+\dots+m-i_2+r-j}  \\
 &\qquad\cdot u^{k_{i_1}+\dots+k_{i_q}+k_{j}}b_{i_1+\dots+i_q+j}
\end{split}
\end{equation}
The total exponent of $u$ satisfies
\begin{equation}
\begin{split}
k_{m-i_1}+\dots+k_{m-i_q}+k_{r-j}+k_{i_1}+\dots+k_{i_q}+k_{j}
 & \le \sum_{l=1}^q(c+\epsilon(m-i_l)+c+\epsilon i_l)+c+\epsilon(r-j)+c+\epsilon j  \\
 & = 2c(q+1)+\epsilon(qm+r)
   = 2c(q+1)+\epsilon n,
\end{split}
\end{equation}
therefore
\begin{equation}
\begin{split}
s_m^qs_r
 & \le u^{2c(q+1)+\lfloor\epsilon n\rfloor}\sum_{i_1=0}^m\sum_{i_2=0}^m\dots\sum_{i_q=0}^m\sum_{j=0}^r\binom{m}{i_1}\cdots\binom{m}{i_q}\binom{r}{j}a_{m-i_1+\dots+m-i_2+r-j}b_{i_1+\dots+i_q+j}  \\
 & = u^{2c(q+1)+\lfloor\epsilon n\rfloor}\sum_{i=0}^n\binom{n}{i}a_{n-i}b_i  \\
 & = u^{2c(q+1)+\lfloor\epsilon n\rfloor}s_n.
\end{split}
\end{equation}
In a similar way we have
\begin{equation}
\begin{split}
s_n
 & = \sum_{i_1=0}^m\sum_{i_2=0}^m\dots\sum_{i_q=0}^m\sum_{j=0}^r\binom{m}{i_1}\cdots\binom{m}{i_q}\binom{r}{j}a_{m-i_1+\dots+m-i_2+r-j}b_{i_1+\dots+i_q+j}  \\
 & \le u^{2c(q+1)+\lfloor\epsilon n\rfloor}\sum_{i_1=0}^m\sum_{i_2=0}^m\dots\sum_{i_q=0}^m\sum_{j=0}^r\binom{m}{i_1}\cdots\binom{m}{i_q}\binom{r}{j}a_{m-i_1+\dots+m-i_2+r-j}b_{i_1+\dots+i_q+j}  \\
 & = u^{2c(q+1)+\lfloor\epsilon n\rfloor}s_m^qs_r.
\end{split}
\end{equation}
Since $2c(q+1)+\lfloor\epsilon n\rfloor\le 2c(\frac{n}{m}+1)+\lfloor\epsilon n\rfloor\le 2c+(\frac{2c}{m}+\epsilon)n$, by choosing $\epsilon$ small and then $m$ large enough, we can make the coefficient of $n$ arbitrarily small.

\ref{it:approximatelygeometrictimes}: If $a\asymptoticeq G(0)$ or $b\asymptoticeq G(0)$, then $a\sequencetimes b=(a_0b_0,0,0,\dots)\asymptoticeq G(0)$, since $0\neq a_0b_0\sim 1$. Otherwise we use
\begin{equation}
\begin{split}
(a\sequencetimes b)_{n_1}(a\sequencetimes b)_{n_2}\dots(a\sequencetimes b)_{n_t}
 & = a_{n_1}a_{n_2}\dots a_{n_t}b_{n_1}b_{n_2}\dots b_{n_t}  \\
 & \le u^{k_{n_1}+\dots+k_{n_t}}a_{n_1+\dots+n_t}u^{k_{n_1}+\dots+k_{n_t}}b_{n_1+\dots+n_t}  \\
 & = u^{2(k_{n_1}+\dots+k_{n_t})}(a\sequencetimes b)_{n_1+\dots+n_t}
\end{split}
\end{equation}
and
\begin{equation}
\begin{split}
(a\sequencetimes b)_{n_1+\dots+n_t}
 & = a_{n_1+\dots+n_t}b_{n_1+\dots+n_t}
  \\
 & \le u^{k_{n_1}+\dots+k_{n_t}}a_{n_1}a_{n_2}\dots a_{n_t}u^{k_{n_1}+\dots+k_{n_t}}b_{n_1}b_{n_2}\dots b_{n_t}  \\
 & = u^{2(k_{n_1}+\dots+k_{n_t})}(a\sequencetimes b)_{n_1}(a\sequencetimes b)_{n_2}\dots(a\sequencetimes b)_{n_t},
\end{split}
\end{equation}
which shows that the sequence $a\sequencetimes b$ is approximately geometric as well.

\ref{it:approximatelygeometricdistributive}:
If $c\asymptoticeq G(0)$, then $(a\sequenceplus b)\sequencetimes c\asymptoticeq(a\sequenceplus b)\sequencetimes G(0)\asymptoticeq G(0)=G(0)\sequenceplus G(0)\asymptoticeq(a\sequencetimes c)\sequenceplus(b\sequencetimes c)$. Otherwise let $k_n$ be a monotone increasing sublinear sequence such that for all $t\in\positiveintegers$ and $n_1,\dots,n_t\in\naturals$ the inequalities $c_{n_1}\dots c_{n_t}\le u^{k_{n_1}+\dots+k_{n_t}}c_{n_1+\dots+n_t}$ and $c_{n_1+\dots+n_t}\le u^{k_{n_1}+\dots+k_{n_t}}c_{n_1}\dots c_{n_t}$ hold. Then
\begin{equation}
\begin{split}
((a\sequenceplus b)\sequencetimes c)_n
 & = \sum_{m=0}^n\binom{n}{m}a_{m-n}b_mc_n  \\
 & \le \sum_{m=0}^n\binom{n}{m}a_{m-n}b_m u^{k_{m-n}+k_m}c_{m-n}c_m  \\
 & \le u^{2k_n}\sum_{m=0}^n\binom{n}{m}a_{m-n}c_{m-n}b_m c_m  \\
 & = u^{2k_n}((a\sequencetimes c)\sequenceplus(b\sequencetimes c))_n
\end{split}
\end{equation}
and
\begin{equation}
\begin{split}
((a\sequencetimes c)\sequenceplus(b\sequencetimes c))_n
 & = \sum_{m=0}^n\binom{n}{m}a_{m-n}c_{m-n}b_m c_m  \\
 & \le \sum_{m=0}^n\binom{n}{m}u^{k_{m-n}+k_m}a_{m-n}b_m c_n  \\
 & \le u^{2k_n}((a\sequenceplus b)\sequencetimes c)_n.
\end{split}
\end{equation}
Since $n\mapsto 2k_n$ is sublinear, these inequalities imply $(a\sequenceplus b)\sequencetimes c\asymptoticeq(a\sequencetimes c)\sequenceplus(b\sequencetimes c)$.

\ref{it:approximatelygeometrichomomorphism}:
If $a\asymptoticeq G(0)$, then $\phi(a)\asymptoticeq\phi(G(0))=G(\phi(0))$, therefore $\phi(a)$ is approximately geometric. Otherwise let $k_n$ be a sublinear sequence such that for all $t\in\positiveintegers$ and $n_1,\dots,n_t\in\naturals$ the inequalities $a_{n_1}\dots a_{n_t}\le u^{k_{n_1}+\dots+k_{n_t}}a_{n_1+\dots+n_t}$ and $a_{n_1+\dots+n_t}\le u^{k_{n_1}+\dots+k_{n_t}}a_{n_1}\dots a_{n_t}$ hold. Let $u_T$ be power universal in $T$ and let $p\in\naturals$ such that $\phi(u)\le u_T^p$. Then
\begin{equation}
\begin{split}
\phi(a)_{n_1}\dots \phi(a)_{n_t}
 & = \phi(a_{n_1}\dots a_{n_t})  \\
 & \le \phi(u^{k_{n_1}+\dots+k_{n_t}}a_{n_1+\dots+n_t})  \\
 & = \phi(u)^{k_{n_1}+\dots+k_{n_t}}\phi(a)_{n_1+\dots+n_t}  \\
 & \le u_T^{pk_{n_1}+\dots+pk_{n_t}}\phi(a)_{n_1+\dots+n_t}
\end{split}
\end{equation}
and
\begin{equation}
\begin{split}
\phi(a)_{n_1+\dots+n_t}
 & = \phi(a_{n_1+\dots+n_t})  \\
 & \le \phi(u^{k_{n_1}+\dots+k_{n_t}}a_{n_1}\dots a_{n_t})  \\
 & = \phi(u)^{k_{n_1}+\dots+k_{n_t}}\phi(a_{n_1}\dots a_{n_t})  \\
 & \le u_T^{pk_{n_1}+\dots+pk_{n_t}}\phi(a)_{n_1}\dots \phi(a)_{n_t}.
\end{split}
\end{equation}
The sequence $n\mapsto pk_n$ is sublinear, therefore $\phi(a)$ is approximately geometric.
\end{proof}

The class of approximately geometric sequences depends on the preorder, and enlarging the preorder, i.e., making more pairs comparable, potentially increases the class as well. In particular, an approximately geometric sequence with respect to the asymptotic preorder $\asymptoticle$ may not be approximately geometric with respect to the preorder $\le$. However, we can show that one can always find another sequence that can act as a substitute that is $\le$-approximately geometric and $\asymptoticle$-asymptotically equivalent to it.
\begin{proposition}\label{prop:asymptoticpreorderapproximatelygeometric}
Let $a\in\nonzeros{S}^\naturals$ be a sequence that is approximately geometric with respect to the asymptotic preorder $\asymptoticle$. Then there exists a sequence $x\in\nonzeros{S}^\naturals$ such that $x$ is approximately geometric with respect to $\le$ and asymptotically equivalent to $a$ with respect to $\asymptoticle$.
\end{proposition}
\begin{proof}
Let $(k_m)_{m\in\naturals}$ be a sublinear sequence such that for all $t\in\positiveintegers$, $m_1,\dots,m_t\in\naturals$ the inequalities
\begin{align}
a_{m_1}\cdots a_{m_t} & \asymptoticle u^{k_{m_1}+\dots+k_{m_t}}a_{m_1+\dots+m_t}  \\
a_{m_1+\dots+m_t} & \asymptoticle u^{k_{m_1}+\dots+k_{m_t}}a_{m_1}\cdots a_{m_t}
\end{align}
hold. For each $t\in\positiveintegers$, $m_1,\dots,m_t\in\naturals$, let $(k_{(m_1,\dots,m_t),n})_{n\in\naturals}$ be a sublinear sequence such that for all $n$ the inequalities
\begin{align}
\left(a_{m_1}\cdots a_{m_t}\right)^n & \le u^{k_{(m_1,\dots,m_t),n}+n(k_{m_1}+\dots+k_{m_t})}a_{m_1+\dots+m_t}^n  \\
a_{m_1+\dots+m_t}^n & \le u^{k_{(m_1,\dots,m_t),n}+n(k_{m_1}+\dots+k_{m_t})}\left(a_{m_1}\cdots a_{m_t}\right)^n
\end{align}
hold. In particular, if $t=1$, the inequalities compare $a_{m_1}^n$ with itself, therefore we may assume $k_{(m_1),n}=0$ for all $m_1,n\in\naturals$.

For $M\in\naturals$, let
\begin{equation}
N_0(M)=\min\setbuild{n_0\in\naturals}{n_0\ge M,\forall n\ge n_0,\forall 1\le m\le M: k_{(m,m,\dots,m,1,\dots,1),n}\le n},
\end{equation}
where the index of $k$ contains $m$ with multiplicity $\lfloor\frac{M}{m}\rfloor$ and $1$ with multiplicity $M-m\lfloor\frac{M}{m}\rfloor$. The minimum exists by sublinearity of $n\mapsto k_{(m,m,\dots,m,1,\dots,1),n}$, and $N_0(M)\ge M$ by definition. By our assumption $k_{(m),n}=0$, we have $N_0(1)=1$, while $N_0(0)=0$ since the condition becomes vacuous in this case.

Let $m_0=n_0=r_0=0$, and for all $l\in\positiveintegers$
\begin{align}
m_l & = \max\setbuild{M\in\naturals}{MN_0(M)\le l}  \\
n_l & = \left\lfloor\frac{l}{m_l}\right\rfloor  \\
r_l & = l-n_lm_l,
\end{align}
in particular $m_1=n_1=1$, $r_1=0$, and $n_l\ge N_0(m_l)$ holds for all $l$. Note also that $m_l\le\sqrt{l}$, $\lim_{l\to\infty}m_l=\infty$, and $0\le r_l\le m_l$. For all $l\in\naturals$ we let $x_l=a_{m_l}^{n_l}a_1^{r_l}$.

We claim that $(x_l)_{l\in\naturals}$ is approximately geometric with respect to $\le$. Let $q\in\positiveintegers$, $i,j\in\naturals$, $0\le j<q$. Then
\begin{multline}\label{eq:xtoqblocks}
x_{iq+j}
 = a_{m_{iq+j}}^{n_{iq+j}}a_1^{r_{iq+j}}  \\
 \le u^{k_{(m_q,\dots,m_q,1,\dots,1),n_{iq+j}}+n_{iq+j}\left[\left\lfloor\frac{m_{iq+j}}{m_q}\right\rfloor k_{m_q}+\left(m_{iq+j}-m_q\left\lfloor\frac{m_{iq+j}}{m_q}\right\rfloor\right)k_1\right]}  \\  \cdot a_{m_q}^{n_{iq+j}\left\lfloor\frac{m_{iq+j}}{m_q}\right\rfloor}a_1^{n_{iq+j}\left(m_{iq+j}-m_q\left\lfloor\frac{m_{iq+j}}{m_q}\right\rfloor\right)+r_{iq+j}},
\end{multline}
and a similar reverse inequality holds with the $u$ factor on the opposite side. To compare with $x_q^ix_1^j=x_q^ia_1^j$, we need to change the exponent of $a_{m_q}$ to $in_q$. We use either the inequality (with the $m_q$-tuple $(1,\ldots,1)$)
\begin{align}\label{eq:amqtopower}
a_{m_q}^h & \le u^{k_{(1,\dots,1),h}+hm_qk_1}a_1^{hm_q}
\intertext{or}
a_1^{hm_q} & \le u^{k_{(1,\dots,1),h}+hm_qk_1}a_{m_q}^h
\end{align}
to achieve this, with
\begin{equation}
h=\left\lvert in_q-n_{iq+j}\left\lfloor\frac{m_{iq+j}}{m_q}\right\rfloor\right\rvert.
\end{equation}
We need to estimate the limit superior of the total exponent of $u$ divided by $iq+j$ as $i\to\infty$, $0\le j<q$. By the choice of $m_l$ and $n_l$, we have
\begin{equation}
\frac{k_{(m_q,\dots,m_q,1,\dots,1),n_{iq+j}}}{iq+j}
 \le \frac{n_{iq+j}}{iq+j}
 \le \frac{1}{m_{iq+j}},
\end{equation}
which goes to $0$ since $l\mapsto m_l$ is unbounded. For the two parts of the second term in the exponent in \eqref{eq:xtoqblocks} we get
\begin{equation}
\limsup_{i\to\infty}\frac{1}{iq+j}n_{iq+j}\left\lfloor\frac{m_{iq+j}}{m_q}\right\rfloor=\frac{k_{m_q}}{m_q}
\end{equation}
and
\begin{equation}
\limsup_{i\to\infty}\frac{1}{iq+j}n_{iq+j}\left(m_{iq+j}-m_q\left\lfloor\frac{m_{iq+j}}{m_q}\right\rfloor\right)k_1
 \le \limsup_{i\to\infty}\frac{1}{m_{iq+j}}m_qk_1
 = 0
\end{equation}

Choose $\epsilon>0,c\in\naturals$ such that $k_{(1,\dots,1),n}\le c+\epsilon n$ for all $n$ ($1$ with multiplicity $m_q$). Then the exponent in \eqref{eq:amqtopower} satisfies
\begin{equation}
k_{(1,\dots,1),h}+hm_qk_1
 \le c+h(\epsilon+m_qk_1).
\end{equation}
To bound $h$, we use
\begin{align}
i(q-m_q)
 & \le in_qm_q
   \le iq  \\
n_{iq+j}m_q\left\lfloor\frac{m_{iq+j}}{m_q}\right\rfloor
 & \le n_{iq+j}m_{iq+j}
   \le iq+j
\end{align}
and
\begin{equation}
\begin{split}
n_{iq+j}m_q\left\lfloor\frac{m_{iq+j}}{m_q}\right\rfloor
 & \ge n_{iq+j}m_q\left(\frac{m_{iq+j}}{m_q}-1\right)  \\
 & = n_{iq+j}m_{iq+j}-n_{iq+j}m_q  \\
 & \ge iq+j-m_{iq+j}-n_{iq+j}m_q.
\end{split}
\end{equation}
These imply
\begin{equation}
-m_q-j\le in_qm_q-n_{iq+j}m_q\left\lfloor\frac{m_{iq+j}}{m_q}\right\rfloor\le m_{iq+j}+n_{iq+j}m_q-j,
\end{equation}
therefore
\begin{equation}
h\le \frac{m_{iq+j}}{m_q}+n_{iq+j}+\frac{j}{m_q}.
\end{equation}
With these we bound the growth rate of the exponent in \eqref{eq:amqtopower} as
\begin{equation}
\begin{split}
\limsup_{i\to\infty}\frac{k_{(1,\dots,1),h}+hm_qk_1}{iq+j}
 & \le \limsup_{i\to\infty}\frac{c+h(\epsilon+m_qk_1)}{iq+j}  \\
 & \le (\epsilon+m_qk_1)\limsup_{i\to\infty}\frac{1}{iq+j}\left(\frac{m_{iq+j}}{m_q}+n_{iq+j}+\frac{j}{m_q}\right)  \\
 & \le (\epsilon+m_qk_1)\limsup_{i\to\infty}\left(\frac{1}{m_q\sqrt{iq+j}}+\frac{1}{m_{iq+j}}\right)  \\
 & = 0.
\end{split}
\end{equation}

Combining these estimates, the total exponent (\eqref{eq:xtoqblocks} and \eqref{eq:amqtopower}) divided by $iq+j$ has limit superior at most $\frac{k_{m_q}}{m_q}$, which can be made arbitrarily small by choosing $q$ sufficiently large. Therefore $x$ is approximately geometric with respect to $\le$.

Between the entries of $x$ and $a$, the asymptotic inequalities
\begin{align}
x_l=a_{m_l}^{n_l}a_1^{r_l} & \asymptoticle u^{n_lk_{m_l}+r_lk_1}a_{n_lm_l+r_l}=u^{n_lk_{m_l}+r_lk_1}a_l
\intertext{and}
a_l=a_{n_lm_l+r_l} & \asymptoticle u^{n_lk_{m_l}+r_lk_1}a_{m_l}^{n_l}a_1^{r_l}=u^{n_lk_{m_l}+r_lk_1}x_l
\end{align}
hold. Since
\begin{equation}
\limsup_{l\to\infty}\frac{1}{l}\left(n_lk_{m_l}+r_lk_1\right)
 \le \limsup_{l\to\infty}\left(\frac{n_lm_l}{l}\frac{k_{m_l}}{m_l}+\frac{m_l}{l}k_1\right)
 = 0,
\end{equation}
this implies that $x$ and $a$ are asymptotically equivalent with respect to $\asymptoticle$.
\end{proof}

\section{Asymptotic completeness}\label{sec:asymptoticcompleteness}

We continue with a notion of completeness related to approximately geometric sequences.
\begin{definition}
A preordered semiring $(S,+,\cdot,0,1,\le)$ is \emph{asymptotically complete} if for every approximately geometric sequence $a\in(\nonzeros{S})^\naturals$ there exists an element $s\in S$ such that $a\asymptoticeq G(s)$.
\end{definition}

\begin{example}\label{ex:asymptoticallycomplete}\leavevmode
\begin{enumerate}
\item The semiring $\naturals$ with its usual preorder is not asymptotically complete. For example, the sequence $a_n=2^{\lfloor n/2\rfloor}$ is approximately geometric by \cref{ex:approximatelygeometric}, but if $s\ge 2$, then $s^n\le 2^{k_n}a_n$ requires $k_n\ge n(-\frac{1}{2}+\log_2s)\ge\frac{n}{2}$, while if $s\le 1$, then $a_n\le 2^{k_n}s^n$ implies $k_n\ge \lfloor n/2\rfloor$.
\item For a semiring $S$, the preordered semiring $S^=$ is asymptotically complete.
\item $\nonnegativereals$, $\tropicalreals$, $\nonnegativereals\opposite$, and $\tropicalreals\opposite$ are asymptotically complete. By \cref{ex:approximatelygeometric}, $(a_n)_{n\in\naturals}\in(\nonzeros{S})^n$ is approximately geometric if the limit $\lim_{n\to\infty}\frac{\ln a_n}{n}=:\ln A$ exists, therefore $k_n=\left\lceil\left\lvert n\log_u A-\log_u a_n\right\rvert\right\rceil$ is sublinear. By construction, $a_n\le u^{k_n}A^n$ and $A^n\le u^{k_n}a_n$ holds for all $n$.
\item $\dualnumbers$ is asymptotically complete. By \cref{ex:approximatelygeometric}, a nonzero sequence $(r_n+s_nX)_{n\in\naturals}$ is approximately geometric if and only if $(r_n)_{n\in\naturals}$ is geometric and the limit $\lim_{n\to\infty}\frac{1}{n}\frac{s_n}{r_n}=:\frac{d}{r_1}$ exists, therefore
\begin{equation}
k_n=\left\lceil\left\lvert n\frac{d}{r_1}-\frac{s_n}{r_n}\right\rvert\right\rceil
\end{equation}
is sublinear. By construction, $r_n+s_nX\le(1+X)^{k_n}(r_1+dX)^n$ and $(r_1+dX)^n\le(1+X)^{k_n}(r_n+s_nX)$ holds for all $n$.
\end{enumerate}
\end{example}
Examining the examples above, we see that while $\naturals$ is not asymptotically complete, it does embed into the asymptotically complete preordered semiring $\nonnegativereals$, and we may ask whether this is true in general, and if there is a unique way to make a semiring asymptotically complete. To illustrate that this requires some care and extra conditions, we can consider the preordered semiring $\nonnegativereals\cup\{2'\}$ with $a+2'=a+2$ and $a\cdot 2'=a\cdot 2$, $a\le 2'\iff a\le 2$, and $2'\le a\iff 2\le a$. This preordered semiring is asymptotically complete and $\naturals$ embeds in the same way as into $\nonnegativereals$. Clearly, the non-uniqueness in this example stems from the preorder's inability to separate $2$ from $2'$.

The situation is analogous to the completion of topological vector spaces: in the non-Hausdorff case, infinitely many non-Hausdorff completions exist, but uniqueness can be guaranteed by requiring the completion to be a Hausdorff space. In a similar way, we ensure uniqueness in the following definition by requiring that the preorder is a closed partial order.
\begin{definition}\label{def:asymptoticcompletion}
Let $S$ be a preordered semiring. An \emph{asymptotic completion} of $S$ is a closed partially ordered semiring $\completion{S}$ together with a monotone homomorphism $i:S\to\completion{S}$ such that
\begin{enumerate}
\item $i(u)$ is power universal in $\completion{S}$
\item $\completion{S}$ is asympotically complete
\item if $T$ is an arbitrary asymptotically complete closed partially ordered semiring and $f:S\to T$ is a monotone homomorphism, then there is a unique $g:\completion{S}\to T$ such that $f=g\circ i$.
\end{enumerate}
\end{definition}
Note that the map $i$ is in general not injective, but if $i(a)=i(b)$, then $a\asymptoticeq b$. In the context of topological vector spaces, this may be compared to quotienting by the closure of the origin, as a source of non-injectivity of the canonical map.

It follows from the definition that asymptotic completions are essentially unique. Suppose that $i_1:S\to\completion{S}_1$ and $i_2:S\to\completion{S}_2$ are the monotone homomorphisms for two asymptotic completions $\completion{S}_1$ and $\completion{S}_2$ of $S$. Applying the definition with $\completion{S}_1$ and $f=i_2$ gives a unique $g_{1\to 2}:\completion{S}_1\to\completion{S}_2$ such that $g_{1\to 2}\circ i_1=i_2$, and by reversing the roles, we obtain a unique $g_{2\to 1}:\completion{S}_2\to\completion{S}_1$ such that $g_{2\to 1}\circ i_2=i_1$. The compositions $g_{2\to 1}\circ g_{1\to 2}:\completion{S}_1\to\completion{S}_1$ and $g_{1\to 2}\circ g_{2\to 1}:\completion{S}_2\to\completion{S}_2$ satisfy $(g_{2\to 1}\circ g_{1\to 2})\circ i_1=i_1$ and $(g_{1\to 2}\circ g_{2\to 1})\circ i_2=i_2$, so by uniqueness both compositions are identity maps.

We will make use of the following lemma in the proof of existence of an asymptotic completion.
\begin{lemma}\label{lem:approximatelygeometricdense}
Let $S\subseteq T$ be preordered semirings such that some $u\in S$ is power universal in $T$ and every geometric sequence in $T$ is asymptotically equivalent to an approximately geometric sequence in $S$. Then every approximately geometric sequence in $T$ is asymptotically equivalent to an approximately geometric sequence in $S$.
\end{lemma}
\begin{proof}
Let $b=(b_n)_{n\in\naturals}$ be an approximately geometric sequence in $T$. If $b\asymptoticeq G(0)$, then it is asymptotically equivalent to an approximately geometric sequence in $S$. Suppose that $b\not\asymptoticeq G(0)$.

For all $m\in\naturals$, let $a_m=(a_{m,n})_{n\in\naturals}\in(\nonzeros{S})^\naturals$ such that $a_m\asymptoticeq G(b_m)$. Let $(k_m)_{m\in\naturals}$ be a sublinear sequence such that for all $t\in\positiveintegers$, $m_1,\dots,m_t\in\naturals$ the inequalities $b_{m_1+\dots+m_t}\le u^{k_{m_1}+\dots+k_{m_t}}b_{m_1}\cdots b_{m_t}$ and $b_{m_1}\cdots b_{m_t}\le u^{k_{m_1}+\dots+k_{m_t}}b_{m_1+\dots+m_t}$ hold. For all $m\in\naturals$, let $(k_{m,n})_{n\in\naturals}$ be sublinear such that $a_{m,n}\le u^{k_{m,n}}b_m^n$ and $b_m^n\le u^{k_{m,n}}a_{m,n}$ for all $m,n$.

For $M\in\naturals$, let
\begin{equation}
N_0(M)=\min\setbuild{n_0\in\naturals}{n_0\ge M,\forall n\ge n_0:k_{M,n}\le n}.
\end{equation}
This minimum exists by sublinearity of $n\mapsto k_{M,n}$, and $N_0(M)\ge M$ by definition.

Let $m_0=n_0=r_0=0$, and for all $l\in\positiveintegers$
\begin{align}
m_l & = \max\setbuild{M\in\naturals}{MN_0(M)\le l}  \\
n_l & = \left\lfloor \frac{l}{m_l}\right\rfloor  \\
r_l & = l-n_lm_l.
\end{align}
Then $m_l\le\sqrt{l}$, $\lim_{l\to\infty}m_l=\infty$, $n_l\ge N_0(m_l)$, $m_ln_l\le l$, and $0\le r_l\le m_l$. Let $x_l=a_{m_l,n_l}a_{1,1}^{r_l}\in\nonzeros{S}$ for all $l\in\naturals$.

We claim that $x\asymptoticeq b$. Indeed,
\begin{equation}
\begin{split}
x_l
 & = a_{m_l,n_l}a_{1,1}^{r_l}  \\
 & \le u^{k_{m_l,n_l}+r_lk_{1,1}}b_{m_l}^{n_l}b_1^{r_l}  \\
 & \le u^{k_{m_l,n_l}+r_lk_{1,1}+n_lk_{m_l}+r_lk_1}b_{m_ln_l+r_l}  \\
 & = u^{k_{m_l,n_l}+r_lk_{1,1}+n_lk_{m_l}+r_lk_1}b_l,
\end{split}
\end{equation}
and we similarly obtain
\begin{equation}
b_l \le u^{k_{m_l,n_l}+r_lk_{1,1}+n_lk_{m_l}+r_lk_1}x_l.
\end{equation}
The growth rate of the exponent can be estimated as
\begin{equation}
\begin{split}
\limsup_{l\to\infty}\frac{k_{m_l,n_l}+r_lk_{1,1}+n_lk_{m_l}+r_lk_1}{l}
 & \le \limsup_{l\to\infty}\frac{n_l+m_l(k_{1,1}+k_1)+n_lk_{m_l}}{l}  \\
 & \le \limsup_{l\to\infty}\frac{1}{m_l}+\frac{k_{1,1}+k_1}{\sqrt{l}}+\frac{k_{m_l}}{m_l}
 = 0
\end{split}
\end{equation}

By \cref{prop:approximatelygeometricproperties}, part \ref{it:approximatelygeometricequivalent}, $x$ is also approximately geometric.
\end{proof}
Suppose that $\completion{S}$, $i:S\to\completion{S}$ is an asymptotic completion of $S$, and consider the subsemiring $i(S)$. For every approximately geometric sequence in $i(S)$ we can find an element $a$ such that the sequence is asymptotically equivalent to $G(a)$. Since the preorder is a closed partial order, the element $a$ is unique ($G(a)\asymptoticeq G(a')$ is equivalent to $a\le a'$ and $a\ge a'$ if the preorder is closed). Let $T$ be the set of all elements $a$ arising in this way. Applying \cref{lem:approximatelygeometricdense} to $i(S)\subseteq T$, we can see that $T$ is asymptotically complete, so by uniqueness, $T=\completion{S}$.

This suggests that an asymptotic completion may be constructed as a quotient of the set of approximately geometric sequences. In the following, the equivalence class of a sequence $s$ with respect to asymptotic equivalence will be denoted by $[s]$.
\begin{theorem}\label{thm:asymptoticcompletion}
Let $\completion{S}$ be the set of asymptotic equivalence classes of approximately geometric sequences in $S$, with addition $+$, multiplication $\cdot$, and preorder $\le$ induced by the operations $\sequenceplus$, $\sequencetimes$, and the preorder $\asymptoticle$. Let $i:S\to\completion{S}$ be defined by $s\mapsto[G(s)]$. Then $\completion{S}$ with $i$ is an asymptotic completion of $S$.

If $\phi:S\to T$ is a monotone semiring homomorphism, then $\completion{\phi}([a]):=[\phi(a)]$ defines a monotone semiring homomorphism $\completion{\phi}:\completion{S}\to\completion{T}$ that satisfies $\completion{\phi}\circ i_S=i_T\circ\phi$.
\end{theorem}
\begin{proof}
By \cref{prop:sequenceoperationsproperties,prop:approximatelygeometricproperties}, the operations on $\completion{S}$ are well defined, associative, commutative, distributive, and compatible with the preorder, and $[G(0)]$ and $[G(1)]$ are the null and the unit elements. Since asymptotic equivalence is the largest equivalence relation containted in the preorder on sequences, the induced preorder on $\completion{S}$ is a partial order. The map $i$ is a monotone semiring homomorphism by \cref{ex:geometric}. In fact, for $a,b\in S$ we have $i(a)\le i(b)$ iff $a\asymptoticle b$.

We show that $i(u)$ is power universal. Since $1\le u$ in $S$, we have $[G(1)]\le[G(u)]$ in $\completion{S}$. Suppose that $a,b\in (\nonzeros{S})^\naturals$ are approximately geometric sequences and $[a]\le[b]$. Then there exists a sublinear sequence $(k_n)_{n\in\naturals}$ such that $a_n\le u^{k_n}b_n$ for all $n$, and for all $t\in\positiveintegers$, $n_1,\dots,n_t\in\naturals$ the inequalitites
\begin{align}
b_{n_1+\dots+n_t} & \le u^{k_{n_1}+\dots+k_{n_t}}b_{n_1}b_{n_2}\cdots b_{n_t}
\intertext{and}
a_{n_1}a_{n_2}\cdots a_{n_t} & \le u^{k_{n_1}+\dots+k_{n_t}}a_{n_1+\dots+n_t}
\end{align}
hold. Let $p\in\naturals$ such that $b_1\le u^p a_1$.
\begin{equation}
b_n
 \le u^{nk_1}b_1^n
 \le u^{nk_1}(u^p a_1)^n
 = u^{n(p+k_1)}a_1^n
 \le u^{n(p+2k_1)}a_n,
\end{equation}
therefore $[b]\le i(u)^{p+2k_1}[a]$.

We show that the partial order $\le$ on $\completion{S}$ is closed. Let $a,b$ be approximately geometric sequences in $S$ and suppose that $[a]\asymptoticle[b]$. Let $\epsilon>0$ and choose $m\in\naturals$ such that $[a]^m\le i(u)^{\lfloor\epsilon m\rfloor}[b]^m$. Let $c\in\naturals$ be such that for all $n$ the inequality $a_n^m\le u^{c+\lfloor\epsilon n\rfloor}u^{n\lfloor \epsilon m\rfloor}b_n^m$ holds, and for all $t\in\positiveintegers$, $n_1,\dots,n_t\in\naturals$ the inequalities
\begin{align}
a_{n_1+\dots+n_t} & \le u^{c+\lfloor\epsilon(n_1+\dots n_t)\rfloor}a_{n_1}a_{n_2}\cdots a_{n_t}
\intertext{and}
b_{n_1}b_{n_2}\cdots b_{n_t} & \le u^{c+\lfloor\epsilon(n_1+\dots n_t)\rfloor}b_{n_1+\dots+n_t}
\end{align}
hold. Let $p\in\naturals$ such that $a_1\le u^p b_1$. Then for all $N=mn+r$, $n,r\in\naturals$, $0\le r<m$ we have
\begin{equation}
\begin{split}
a_N
 & \le u^{c(m+r)+\lfloor\epsilon N\rfloor}a_n^ma_1^r  \\
 & \le u^{c(m+r)+\lfloor\epsilon N\rfloor+\lfloor\epsilon n\rfloor+n\lfloor\epsilon m\rfloor+pr}b_n^mb_1^r  \\
 & \le u^{2c(m+r)+2\lfloor\epsilon N\rfloor+\lfloor\epsilon n\rfloor+n\lfloor\epsilon m\rfloor+pr}b_N  \\
 & \le u^{(4c+p)m+4\lfloor\epsilon N\rfloor}b_N.
\end{split}
\end{equation}
This is true for arbitrary $\epsilon>0$ and for all $N$ (with $c$ depending only on $\epsilon$), therefore we have $a\asymptoticle b$, i.e., $[a]\le[b]$.

Let $a=(a_n)_{n\in\naturals}$ be approximately geometric in $S$. We claim that $i(a)\asymptoticeq G([a])$. Let $(k_n)_{n\in\naturals}$ be sublinear such that $a_{n_1+\dots+n_t}\le u^{k_{n_1}+\dots+k_{n_t}}a_{n_1}\cdots a_{n_t}$ and $a_{n_1}\cdots a_{n_t}\le u^{k_{n_1}+\dots+k_{n_t}}a_{n_1+\dots+n_t}$ holds for all $t\in\positiveintegers$, $n_1,\dots,n_t\in\naturals$. The elements \begin{equation}
i(a)_n=i(a_n)=[(1,a_n,a_n^2,\dots)]
\end{equation}
and
\begin{equation}
G([a])_n=[(a_0^n,a_1^n,a_2^n,\dots)]
\end{equation}
of $\completion{S}$ are represented by sequences that differ in the exchange of the index and the exponent, and
\begin{equation}
a_n^m
 \le u^{mk_n}a_{mn}
 \le u^{mk_n+nk_m}a_m^n,
\end{equation}
for all $m,n\in\naturals$, therefore $i(a)_n\le u^{k_n}G([a])_n$ and $G([a])_n\le u^{k_n}i(a)_n$. Since $n\mapsto k_n$ is sublinear, this implies $i(a)\asymptoticeq G([a])$.

We show that $\completion{S}$ is asymptotically complete. By the previous observation, every geometric sequence in $\completion{S}$ is asymptotically equivalent to an approximately geometric one in $i(S)$. Let $b$ be an approximately geometric sequence in $\completion{S}$. By \cref{lem:approximatelygeometricdense}, it is also asymptotically equivalent to an approximately geometric one in $i(S)$, say $i(a)=(i(a_n))_{n\in\naturals}$ with $a=(a_n)_{n\in\naturals}$, $a_n\in S$ for all $S$. This means that $(a_n)_{n\in\naturals}$ is an approximately geometric sequence with respect to the asymptotic preorder of $S$. By \cref{prop:asymptoticpreorderapproximatelygeometric}, there is a sequence $x$ in $S$ that is approximately geometric with respect to the preorder of $S$ and asymptotically equivalent to $a$ with respect to the asymptotic preorder, therefore $i(x)\asymptoticeq i(a)$ and $i(x)\asymptoticeq G([x])$, again by the previous observation. Therefore $b\asymptoticeq G([x])$.

Let $T$ be an asymptotically complete closed partially ordered semiring and $f:S\to T$ a monotone homomorphism. If $a$ is an approximately geometric sequence in $S$, then $f(a)$ is approximately geometric by \cref{prop:approximatelygeometricproperties}, therefore there is a $t\in T$ such that $f(a)\asymptoticeq G(t)$. Since the preorder of $T$ is a closed partial order, this $t$ is unique. If $a\asymptoticeq a'$, then $f(a)\asymptoticeq f(a')$ by \cref{prop:sequenceoperationsproperties}. This means that there is a unique map $g:\completion{S}\to T$ such that $G(g([a]))\asymptoticeq f(a)$. This map is a monotone semiring homomorphism by \cref{prop:sequenceoperationsproperties,ex:geometric}. If $s\in S$, then $i(s)=[G(s)]$, and $f(G(s))=G(f(s))$, therefore $g(i(s))=f(s)$, i.e., $f=g\circ i$.

Suppose that $g':\completion{S}\to T$ is a monotone semiring homomorphism that satisfies $f=g'\circ i$. Let $a$ be an approximately geometric sequence in $S$, and let $(k_n)_{n\in\naturals}$ be sublinear such that for all $t\in\positiveintegers$, $n_1,\dots,n_t$ the inequalities $a_{n_1}\cdots a_{n_t}\le u^{k_{n_1}+\dots+k_{n_t}}a_{n_1+\dots+n_t}$ and $a_{n_1+\dots+n_t}\le u^{k_{n_1}+\dots+k_{n_t}}a_{n_1}\cdots a_{n_t}$ hold. Then for all $m,n\in\naturals$ we have
\begin{equation}
a_n^m
 \le u^{mk_n}a_{mn}
 \le u^{mk_n+nk_m}a_m^n
 = u^{mk_n}G(u^{k_m}a_m)_n,
\end{equation}
which means $a^m\asymptoticle G(u^{k_m}a_m)$. Taking equivalence classes and applying $g'$, we obtain for all $m\in\naturals$ the inequality
\begin{equation}
g'([a])^m
 \asymptoticle g'([G(u^{k_m}a_m)])
 = f(u^{k_m}a_m)
 = f(u)^{k_m}f(a_m),
\end{equation}
i.e., $G(g'([a]))\asymptoticle f(a)$. In a similar way we obtain $G(g'([a]))\asymptoticge f(a)$. This implies that we have $G(g'([a]))\asymptoticeq f(a)\asymptoticeq G(g([a]))$, so by uniqueness $g'([a])=g([a])$.

Let $\phi:S\to T$ be a monotone semiring homomorphism, and let $i_S:S\to\completion{S}$ and $i_T:T\to\completion{T}$ denote the maps $a\mapsto [G(a)]$ for $a\in S$ and $a\in T$, respectively. By \cref{prop:sequenceoperationsproperties}, the map $\completion{\phi}([a]):=[\phi(a)]$ is a well-defined monotone semiring homomorphism. If $a\in S$, then
\begin{equation}
(\completion{\phi}\circ i_S)(a)
 = \completion{\phi}([G(a)])
 = [\phi(G(a))]
 = [G(\phi(a))]
 = (i_T\circ\phi)(a).
\end{equation}

\end{proof}

\begin{remark}
We can see from the proof (or from the definition, and combining \cref{lem:approximatelygeometricdense,prop:asymptoticpreorderapproximatelygeometric}) that every element $a$ of $\completion{S}$ can be represented by an approximately geometric sequence $x=(x_n)_{n\in\naturals}\in S^\naturals$ in $S$ in the sense that $G(a)\asymptoticeq i(x)$. In the concrete realization constructed in \cref{thm:asymptoticcompletion}, in addition $a=[x]$ holds.
\end{remark}

\section{Spectrum}\label{sec:spectrum}

There exist various related notions of a spectrum of a preordered semiring (with suitable properties) that encode asymptotic information in the sense that elements of the semiring give rise to evaluation maps on the spectrum, with the asymptotic preorder corresponding to the pointwise order between the functions. Our goal in this section is to extend these to the asymptotic preorder on approximately geometric sequences, i.e., the order on the asymptotic completion.

In particular, the \emph{asymptotic spectrum} of a Strassen semiring $S$ is (as a set) $\Hom(S,\nonnegativereals)$ \cite{strassen1988asymptotic}, the \emph{test spectrum} of a preordered semiring of polynomial growth satisfying $0\le 1$ is a subset of $\Hom(S,\nonnegativereals)\cup\Hom(S,\tropicalreals)$ \cite{fritz2023abstract}, while for a more general class of preordered semirings, the \emph{test spectrum} is constructed from monotone homomorphisms into $\nonnegativereals$, $\tropicalreals$, $\nonnegativereals\opposite$, $\tropicalreals\opposite$, and $\dualnumbers$ \cite{fritz2021abstract2}. In each case, the spectrum can be endowed with a compact topology generated by (logarithmic) evaluation or comparision maps, but we will not use this structure.

Since the spectra contain homomorphisms into asymptotically complete preordered semirings $\mathbb{K}$ (see \cref{ex:asymptoticallycomplete}), it follows from \cref{def:asymptoticcompletion} that composition with $i$ provides a bijection $\Hom(\completion{S},\mathbb{K})\to\Hom(S,\mathbb{K})$.

Under this bijection, a degenerate homomorphism $\norm[S]{\cdot}:S\to\nonnegativereals$ corresponds to a degenerate homomorphism $\norm[\completion{S}]{\cdot}:\completion{S}\to\nonnegativereals$, and if it is the first component of a monotone homomorphism $\psi:S\to\dualnumbers$, i.e., the second component is a derivation $D$ at $\norm[S]{\cdot}$, then the second component of $\completion{\psi}$ is a derivation $\completion{D}$ at $\norm[\completion{S}]{\cdot}$. We mention that if $D_1,D_2:S\to\reals$ are interchangeable (in the sense of \cite[8.2. Definition]{fritz2021abstract2}) derivations at $\norm[S]{\cdot}$, then the induced derivations $\completion{D_1},\completion{D_2}:S\to\reals$ are interchangeable as well.

\begin{theorem}\label{thm:dualitycompletion}
Let $\mathcal{K}$ be a set consisting of asymptotically closed partially ordered semirings. Suppose that $S$ is a preordered semiring of polynomial growth such that for all $a,b\in S$ the asymptotic inequality $a\asymptoticle b$ is equivalent to the collection of inequalities $\phi(a)\le\phi(b)$ for all $\mathbb{K}\in\mathcal{K}$ and all $\phi\in\Hom(S,\mathbb{K})$.

Let $a,b\in\completion{S}$ such that $a\sim b$. Then $a\le b$ is equivalent to the collection of inequalities $\completion{\phi}(a)\le\completion{\phi}(b)$ for all $\mathbb{K}\in\mathcal{K}$ and all $\completion{\phi}\in\Hom(\completion{S},\mathbb{K})$.
\end{theorem}
\begin{proof}
We work with the concrete realization of $\completion{S}$ constructed in \cref{thm:asymptoticcompletion}. Let $x=(x_n)_{n\in\naturals}$ and $y=(y_n)_{n\in\naturals}$ be approximately geometric sequences in $S$ such that $i(x)\asymptoticeq G(a)$ and $i(y)\asymptoticeq G(b)$, therefore $a=[x]$ and $b=[y]$.

Let $\mathbb{K}\in\mathcal{K}$, and $v\in\mathbb{K}$ power universal. For $\phi\in\Hom(S,\mathbb{K})$ and the corresponding $\completion{\phi}\in\Hom(\completion{S},\mathbb{K})$, the inequality $[\phi(x)]=\completion{\phi}(a)\le\completion{\phi}(b)=[\phi(y)]$ means
$\phi(x)\asymptoticle \phi(y)$, i.e., there is a sublinear sequence $(k_n)_{n\in\naturals}$ such that $\phi(x_n)\le v^{k_n}\phi(y_n)$ holds for all $n\in\naturals$.

Let $p\in\naturals$ be such that $a_1\le u^pb_1$, and for $\epsilon>0$, let $m\in\naturals$ and $c\in\naturals$ be such that for all $q,r\in\naturals$, $0\le r<m$ the inequalities
\begin{align}
x_{qm+r} & \le u^{c+\lfloor\epsilon qm\rfloor}x_m^qx_1^r  \\
x_m^qa_1^r & \le u^{c+\lfloor\epsilon qm\rfloor}x_{qm+r}  \\
y_{qm+r} & \le u^{c+\lfloor\epsilon qm\rfloor}y_m^qy_1^r  \\
y_m^qy_1^r & \le u^{c+\lfloor\epsilon qm\rfloor}y_{qm+r}
\end{align}
hold.

Setting $r=1$, for every $q\in\naturals$ we have
\begin{equation}
\begin{split}
\phi(x_m)^q
 & = \phi(x_m^q)  \\
 & \le \phi(u^{c+\lfloor\epsilon qm\rfloor}x_{qm})  \\
 & = \phi(u)^{c+\lfloor\epsilon qm\rfloor}\phi(x_{qm})  \\
 & \le \phi(u)^{c+\lfloor\epsilon qm\rfloor}v^{k_{qm}}\phi(y_{qm})  \\
 & \le \phi(u)^{c+\lfloor\epsilon qm\rfloor}v^{k_{qm}}\phi(u^{c+\lfloor\epsilon qm\rfloor}y_m^q)  \\
 & \le \phi(u)^{2c}v^{k_{qm}}\phi(u^{2\lceil\epsilon m\rceil} y_m)^q.
\end{split}
\end{equation}
As this holds for all $q$ and the preorder of $\mathbb{K}$ is closed, we have $\phi(x_m)\le\phi(u^{2\lceil\epsilon m\rceil} y_m)$. The inequality is satisfied for all $\mathbb{K}\in\mathcal{K}$ and $\phi\in\Hom(S,\mathbb{K})$, therefore $x_m\asymptoticle u^{2\lceil\epsilon m\rceil} y_m$, which means that for some sublinear sequence $(k_n)_{n\in\naturals}$, the inequalities $x_m^q\asymptoticle u^{k_q+2\lceil\epsilon m\rceil q} y_m^q$ hold for all $q$. Then
\begin{equation}
\begin{split}
x_{qm+r}
 & \le u^{c+\lfloor\epsilon qm\rfloor}x_m^qx_1^r  \\
 & \le u^{c+\lfloor\epsilon qm\rfloor}u^{k_q+2\lceil\epsilon m\rceil q} y_m^q u^{rp}y_1^r  \\
 & \le u^{2c+2\lfloor\epsilon qm\rfloor+k_q+2\lceil\epsilon m\rceil q+rp} y_{qm+r},
\end{split}
\end{equation}
which holds for all $\epsilon>0$, all sufficiently large $m\in\naturals$, and all $q\in\naturals$, $r\in\{0,1,\dots,m-1\}$. This implies $x\asymptoticle y$, i.e., $a\le b$.
\end{proof}

If $S$ has a Strassen preorder, then the assumption of \cref{thm:dualitycompletion} is satisfied with $\mathcal{K}=\{\nonnegativereals\}$, which leads to the following corollary:
\begin{corollary}\label{cor:dualityrealhomomorphisms}
Let $S$ be a preordered semiring with a Strassen preorder, and let $a,b\in\completion{S}$. Then $a\le b$ if and only if $\phi(a)\le\phi(b)$ holds for all $\phi\in\Hom(\completion{S},\nonnegativereals)$.
\end{corollary}
More generally, the same conclusion holds under the assumptions of \cite[Theorem 2]{vrana2021generalization}.

If $S$ is a preordered semiring of polynomial growth with $1\ge 0$, then \cref{thm:Vergleichsstellentropical} applies, which implies the following by \cref{thm:dualitycompletion}:
\begin{corollary}\label{cor:dualitytropicalhomomorphisms}
Let $S$ be a preordered semiring of polynomial growth with $1\ge 0$, and let $a,b\in\completion{S}$. Then $a\le b$ if and only if $\phi(a)\le\phi(b)$ holds for all $\phi\in\Hom(\completion{S},\nonnegativereals)$ and for all $\phi\in\Hom(\completion{S},\tropicalreals)$.
\end{corollary}

Finally, by combining \cref{thm:Vergleichsstellenderivations} with \cref{thm:dualitycompletion}, we obtain the following:
\begin{corollary}\label{cor:dualityderivations}
Let $S$ is a preordered semiring of polynomial growth, and suppose that for some $d\in\naturals$ there is a surjective homomorphism $\norm{\cdot}:S\to\positivereals^d\cup\{0\}$ such that $a\le b\implies\norm{a}=\norm{b}\implies a\sim b$ for $a,b\in S$. Let $a,b\in\completion{S}$. Then $a\le b$ if and only if $\phi(a)\le\phi(b)$ holds for all $\phi\in\Hom(\completion{S},\mathbb{K})$ with $\mathbb{K}\in\{\nonnegativereals,\nonnegativereals\opposite,\tropicalreals,\tropicalreals\opposite\}$, and for every $i=1,\dots,d$ and monotone $\norm[i]{\cdot}^{\completion{\phantom{.}}}$-derivation $D:\completion{S}\to\reals$, $D(a)\le D(b)$.
\end{corollary}

Some asymptotic parameters of interest (such as the asymptotic tensor rank or the Shannon capacity) are not spectral points but possess weaker properties in connection with the semiring operations and the preorder. We give a sufficient condition in the context of Strassen preorders in terms of the following definitions \cite{bugar2024interpolating}.
\begin{definition}
Let $S$ be a preordered semiring. A map $f:S\to\nonnegativereals$ is a \emph{lower functional} if $f(0)=0$, $f(1)=1$, and for $x,y\in S$ it satisfies $f(x+y)\ge f(x)+f(y)$ and $f(xy)\ge f(x)f(y)$, and $x\le y$ implies $f(x)\le f(y)$. A map $g:S\to\nonnegativereals$ is an \emph{upper functional} if $g(0)=0$, $g(1)=1$, and for $x,y\in S$ it satisfies $g(x+y)\le g(x)+g(y)$ and $g(xy)\le g(x)g(y)$, and $x\le y$ implies $g(x)\le g(y)$. The lower (upper) functional $f$ ($g$) is \emph{regular}, if $f(p(x))=p(f(x))$ ($g(p(x))=p(g(x))$) holds for all $x\in S$ and all univariate polynomials $p$ with natural numbers as coefficients.
\end{definition}

\begin{proposition}\label{prop:regularfunctionalextension}
Let $S$ be a semiring with a Strassen preorder, and $i:S\to\completion{S}$ an asymptotic completion.
\begin{enumerate}
\item If $f:S\to\nonnegativereals$ is a regular lower functional, then there exists a unique regular lower functional $\completion{f}:\completion{S}\to\nonnegativereals$ such that $\completion{f}\circ i=f$.
\item If $g:S\to\nonnegativereals$ is a regular upper functional, then there exists a unique regular upper functional $\completion{g}:\completion{S}\to\nonnegativereals$ such that $\completion{g}\circ i=g$.
\end{enumerate}
\end{proposition}
\begin{proof}
Let $a\in\completion{S}$, and let $x=(x_n)_{n\in\naturals}$ be an approximately geometric sequence in $S$ satisfying $G(a)\asymptoticeq i(x)$. This means that there exists a sublinear sequence $(k_n)_{n\in\naturals}$ such that $a^n\le i(2)^{k_n}i(x_n)$ and $i(x_n)\le i(2)^{k_n}a^n$ for all $n\in\naturals$. If some map $\completion{f}$ satisfies the stated properties, then
\begin{align}
\completion{f}(a)^n
 & = \completion{f}(a^n)
 \le \completion{f}(i(2^{k_n}x_n))
 = 2^{k_n}f(x_n)
\intertext{and}
f(x_n)
 & = \completion{f}(i(x_n))
 \le \completion{f}(i(2)^{k_n}a^n)
 = 2^{k_n}\completion{f}(a)^n,
\end{align}
therefore $2^{-\frac{k_n}{n}}\sqrt[n]{f(x_n)}\le\completion{f}(a)\le 2^{\frac{k_n}{n}}\sqrt[n]{f(x_n)}$ holds for all $n$. Since $(k_n)_{n\in\naturals}$ is sublinear, this implies
\begin{equation}
\completion{f}(a)=\lim_{n\to\infty}\sqrt[n]{f(x_n)},
\end{equation}
and similarly
\begin{equation}
\completion{g}(a)=\lim_{n\to\infty}\sqrt[n]{g(x_n)}.
\end{equation}
This shows that at most one such $\completion{f}$ and $\completion{g}$ exist.

We prove that the above limits exist. For $\epsilon>0$, let $m\in\positiveintegers$ and $c\in\naturals$ such that for all $n\in\naturals$ the inequalities $x_n\le 2^{c+\lfloor\epsilon n\rfloor}x_m^{\left\lfloor\frac{n}{m}\right\rfloor}$ and $x_m^{\left\lfloor\frac{n}{m}\right\rfloor}\le 2^{c+\lfloor\epsilon n\rfloor}x_n$ hold. Applying $f$, taking $n$th roots and $n\to\infty$ gives
\begin{equation}
2^{-\epsilon}\sqrt[m]{f(x_m)}\le\liminf_{n\to\infty}\sqrt[n]{f(x_n)}\le\limsup_{n\to\infty}\sqrt[n]{f(x_n)}\le 2^{\epsilon}\sqrt[m]{f(x_m)}.
\end{equation}
As $\epsilon>0$ is arbitrary, this implies that the limit exists. The proof for $g$ is similar.

If $x'$ is an approximately geometric sequence such that $i(x')\asymptoticeq G(a)$ as well, then $i(x)\asymptoticeq i(x')$, which implies $\lim_{n\to\infty}\sqrt[n]{f(x'_n)}=\lim_{n\to\infty}\sqrt[n]{f(x_n)}$ and $\lim_{n\to\infty}\sqrt[n]{g(x'_n)}=\lim_{n\to\infty}\sqrt[n]{g(x_n)}$, therefore $\completion{f}$ and $\completion{g}$ can be unambiguously defined by the above limits.

If $x$ is geometric, then the limits evaluate to $f(x_1)$ and $g(x_1)$, therefore $\completion{f}\circ i=f$ and $\completion{g}\circ i=g$. This also implies that $f$ and $g$ evaluate to $0$ on the zero element and to $1$ on the unit.

Let $a,b\in\completion{S}$, and let $x,y$ be approximately geometric sequences in $S$ such that $G(a)\asymptoticeq i(x)$ and $G(b)\asymptoticeq i(y)$. Then $G(a+b)=G(a)\sequenceplus G(b)\asymptoticeq i(x)\sequenceplus i(y)=i(x\sequenceplus y)$, therefore
\begin{equation}
\begin{split}
\completion{f}(a+b)
 & = \lim_{n\to\infty}\sqrt[n]{f\left(\textstyle\sum_{m=0}^n\binom{m}{n}x_{n-m}y_m\right)}  \\
 & \ge \lim_{n\to\infty}\sqrt[n]{\textstyle\sum_{m=0}^n\binom{m}{n}f(x_{n-m})f(y_m)}  \\
 & = \lim_{n\to\infty}\sqrt[n]{f(x_n)}+\lim_{n\to\infty}\sqrt[n]{f(y_n)}  \\
 & = \completion{f}(a)+\completion{f}(b)
\end{split}
\end{equation}
and
\begin{equation}
\begin{split}
\completion{g}(a+b)
 & = \lim_{n\to\infty}\sqrt[n]{g\left(\textstyle\sum_{m=0}^n\binom{m}{n}x_{n-m}y_m\right)}  \\
 & \le \lim_{n\to\infty}\sqrt[n]{\textstyle\sum_{m=0}^n\binom{m}{n}g(x_{n-m})g(y_m)}  \\
 & = \lim_{n\to\infty}\sqrt[n]{g(x_n)}+\lim_{n\to\infty}\sqrt[n]{g(y_n)}  \\
 & = \completion{g}(a)+\completion{g}(b).
\end{split}
\end{equation}

We also have $G(ab)=G(a)\sequencetimes G(b)\asymptoticeq i(x)\sequencetimes i(y)=i(x\sequencetimes y)$, therefore
\begin{equation}
\begin{split}
\completion{f}(a+b)
 & = \lim_{n\to\infty}\sqrt[n]{f(x_ny_n)}  \\
 & \ge \lim_{n\to\infty}\sqrt[n]{f(x_n)f(y_n)}  \\
 & = \lim_{n\to\infty}\sqrt[n]{f(x_n)}\lim_{n\to\infty}\sqrt[n]{f(y_n)}  \\
 & = \completion{f}(a)\completion{f}(b)
\end{split}
\end{equation}
and
\begin{equation}
\begin{split}
\completion{g}(a+b)
 & = \lim_{n\to\infty}\sqrt[n]{g(x_ny_n)}  \\
 & \le \lim_{n\to\infty}\sqrt[n]{g(x_n)f(y_n)}  \\
 & = \lim_{n\to\infty}\sqrt[n]{g(x_n)}\lim_{n\to\infty}\sqrt[n]{g(y_n)}  \\
 & = \completion{g}(a)\completion{g}(b).
\end{split}
\end{equation}

If $a\le b$, then $x\asymptoticle y$, therefore $x_n\le 2^{k_n}y_n$ for some sublinear $(k_n)_{n\in\naturals}$ and all $n$, which implies
\begin{equation}
\completion{f}(a)
 = \lim_{n\to\infty}\sqrt[n]{f(x_n)}
 \le \lim_{n\to\infty}\sqrt[n]{f(2^{k_n}y_n)}
 = \lim_{n\to\infty}2^{\frac{k_n}{n}}\sqrt[n]{f(y_n)}
 = \completion{f}(b)
\end{equation}
and similarly $\completion{g}(a)\le\completion{g}(b)$.

We show that $\completion{f}$ and $\completion{g}$ are regular. Let $p$ be a polynomial with nonnegative integer coefficients $p_0,p_1,\dots,p_d$. Let $a\in\completion{S}$, and let $x=(x_n)_{n\in\naturals}$ be an approximately geometric sequence in $S$ satisfying $G(a)\asymptoticeq i(x)$. For $\epsilon>0$, let $m\in\positiveintegers$ and $c\in\naturals$ such that for all $n\in\naturals$ the inequalities $x_n\le 2^{c+\lfloor\epsilon n\rfloor}x_m^{\left\lfloor\frac{n}{m}\right\rfloor}$ and $x_m^{\left\lfloor\frac{n}{m}\right\rfloor}\le 2^{c+\lfloor\epsilon n\rfloor}x_n$ hold.

The sequence $G(p(a))$ is asymptotically equivalent to the sequence
\begin{equation}
G(p_0)\sequenceplus \left(G(p_1)\sequencetimes x\right)\sequenceplus \left(G(p_2)\sequencetimes x^{\sequencetimes 2}\right)\sequenceplus\cdots\sequenceplus \left(G(p_d)\sequencetimes x^{\sequencetimes d}\right),
\end{equation}
the $n$th element of which is
\begin{equation}\label{eq:polynomialpowerelements}
\sum_{\substack{n_0,\dots,n_d\in\naturals  \\  n_0+\dots+n_d=n}}\binom{n}{n_0,\dots,n_d}p_0^{n_0}p_1^{n_1}x_{n_1}p_2^{n_2}x_{n_2}^2\cdots p_d^{n_d}x_{n_d}^d
\end{equation}
The factors coming from the sequence $x$ can be bounded as
\begin{equation}
\prod_{i=0}^dx_{n_i}^i
 \le \prod_{i=0}^d 2^{c+\lfloor\epsilon n_i\rfloor}x_m^{i\left\lfloor\frac{n}{m}\right\rfloor}
 \le 2^{(d+1)c+\lfloor\epsilon n\rfloor}x_m^{\left\lfloor\sum_{i=0}^d \frac{in_i}{m}\right\rfloor}
\end{equation}
and
\begin{equation}
x_m^{\left\lfloor\sum_{i=0}^d \frac{in_i}{m}\right\rfloor}
 \le 2^{(d+1)c+\lfloor\epsilon n\rfloor}\prod_{i=0}^dx_{n_i}^i
\end{equation}
The upper and lower bounds on the sum \eqref{eq:polynomialpowerelements} are polynomials in $x_m$ with nonnegative integer coefficients therefore, by the regularity of $f$, they evaluate on the sum to
\begin{equation}
\sum_{\substack{n_0,\dots,n_d\in\naturals  \\  n_0+\dots+n_d=n}}\binom{n}{n_0,\dots,n_d}p_0^{n_0}p_1^{n_1}p_2^{n_2}\cdots p_d^{n_d}f(x_m)^{\left\lfloor\sum_{i=0}^d \frac{in_i}{m}\right\rfloor}
\end{equation}
up to factors of $2^{(d+1)c+\lfloor\epsilon n\rfloor}$, which results in $2^{\pm\epsilon}$ after taking $n$th roots and limits. The limit of the $n$th root of the sum is $p(\sqrt[m]{f(x_m)})$, which can be seen using that it has polynomially many terms, therefore the largest exponential dominates. Since $\epsilon$ can be made arbitrarily small by choosing $m$ large enough, the equality $\completion{f}(p(a))=p(\completion{f}(a))$ follows. Regularity of $\completion{g}$ is shown in the same way.
\end{proof}

\section{Applications}\label{sec:applications}

\subsection{Powers with real exponents}\label{sec:powers}

\begin{lemma}\label{lem:approximatelygeometrictermdistance}
Let $a$ be an approximately geometric sequence, $(k_n)$ sublinear as in the definition, and let $p\in\naturals$ be such that $a_1\le u^p$ and $1\le u^pa_1$. Then for all $m,n\in\naturals$ the inequality
\begin{equation}
a_m\le u^{\lvert n-m\rvert(p+k_1)+k_{\min\{m,n\}}} a_n
\end{equation}
holds.
\end{lemma}
\begin{proof}
If $m\le n$, then
\begin{equation}
\begin{split}
a_m
 & \le a_m (u^pa_1)^{n-m}  \\
 & \le u^{(n-m)p+k_m+(n-m)k_1}a_n  \\
 & = u^{\lvert n-m\rvert(p+k_1)+k_{\min\{m,n\}}} a_n,
\end{split}
\end{equation}
while if $n\le m$, then
\begin{equation}
\begin{split}
a_m
 & \le u^{k_n+(m-n)k_1}a_n(a_1)^{m-n}  \\
 & \le u^{(m-n)p+k_n+(m-n)k_1}a_n  \\
 & = u^{\lvert n-m\rvert(p+k_1)+k_{\min\{m,n\}}} a_n.
\end{split}
\end{equation}
\end{proof}

\begin{proposition}\label{prop:rescaledproperties}
Let $a,b$ be approximately geometric sequences such that $0\neq a_1\sim 1$ and  $0\neq b_1\sim 1$. Let $\beta,\gamma:\naturals\to\naturals$ be sequences such that the limits $\lim_{n\to\infty}\frac{\beta_n}{n}$ and $\lim_{n\to\infty}\frac{\gamma_n}{n}$ exist (in $\nonnegativereals$). Then
\begin{enumerate}
\item\label{it:rescaledapproximatelygeometric} $a\circ\beta$ is approximately geometric
\item\label{it:rescaledequivalence} if $\lim_{n\to\infty}\frac{\beta_n}{n}=\lim_{n\to\infty}\frac{\gamma_n}{n}$, then $a\circ\beta\asymptoticeq a\circ\gamma$
\item\label{it:rescaledadditive} $a\circ(\beta+\gamma)\asymptoticeq(a\circ\beta)\sequencetimes(a\circ\gamma)$
\item\label{it:rescaledmultiplicative} $(a\sequencetimes b)\circ\beta=(a\circ\beta)\sequencetimes(b\circ\beta)$
\item\label{it:rescaledmonotone} if $a\asymptoticle b$, then $a\circ\gamma\asymptoticle b\circ\gamma$
\item\label{it:rescaledpower} if $\lim_{n\to\infty}\frac{\beta_n}{n}=t\in\naturals$, then $a\circ\beta\asymptoticeq a^d$.
\end{enumerate}
\end{proposition}
\begin{proof}
Let $p\in\naturals$ such that $a_1\le u^p$, $1\le u^pa_1$, $b_1\le u^p$, and $1\le u^pb_1$. Let $(k_n)_{n\in\naturals}$ be a nondecreasing sublinear sequence such that for all $t\in\positiveintegers$, $n_1,\dots,n_t\in\naturals$ the inequalities $a_{n_1}\cdots a_{n_t}\le u^{k_{n_1}+\dots+k_{n_t}}a_{n_1+\dots+n_t}$, $a_{n_1}\cdots a_{n_t}\le u^{k_{n_1}+\dots+k_{n_t}}a_{n_1+\dots+n_t}$, $b_{n_1}\cdots b_{n_t}\le u^{k_{n_1}+\dots+k_{n_t}}b_{n_1+\dots+n_t}$ and $b_{n_1}\cdots b_{n_t}\le u^{k_{n_1}+\dots+k_{n_t}}b_{n_1+\dots+n_t}$ hold.

\ref{it:rescaledapproximatelygeometric}:
Let $m\in\positiveintegers$. For all $q,r\in\naturals$ we have
\begin{equation}
\begin{split}
(a\circ\beta)_m^q(a\circ\beta)_r
 & = a_{\beta_m}^qa_{\beta_r}  \\
 & \le u^{qk_{\beta_m}+k_{\beta_r}}a_{q\beta_m+\beta_r}  \\
 & \le u^{qk_{\beta_m}+k_{\beta_r}+\lvert q\beta_m+\beta_r-\beta_{qm+r}\rvert(p+k_1)+k_{\min\{q\beta_m+\beta_r,\beta_{qm+r}\}}}a_{\beta_{qm+r}}  \\
 & \le u^{qk_{\beta_m}+k_{\beta_r}+\lvert q\beta_m+\beta_r-\beta_{qm+r}\rvert(p+k_1)+k_{\beta_{qm+r}}}(a\circ\beta)_{qm+r},
\end{split}
\end{equation}
in the second inequality using \cref{lem:approximatelygeometrictermdistance}, and similarly
\begin{equation}
(a\circ\beta)_{qm+r}
 \le u^{qk_{\beta_m}+k_{\beta_r}+\lvert q\beta_m+\beta_r-\beta_{qm+r}\rvert(p+k_1)+k_{\beta_{qm+r}}}(a\circ\beta)_m^q(a\circ\beta)_r.
\end{equation}
With $n=qm+r$, $0\le r<m$, the growth rate of the exponent satisfies
\begin{multline}
\lim_{n\to\infty}\frac{qk_{\beta_m}+k_{\beta_r}+\lvert q\beta_m+\beta_r-\beta_{qm+r}\rvert(p+k_1)+k_{\beta_{qm+r}}}{n}  \\
 = \frac{k_{\beta_m}}{m}+\left\lvert\frac{\beta_m}{m}-\lim_{n\to\infty}\frac{\beta_n}{n}\right\rvert(p+k_1)
\end{multline}
since $k\circ\beta$ is sublinear. This can be made arbitrarily small by choosing $m$ sufficiently large, therefore $a\circ\beta$ is approximately geometric.

\ref{it:rescaledequivalence}:
By \cref{lem:approximatelygeometrictermdistance} we have
\begin{equation}
\begin{split}
(a\circ\beta)_n
 & = a_{\beta_n}  \\
 & \le u^{\lvert\beta_n-\gamma_n\rvert(p+k_1)+k_{\min\{\beta_n,\gamma_n\}}}a_{\gamma_n}  \\
 & \le u^{\lvert\beta_n-\gamma_n\rvert(p+k_1)+k_{\beta_n}}(a\circ\gamma)_n
\end{split}
\end{equation}
Since
\begin{equation}
\lim_{n\to\infty}\frac{\lvert\beta_n-\gamma_n\rvert(p+k_1)+k_{\beta_n}}{n}
 = \left\lvert\lim_{n\to\infty}\frac{\beta_n}{n}-\lim_{n\to\infty}\frac{\gamma_n}{n}\right\rvert(p+k_1)
 = 0 
\end{equation}
by sublinearity of $k\circ\beta$, we have $a\circ\beta\asymptoticle a\circ\gamma$. The opposite inequality is obtained by reversing the roles of $\beta$ and $\gamma$.

\ref{it:rescaledadditive}:
Using that $a$ is approximately geometric, we obtain
\begin{equation}
\begin{split}
(a\circ(\beta+\gamma))_n
 = a_{\beta_n+\gamma_n}
 \le u^{k_{\beta_n}+k_{\gamma_n}}a_{\beta_n}a_{\gamma_n}
 = u^{k_{\beta_n}+k_{\gamma_n}}((a\circ\beta)\sequencetimes(a\circ\gamma))_n
\end{split}
\end{equation}
and in a similar way $((a\circ\beta)\sequencetimes(a\circ\gamma))_n\le u^{k_{\beta_n}+k_{\gamma_n}}(a\circ(\beta+\gamma))_n$ for all $n\in\naturals$. Since $k\circ\beta$ and $k\circ\gamma$ are sublinear, these inequalities imply $a\circ(\beta+\gamma)\asymptoticeq(a\circ\beta)\sequencetimes(a\circ\gamma)$.

\ref{it:rescaledmultiplicative}:
For all $n\in\naturals$ we have
\begin{equation}
((a\sequencetimes b)\circ\beta)_n
 = a_{\beta_n}b_{\beta_n}
 = (a\circ\beta)_n(b\circ\beta)_n
 = ((a\circ\beta)\sequencetimes(b\circ\beta))_n,
\end{equation}
therefore $(a\sequencetimes b)\circ\beta=(a\circ\beta)\sequencetimes(b\circ\beta)$.

\ref{it:rescaledmonotone}:
Let $k_n$ be a sublinear sequence such that $a_n\le u^{k_n}b_n$ holds for all $n$. Then
\begin{equation}
(a\circ\gamma)_n
 = a_{\gamma_n}
 \le u^{k_{\gamma_n}}b_{\gamma_n}
 = u^{k_{\gamma_n}}(b\circ\gamma)_n
\end{equation}
for all $n$. Since $k\circ\gamma$ is sublinear, this implies $a\circ\gamma\asymptoticle b\circ\gamma$.

\ref{it:rescaledpower}:
By part \ref{it:rescaledequivalence} we may assume that $\beta_n=dn$ for all $n$. Then
\begin{equation}
(a^d)_n
 = a_n^d
 \le u^{dk_n}a_{dn}
 = u^{dk_n}(a\circ\beta)_n,
\end{equation}
and similarly $(a\circ\beta)_n\le u^{dk_n}(a^d)_n$ for all $n\in\naturals$, therefore $a\circ\beta\asymptoticeq a^d$.
\end{proof}

\begin{definition}
Let $S$ be a closed partially ordered asymptotically complete semiring and $a\in \nonzeros{S}$, $a\sim 1$. For $r\in\nonnegativereals$ we define $a^r$ to be the unique element of $S$ such that $G(a^r)\asymptoticeq (a^{\lfloor r n\rfloor})_{n\in\naturals}$.

If in addition $a$ is invertible, then we define $a^{-r}$ to be $(\frac{1}{a})^r$ for $r>0$.
\end{definition}
By part \ref{it:rescaledpower} of \cref{prop:rescaledproperties}, the notation is consistent with the notation for iterated multiplication when $r\in\naturals$.
\begin{corollary}\label{cor:powerproperties}
If $a$ and $b$ are nonzero elements of a closed and asymptotically complete partially ordered semiring, $a\sim b\sim 1$, and $r,s\in\nonnegativereals$, then
\begin{enumerate}
\item $a^{r+s}=a^ra^s$
\item $(a^r)^s=a^{rs}$
\item $(ab)^r=a^rb^r$
\item if $a\le b$, then $a^r\le b^r$
\item if $1\le a$, then the map $r\mapsto a^r$ is monotone increasing
\item if $a\le 1$, then the map $r\mapsto a^r$ is monotone decreasing
\end{enumerate}
If $a$ and $b$ are invertible, then these properties hold for all $r,s\in\reals$, except that the inequality is reversed if $r<0$.
\end{corollary}
\begin{proof}
For nonnegative exponents, the properties directly follow from \cref{prop:rescaledproperties} and the definition.

If $a$ is invertible, then $a^{\lfloor rn\rfloor}(\frac{1}{a})^{\lfloor rn\rfloor}=1$ for all $n$, therefore $G(a^r)G(a^{-r})\asymptoticeq G(1)$, which implies that $a^r$ is invertible and its inverse is $a^{-r}$. Using this observation, the properties can be extended to all $r,s\in\reals$.
\end{proof}

\begin{example}
In the closed and asymptotically complete partially ordered semiring $\nonnegativereals$, $a^r$ agrees with the exponentiation for all $a\neq 0$ and $r\in\reals$.

Let $S$ be a semiring with a Strassen preorder, and $i:S\to\completion{S}$ its asymptotic completion. Then $\setbuild{i(2)^r}{r\in\nonnegativereals}\subseteq\completion{S}$ can be identified with the subset $[1,\infty)\subseteq\nonnegativereals$. If some natural number $n$ is invertible in $S$, then $\setbuild{i(n)^r}{r\in\reals}\subseteq\completion{S}$ can be identified with $(0,\infty)\subseteq\nonnegativereals$.
\end{example}

\begin{remark}\label{rem:asymptoticallycompletecriterion}
If $S$ and $T$ are closed and asymptotically complete partially ordered semirings, and $\phi:S\to T$ is a monotone homomorphism, then $\phi(a^r)=\phi(a)^r$ holds for all $r\in\nonnegativereals$, because left and right compositions commute. This gives a simple necessary condition for a preordered semiring to be asymptotically complete: if $\phi:S\to\nonnegativereals$ is a monotone homomorphism, and there is an element $a\in S$ such that $\phi(a)>1$ but the image of $\phi$ does not contain $(1,\infty)$, then $S$ is not asymptotically complete.
\end{remark}

If $S$ is a semiring with a Strassen preorder and $s\in S$, then the (abstract) rank of $s$ is $\rank(s)=\min\setbuild{m\in\naturals\subseteq S}{s\le m}$, and the (abstract) subrank of $s$ is $\subrank(s)=\max\setbuild{m\in\naturals}{m\le s}$. The asymptotic rank and the asymptotic subrank are defined as the limits $\asymptoticrank(s)=\lim_{n\to\infty}\sqrt[n]{\rank(s^n)}$, $\asymptoticsubrank(s)=\lim_{n\to\infty}\sqrt[n]{\subrank(s^n)}$, which exist by the Fekete lemma. In terms of the asymptotic completion, they can be characterized as follows:
\begin{proposition}
Let $S$ be a semiring with a Strassen preorder, $i:S\to\completion{S}$ its asymptotic completion, and $s\in S$. Then
\begin{align}
\asymptoticrank(s)
 & = \min\setbuild{m\in\{0\}\cup[1,\infty)=\completion{\naturals}\subseteq\completion{S}}{i(s)\le m}
\intertext{and if $2\le s^p$ for some $p\in\naturals$, then}
\asymptoticsubrank(s)
 & = \max\setbuild{m\in\{0\}\cup[1,\infty)=\completion{\naturals}\subseteq\completion{S}}{m\le i(s)}.
\end{align}
\end{proposition}
\begin{proof}
The sequences $x_n=\subrank(s^n)$ and $y_n=\rank(s^n)$ are approximately geometric in $\naturals\subseteq S$, and satisfy $x_n\le s^n\le y_n$ by definition. The corresponding elements in $\completion{S}$ are $\asymptoticsubrank(s)$ and $\asymptoticrank(s)$, therefore $\asymptoticsubrank(s)\le i(s)\le\asymptoticrank(s)$.

If $x'$ and $y'$ are approximately geometric sequences in $\naturals$ such that $x'\asymptoticle G(s)$ and $G(s)\asymptoticle y'$, then there exists a sublinear sequence $(k_n)_{n\in\naturals}$ such that $x'_n\le 2^{k_n}s^n$ and $s^n\le 2^{k_n}y'_n$ for all $n$. The second inequality implies $y_n\le 2^{k_n}y'_n$, therefore $y\asymptoticle y'$.

In the case of the subrank, if $2\le s^p$ holds for some $p$, then $x'_n\le s^{n+k_np}$, therefore $x'_n\le\subrank(s^{n+k_np})$, which implies
\begin{equation}
\lim_{n\to\infty}\sqrt[n]{x'_n}
 \le \lim_{n\to\infty}\sqrt[n]{\subrank(s^{n+k_np})}
 = \lim_{n\to\infty}\left(\sqrt[n+k_np]{\subrank(s^{n+k_np})}\right)^{\frac{n+k_np}{n}}
 = \asymptoticsubrank(s),
\end{equation}
therefore $x'\asymptoticle x$.
\end{proof}

We conclude this section with two statements involving ``vanishingly small'' powers of $u$. The first one says that inequality with respect to a closed partial order is equivalent to the infinite family of formally weaker inequalities obtained by multiplying the larger element with arbitrarily small powers of $u$. The second one is a more direct formulation of the intuition that approximately geometric sequences are precisely the ones that can be sandwiched between geometric ones generated by roots of its elements, up to an arbitrarily small factor.
\begin{proposition}
Let $S$ be an asymptotically complete and closed partially ordered semiring, and let $a,b\in S$. Then $a\le b$ if and only if for all $\epsilon>0$ the inequality $a\le u^\epsilon b$ holds.
\end{proposition}
\begin{proof}
Since $1\le u$, we have $1\le u^\epsilon$ as well for all $\epsilon>0$ by \cref{cor:powerproperties}. Therefore $a\le b$ implies $a\le u^\epsilon b$.

Suppose that $a\le u^\epsilon b$ for some $\epsilon>0$. Then there is a $c\in\naturals$ such that $a^n\le u^{c+2\lfloor \epsilon n\rfloor}b^n$ holds for all $n$. If this is true for all $\epsilon>0$, then $a\asymptoticle b$. Since the partial order is assumed to be closed, this implies $a\le b$.
\end{proof}

\begin{proposition}
Let $S$ be an asymptotically complete and closed partially ordered semiring, and let $a=(a_n)_{n\in\naturals}\in(\nonzeros{S})^\naturals$. Suppose that $a_n\sim 1$ holds for all $n$. Then $a$ is approximately geometric if and only if for all $\epsilon>0$ there is an $m\in\positiveintegers$ such that $a\asymptoticle G(u^\epsilon a_m^{1/m})$ and $G(a_m^{1/m})\asymptoticle G(u^\epsilon) a$.
\end{proposition}
\begin{proof}
Recall that $G(a_m^{1/m})\asymptoticeq (a_m^{\left\lfloor\frac{n}{m}\right\rfloor})_{n\in\naturals}$ and $G(u^\epsilon)\asymptoticeq (u^{\lfloor\epsilon n\rfloor})_{n\in\naturals}$.

Suppose that $a$ is approximately geometric. For $\epsilon>0$ arbitrary, let $m\in\positiveintegers$ and $c\in\naturals$ such that for all $n\in\naturals$ the inequalities $a_n\le u^{c+\lfloor\epsilon n\rfloor}a_m^{\left\lfloor\frac{n}{m}\right\rfloor}$ and $a_m^{\left\lfloor\frac{n}{m}\right\rfloor}\le u^{c+\lfloor\epsilon\rfloor}a_n$ hold. Since the constant sequence $k_n=c$ is sublinear, these inequalities imply
\begin{gather}
a
 \asymptoticle \left(u^{\lfloor\epsilon n\rfloor}a_m^{\left\lfloor\frac{n}{m}\right\rfloor}\right)_{n\in\naturals}
 \asymptoticeq G(u^\epsilon a_m^{1/m})
\intertext{and}
G(a_m^{1/m})
 \asymptoticeq \left(a_m^{\left\lfloor\frac{n}{m}\right\rfloor}\right)_{n\in\naturals}
 \asymptoticle \left(u^{\lfloor\epsilon n\rfloor}a_n\right)_{n\in\naturals}
 \asymptoticeq G(u^\epsilon)a.
\end{gather}

Assume now that for all $\epsilon>0$ there is an $m\in\positiveintegers$ such that the asymptotic inequalities $a\asymptoticle G(u^\epsilon a_m^{1/m})$ and $G(a_m^{1/m})\asymptoticle G(u^\epsilon) a$ hold. For such an $\epsilon$ and $m$, there exists a sublinear sequence $k_n$ such that for all $n\in\naturals$ we have
\begin{gather}
a_n
 \le u^{k_n+\lfloor\epsilon n\rfloor}a_m^{\left\lfloor\frac{n}{m}\right\rfloor}
\intertext{and}
a_m^{\left\lfloor\frac{n}{m}\right\rfloor}
 \le u^{k_n+\lfloor\epsilon n\rfloor}a_n.
\end{gather}
There exists $c\in\naturals$ such that $k_n\le c+\lfloor\epsilon n\rfloor$ holds for all $n$, which allows us to replace the exponent of $u$ with $c+2\lfloor\epsilon n\rfloor\le c+\lfloor 2\epsilon n\rfloor$. Since this is true for arbitrary $\epsilon>0$, the sequence is approximately geometric.
\end{proof}

\subsection{Linear recurrences}\label{sec:linearrecurrence}

Over the complex numbers, sequences satisfying a linear recurrence with constant coefficients are linear combinations of exponential polynomials, with long-term behaviour governed by the largest exponent. We consider analogous sequences in preordered semirings in their matrix form, i.e., $a_n=x^TA^ny$ where $A\in S^{d\times d}$ is a matrix, and $x,y\in S^d$ are column vectors. If $S=\nonnegativereals$, then the Perron--Frobenius theorem implies that the sequence grows like powers of the spectral radius of $A$ if $x$ and $y$ are not too special. 

In general, such a sequence is not necessarily approximately geometric. For instance,
\begin{equation}
a_n=\begin{bmatrix}
1 & 0
\end{bmatrix}
\begin{bmatrix}
1 & 1  \\
1 & 0
\end{bmatrix}
\begin{bmatrix}
1 \\
0
\end{bmatrix}
\end{equation}
provides the sequence of natural numbers $(1,1,2,3,5,8,13,\dots)$. Viewed as a sequence in $\naturals$, it is approximately geometric, but as a sequence in $\naturals^=$, it is not. Unfortunately, we do not have a characterization of triples $(A,x,y)$ such that $a_n=x^TA^ny$ is approximately geometric. On the other hand, we can present a sufficient condition that seems to be reasonably broad.
\begin{proposition}\label{prop:linearrecursionapproximatelygeometric}
Let $x,y\in S^d$, $A\in S^{d\times d}$, and $a_n=x^TA^ny$ for all $n\in\naturals$. Suppose that there exist $p,m\in\naturals$ such that the entrywise inequalities
\begin{align}
A^{2m} & \le u^pA^myx^TA^m  \\
A^myx^TA^m & \le u^pA^{2m}
\end{align}
and the inequalities $a_n\le u^pa_1^n$ and $a_1^n\le u^pa_n$ for all $n<2m$ hold. Then $a=(a_n)_{n\in\naturals}$ is approximately geometric.

If in addition $x',y'\in S^d$, $a'=(a'_n)_{n\in\naturals}$ where $a'_n={x'}^TA^ny'$, $A^myx^TA^m\le u^p A^my'{x'}^TA^m$ and $A^my'{x'}^TA^m\le u^p A^myx^TA^m$, and $a'_n\le u^pa_n$ and $a_n\le u^pa'_n$ for all $n<2m$, then $a\asymptoticeq a'$.
\end{proposition}
\begin{proof}
Let $t\in\positiveintegers$, $n_1,\dots,n_t\in\naturals$. If $n_1,n_2,\dots,n_t\ge m$, then we get
\begin{equation}\label{eq:allfactorslarge}
\begin{split}
a_{n_1}\cdots a_{n_t}
 & = x^TA^{n_1}yx^TA^{n_2}y\cdots x^TA^{n_t}y  \\
 & \le u^p x^TA^{n_1+n_2}y\cdots x^TA^{n_t}y  \\
 & \le u^{p(t-1)} x^TA^{n_1+n_2+\dots+n_t}y  \\
 & = u^{p(t-1)}a_{n_1+\dots+n_t}
\end{split}
\end{equation}
by repeatedly applying the first inequality, and we get $a_{n_1+\dots+n_t}\le u^{p(t-1)}a_{n_1}\cdots a_{n_t}$ by similar steps in the opposite direction.

Otherwise suppose that there are exactly $r$ factors $a_n$ with $n<m$, and the sum of their indices is $l$. Their product can be replaced with $a_1^l$ up to an additional factor of $u^{rp}$. If $l\ge m$, then
\begin{align}
a_1^l
 & \le u^{p\left\lfloor\frac{l}{m}\right\rfloor}a_m^{\left\lfloor\frac{l}{m}\right\rfloor-1}a_{l+m-m\left\lfloor\frac{l}{m}\right\rfloor}
\intertext{and}
a_m^{\left\lfloor\frac{l}{m}\right\rfloor-1}a_{l+m-m\left\lfloor\frac{l}{m}\right\rfloor}
 & \le u^{p\left\lfloor\frac{l}{m}\right\rfloor}a_1^l
\end{align}
At this point we have $t-r+\left\lfloor\frac{l}{m}\right\rfloor\le t$ factors with all indices at least $m$, so we can apply \eqref{eq:allfactorslarge}, for a total cost of $u^{rp+(t-1)p}\le u^{2(t-1)p}$.

If $l<m$ and there are no other factors (i.e., $a_n$ with $m\le n$), then we use $a_1^l\le u^pa_l$ and $a_l\le u^pa_1^l$.

If $l<m$ and there is a factor $a_n$ with $m\le n<2m$, then using $a_n\le u^pa_1^n$ and $a_1^n\le u^pa_n$ we can ensure that $a_1$ appears with exponent $n+l\ge m$, so the previous case applies.

Finally, if $l<m$ and the factor with the next smallest index is $a_n$ with $2m\le n$, then by \eqref{eq:allfactorslarge} we have $a_ma_{n-m}\le u^pa_n$ and $a_n\le u^pa_ma_{n-m}$, reducing to the previous case at an extra cost of $p$.

Taking all cases into account, we have $a_{n_1}\cdots a_{n_t}\le u^ha_{n_1+\dots+n_t}$ and $a_{n_1+\dots+n_t}\le u^h a_{n_1}\cdots a_{n_t}$ with
\begin{equation}
h
 \le p+p+rp+p\frac{n+l}{m}+p(t-1)
 \le p(2r+4+t-1)
 \le 3p(t+3)
 \le 12pt.
\end{equation}

By choosing $k_n=12p$ for all $n\in\naturals$, we can thus ensure
\begin{align}
a_{n_1}\cdots a_{n_t} & \le u^{k_{n_1}+\cdots+k_{n_t}}a_{n_1+\dots+n_t}  \\
a_{n_1+\dots+n_t} & \le u^{k_{n_1}+\cdots+k_{n_t}}a_{n_1}\cdots a_{n_t}.
\end{align}

Let $x',y'\in S^d$ and $a'_n={x'}^TA^ny'$ satisfy the assumptions. Then if $n\ge 2m$, we
\begin{align}
a_n
 & = \Tr A^nyx^T
 = \Tr A^myx^TA^mA^{n-2m}
 \le \Tr u^pA^my'{x'}^TA^mA^{n-2m}
 = u^pa'_n
\intertext{and}
a'_n
 & = \Tr A^ny'{x'}^T
 = \Tr A^my'{x'}^TA^mA^{n-2m}
 \le \Tr u^pA^myx^TA^mA^{n-2m}
 = u^pa_n,
\end{align}
while for smaller $n$ we have $a_n\le u^pa'_n$ and $a_n\le u^pa'_n$ by assumption, therefore $a\asymptoticeq a'$.
\end{proof}

In the following we will assume that $0\le 1$ in $S$. If $a,b\in \nonzeros{S}$, then multiplying the inequality by $a$ and $b$ gives $0\le a$ and $0\le b$, which in turn imply $b\le a+b$ and $a\le a+b$. If $u$ is a power universal element, then $a+b\le u^pa$ and $a+b\le u^pb$ for some $p$, therefore $a\le u^pb$ and $b\le u^pa$. In this case, the conditions in \cref{prop:linearrecursionapproximatelygeometric} reduce to combinatorial ones that depend only on the zero patterns of $A$, $x$, and $y$. In particular, let us say that $A\in S^{d\times d}$ is \emph{primitive} if $A^m\in\nonzeros{S}^{d\times d}$ for some $m$. If $A$ is primitive, then $a_n=x^TA^ny$ is approximately geometric for all $x,y\in \nonzeros{S}^d$, and the sequences obtained in this way from the same $A$ are asymptotically equivalent.
\begin{definition}
Let $S$ be a preordered semiring of polynomial growth with $0\le 1$. If $A\in S^{d\times d}$ is primitive, then we denote by $\rho(A)$ the element of $\completion{S}$ that satisfies $G(\rho(A))\asymptoticeq i(a)$ for any sequence $a=(a_n)_{n\in\naturals}$ defined by $a_n=x^TA^ny$, with $x,y\in \nonzeros{S}^d$.
\end{definition}

\begin{example}\label{ex:spectralradius}\leavevmode
\begin{enumerate}
\item If $S=\nonnegativereals$, then we can choose the Perron--Frobenius eigenvector of $A$ for $y$ and $x=y/\norm{y}^2$. If $\lambda$ denotes the corresponding eigenvalue, then $a_n=x^TA^ny=x^T\lambda^ny=\lambda^n=G(\lambda)_n$, therefore $\rho(A)$ is equal to the spectral radius of $A$ (after identifying $\completion{S}$ with $\nonnegativereals$ as well).
\item If $S=\tropicalreals$ then, considering $A$ as a weighted graph, the entries of $A^n$ are the maximal products along $n$-step walks with start and endpoints given by the row and column index. It follows that $\rho(A)$ is the maximum cycle geometric mean.
\item If $T$ is a preordered semiring of polynomial growth and $\phi:S\to T$ is a monotone homomorphism (note that this requires $0\le 1$ in $T$ as well), then we can extend $\phi$ to matrices and vectors by applying it entrywise. Clearly, this is compatible with the matrix product, and preserves the zero patterns. In particular, if $A\in S^{d\times d}$ is primitive, then $\phi(A)$ is primitive as well, and $\rho(\phi(A))=\completion{\phi}(\rho(A))$. This allows us to compute $\completion{\phi}(\rho(A))$ when $T\in\{\nonnegativereals,\tropicalreals\}$ by applying $\phi$ entrywise and finding the dominant eigenvalue or the maximum cycle geometric mean.
\end{enumerate}
\end{example}

\subsection{Tensors}\label{sec:tensors}

We continue with applications in concrete preordered semirings, starting with the semiring of tensors. Instead of the generic notations used in the abstract setting so far, we will adopt the notations for the operations that are already in use in the literature. In particular, the notation for powers with respect to the product operation will be indicated in exponent (such as $t^{\otimes 2}=t\otimes t$), and we modify the notation for the case of real exponents accordingly (with the exception of powers of natural numbers and their images in the semiring, which we identify with $\{0\}\cup[1,\infty)$). In addition, to simplify notation, we will omit the canonical map $i:S\to\completion{S}$, i.e., view elements of $S$ as elements of the asymptotic completion as well, even if $i$ is not injective. However, we will distinguish between the preorders by using $\le$ for $S$ and $\asymptoticle$ for $\completion{S}$.

A tensor (of order $3$) is an element of $V_1\otimes V_2\otimes V_3$ where $V_1,V_2,V_3$ are finite-dimensional complex vector spaces. A tensor $s\in V_1\otimes V_2\otimes V_3$ restricts to another tensor $t\in W_1\otimes W_2\otimes W_3$ if there exist linear maps $A_j:V_j\to W_j$ for $j=1,2,3$ such that $(A_1\otimes A_2\otimes A_3)s=t$. $s$ and $t$ are equivalent if $s$ restricts to $t$ and $t$ restricts to $s$. Note that every tensor is equivalent to one in $\complexes^{d_1}\otimes\complexes^{d_2}\otimes\complexes{d_3}$ for suitably large $d_1,d_2,d_3\in\naturals$.

Let $S$ be the set of equivalence classes of tensors in $\complexes^{d_1}\otimes\complexes^{d_2}\otimes\complexes{d_3}$ for all $d_1,d_2,d_3\in\naturals$. The direct sum and (Kronecker) tensor product of tensors induce operations $\oplus,\otimes$ on $S$ that, together with the restriction preorder $\le$, endow $S$ with the structure of a preordered semiring. \cite{strassen1988asymptotic}. The image of the natural numbers in this semiring are the (equivalence classes of) \emph{unit tensors}
\begin{equation}
\unittensor{r}=\sum_{i=1}^r e_i\otimes e_i\otimes e_i\in\complexes^r\otimes\complexes^r\otimes\complexes^r.
\end{equation}

The semiring of tensors is not asymptotically complete, which can be seen using the criterion in \cref{rem:asymptoticallycompletecriterion}. Every nonzero tensor restricts to $\unittensor{1}$, and the flattening ranks are $\naturals$-valued, therefore their images do not contain $(1,\infty)$. More generally, the image of every element of the asymptotic spectrum of tensors is well ordered \cite{christandl2025asymptotic}, therefore none of them contains $(1,\infty)$.

Tensors can be viewed either as trilinear maps $\complexes^{d_1}\otimes\complexes^{d_2}\otimes\complexes^{d_3}\to\complexes$ or bilinear maps $\complexes^{d_1}\otimes\complexes^{d_2}\to\complexes^{d_3}$ (in three different ways), and restrictions correspond to reductions between such maps, regarded as computational problems. In particular, the matrix multiplication tensors, parametrized by a triple $l,m,n\in\naturals$ are
\begin{equation}
\langle l,m,n\rangle=\sum_{i=1}^l\sum_{j=1}^m\sum_{k=1}^n e_{(i,j)}\otimes e_{(j,k)}\otimes e_{(k,i)}\in\complexes^{lm}\otimes\complexes^{mn}\otimes\complexes^{nl},
\end{equation}
with double-indexed basis vectors corresponding to matrix units. In the bilinear picture, $\langle l,m,n\rangle$ correspond to the multiplication of an $l\times m$ and an $m\times n$ matrix (up to permuting the parameters), while as a trilinear map it computes the trace of the product of three matrices.

The relevance of the semiring of tensors and its asymptotic preorder in the study of the complexity of matrix multiplication stems from two facts. First, the tensor rank (which is equal to the abstract rank in this semiring) is, up to a bounded factor, the arithmetic complexity of computing the corresponding bilinear map. Second, block matrices can be multiplied in the same way as if the blocks were scalars (if commutativity is not used), which can be concisely expressed as $\langle l_1,m_1,n_1\rangle\otimes\langle l_2,m_2,n_2\rangle=\langle l_1l_2,m_1m_2,n_1n_2\rangle$ in the semiring. These imply that $\omega:=\log\asymptoticrank(\langle 2,2,2\rangle)$ is the smallest exponent such that two $n\times n$ matrices can be multiplied using $O(n^{\omega+\epsilon})$ arithmetic operations for every $\epsilon>0$.

It is a very special property that the family of square matrix multiplications of different sizes form (when viewed on an exponential scale) a geometric sequence in the semiring of tensors, which is generally not the case for computational problems. Moreover, more general sequences can be useful for bounding sequences of tensor powers as well. In fact, a variant of the matrix multiplication exponent already refers to sequences that are not tensor powers. In the asymptotic study of \emph{rectangular} matrix multiplications, we consider the multiplication of a $n\times\lceil n^p\rceil$ and an $\lceil n^p\rceil\times n$ matrix for fixed $p$ as $n\to\infty$. As in the square case, it is sufficient to consider a subsequence where $n$ grows as a geometric sequence, but this time $\lceil n^p\rceil$ will not be exactly geometric unless $p$ is an integer. The analogue of $\omega$ for this problem is the exponent $\omega(p)$ such that the above rectangular matrix multiplication is possible using $O(n^{\omega(p)+\epsilon})$ arithmetic operations for every $\epsilon>0$. Since
\begin{equation}
\langle 2^n,2^n,2^{\lceil pn\rceil-1}\rangle\le\langle 2^n,2^n,\lceil (2^n)^p\rceil\rangle\le\langle 2^n,2^n,2^{\lceil pn\rceil}\rangle
\end{equation}
and, by the multiplicative property of matrix multiplication tensors,
\begin{equation}
\langle 2^n,2^n,2^{\lceil pn\rceil}\rangle
 = \langle 2,1,1\rangle^{\otimes n}\otimes\langle 1,2,1\rangle^{\otimes n}\otimes\langle 1,1,2\rangle^{\otimes \lceil pn\rceil},
\end{equation}
we can extend the notation as $\langle 2,2,2^p\rangle=\langle 2,1,1\rangle\otimes\langle 1,2,1\rangle\otimes\langle 1,1,2\rangle^{\otimes p}$ or, more generally, as $\langle x,y,z\rangle=\langle 2,1,1\rangle^{\otimes\log x}\otimes\langle 1,2,1\rangle^{\otimes\log y}\otimes\langle 1,1,2\rangle^{\otimes\log z}$ when $x,y,z\ge 1$, and write $\omega(p)=\log\asymptoticrank(\langle 2,2,2^p\rangle)$, in complete analogy with the square case.

To our knowledge, sequences of tensors on an exponential scale have been explicitly studied in two papers, in the context of bounds on the complexity of matrix multiplication in \cite{christandl2025barriers} (where $\langle 2,2,2^p\rangle$ is also introduced as a formal symbol called \emph{virtual matrix multiplication tensor}), and as computational problems in their own right in \cite{bjorklund2026kronecker}. In the following we compare their notions with our definition.

In \cite{christandl2025barriers} Christandl, Le Gall, Lysikov, and Zuiddam formulate a barrier result for a type of upper bound methods on the complexity of rectangular matrix multiplication, called asymptotic mixed methods. Such methods are defined in terms of sequences of tensors behaving well with respect to \emph{adequate} tensor parameters, i.e., functionals $F:S\to\nonnegativereals$ that are monotone, $\otimes$-submultiplicative, multiplicative on matrix multiplication tensors, satisfy $F(\unittensor{r}\otimes T)=r F(T)$, and $F\le\asymptoticrank$. For a tensor $T$ and adequate $F$ we also introduce $\undertilde{F}(T)=\lim_{n\to\infty}\sqrt[n]{F(T^{\otimes n})}$, where the limit exists by submultiplicativity, and $\undertilde{F}$ is adequate as well.
\begin{definition}
A sequence of tensors $T_0,T_1,T_2,T_3,\dots$ is \emph{almost exponential} if the sequence $\sqrt[n]{\asymptoticrank(T_n)}$ converges and $\sqrt[n]{F(T_n)}$ is bounded for each adequate $F$. For such a sequence $T=(T_n)_{n\in\naturals}$ they define $\asymptoticrank(T)=\lim_{n\to\infty}\sqrt[n]{\asymptoticrank(T_n)}$ and $F(T)=\limsup_{n\to\infty}\sqrt[n]{F(T_n)}$.
\end{definition}
\begin{proposition}\label{prop:almostexponential}
Let $T=(T_n)_{n\in\naturals}$ be an approximately geometric sequence in the semiring of tensors. Then the sequence is almost exponential. Moreover, the limit
\begin{equation}
\lim_{n\to\infty}\sqrt[n]{F(T_n)}
\end{equation}
exists for every adequate $F$, and is equal to $\lim_{n\to\infty}\sqrt[n]{\undertilde{F}(T_n)}$.
\end{proposition}
\begin{proof}
The asymptotic rank $\asymptoticrank$ is a regular upper functional, therefore $\completion{\asymptoticrank}(T)=\lim_{n\to\infty}\sqrt[n]{\asymptoticrank(T_n)}=\asymptoticrank(T)$ exists by \cref{prop:regularfunctionalextension}.

Let $\epsilon>0$ and $m\in\positiveintegers$, $c\in\naturals$ such that for all $n\in\naturals$ the inequalities $T_n\le 2^{c+\lfloor\epsilon n\rfloor}T_m^{\otimes\left\lfloor\frac{n}{m}\right\rfloor}$ and $T_m^{\otimes\left\lfloor\frac{n}{m}\right\rfloor}\le 2^{c+\lfloor\epsilon n\rfloor}T_n$ hold (see the discussion following \cref{prop:approxgeomcharacterizations}). Then
\begin{equation}
\begin{split}
\limsup_{n\to\infty}\sqrt[n]{F(T_n)}
 & \le \limsup_{n\to\infty}\sqrt[n]{F\left(2^{c+\lfloor\epsilon n\rfloor}T_m^{\otimes\left\lfloor\frac{n}{m}\right\rfloor}\right)}  \\
 & \le \limsup_{n\to\infty}\sqrt[n]{2^{c+\lfloor\epsilon n\rfloor}F\left(T_m^{\otimes\left\lfloor\frac{n}{m}\right\rfloor}\right)}  \\
 & = 2^\epsilon\sqrt[m]{\undertilde{F}(T_m)},
\end{split}
\end{equation}
therefore $\sqrt[n]{F(T_n)}$ is bounded. Similarly,
\begin{equation}
F\left(T_m^{\otimes\left\lfloor\frac{n}{m}\right\rfloor}\right)\le 2^{c+\lfloor\epsilon n\rfloor}F(T_n),
\end{equation}
therefore
\begin{equation}
2^{-\epsilon}\sqrt[m]{\undertilde{F}(T_m)}\le\liminf_{n\to\infty}\sqrt[n]{F(T_n)}
\end{equation}
Since $\epsilon$ can be chosen arbitrarily small, the limit exists, and is equal to the limit of $\sqrt[n]{\undertilde{F}(T_n)}$.
\end{proof}
We note that the multiplicativity of each adequate $F$ on matrix multiplication tensors extends to the elements $\langle x,y,z\rangle$ of $\completion{S}$.

As a concrete example, sequences of tensors of the form $T_n=S_1^{\otimes f_1(n)}\otimes S_2^{\otimes f_2(n)}$ where $f_1,f_2:\naturals\to\naturals$, $f_i(n)=a_in+o(n)$ are approximately geometric due to \cref{prop:rescaledproperties}, therefore they are almost exponential as well by \cref{prop:almostexponential}. This is proved in \cite{christandl2025barriers} using the spectral characterization of the asymptotic rank.

For an approximately geometric sequence $T=(T_0,T_1,\dots)$, a $\kappa$-catalytic asymptotic mixed method in the sense of \cite[Definition 3.19]{christandl2025barriers} can be phrased as the asymptotic inequality $\unittensor{2}^{\otimes\kappa}\otimes\langle 2,2,2^p\rangle\asymptoticle T$, which proves the upper bound $\omega(p)\le\hat{\omega}(p)=\log\asymptoticrank(T)-\kappa$.

In \cite{bjorklund2026kronecker} Björklund, Kaski, Koana, and Nederlof consider sequences of tensors that behave like Kronecker powers, and extend the notion of exponent of a tensor, as well as the connection of trilinear and algebraic complexity to such sequences. More precisely, they consider the following property.
\begin{definition}
A sequence of $3$-tensors $T_0,T_1,T_2,T_3,\dots$ has the \emph{Kronecker scaling} property if for all $\delta>0$ there exist infinitely many $d=1,2,\dots$ such that for all large enough $n=1,2,\dots$ in an arithmetic progression, the tensor $T_n$ is a sum of at most $2^{\delta n}$ tensors, each of which is a restriction of $T_d^{\otimes s}$ for $s\le(1+\delta)n/d$.
\end{definition}
\begin{proposition}
Let $(T_n)_{n\in\naturals}$ be an approximately geometric sequence in the semiring of tensors. Then the sequence has the Kronecker scaling property.
\end{proposition}
\begin{proof}
Let $\delta>0$. For $\epsilon=\frac{\delta}{4}$, choose $m\in\positiveintegers$ and $c\in\naturals$ such that for all $n\in\naturals$ the inequalities $T_n\le 2^{c+\lfloor\epsilon n\rfloor}T_m^{\otimes\left\lfloor\frac{n}{m}\right\rfloor}$ and $T_m^{\otimes\left\lfloor\frac{n}{m}\right\rfloor}\le 2^{c+\lfloor\epsilon n\rfloor}T_n$ hold. For $k\in\naturals$, $k\ge\frac{c}{\epsilon m}$ but otherwise arbitrary, let $d=km$. In particular, $T_m^{\otimes k}\le 2^{c+\lfloor\epsilon d\rfloor}T_d$. For all $n\in\naturals$ we have
\begin{equation}
\begin{split}
T_n
 & \le 2^{c+\lfloor\epsilon n\rfloor}T_m^{\otimes\left\lfloor\frac{n}{m}\right\rfloor}  \\
 & \le 2^{c+\lfloor\epsilon n\rfloor}T_m^{\otimes k\left\lceil\frac{\left\lfloor\frac{n}{m}\right\rfloor}{k}\right\rceil}  \\
 & \le 2^{c+\lfloor\epsilon n\rfloor+\left\lceil\frac{\left\lfloor\frac{n}{m}\right\rfloor}{k}\right\rceil(c+\lfloor\epsilon d\rfloor)}T_d^{\otimes\left\lceil\frac{\left\lfloor\frac{n}{m}\right\rfloor}{k}\right\rceil}.
\end{split}
\end{equation}
In this inequality, the exponent $s$ of $T_d$ satisfies
\begin{equation}
s
 = \left\lceil\frac{\left\lfloor\frac{n}{m}\right\rfloor}{k}\right\rceil
 \le \frac{n}{d}+1
 = \left(1+\frac{d}{n}\right)\frac{n}{d},
 \le (1+\delta)\frac{n}{d},
\end{equation}
for all large enough $n$. Since $d\ge\frac{c}{\epsilon}$, for the exponent of $2$ we have
\begin{equation}
\begin{split}
c+\lfloor\epsilon n\rfloor+\left\lceil\frac{\left\lfloor\frac{n}{m}\right\rfloor}{k}\right\rceil(c+\lfloor\epsilon d\rfloor)
 & \le c+\epsilon n+\left(\frac{n}{d}+1\right)(c+\epsilon d)  \\
 & \le c+\epsilon n+\left(\frac{n}{d}+1\right)2\epsilon d  \\
 & = c+3\epsilon n+2\epsilon d  \\
 & = (c+2\epsilon d)+\frac{3}{4}\delta n,
\end{split}
\end{equation}
which is less than $\delta n$ for all large enough $n$.

This shows that for infinitely many $d$ all large $n$, $T_n$ is a sum of at most $2^{\delta n}$ tensors that are restrictions of $T_d^{\otimes s}$, therefore $T$ has the Kronecker scaling property.
\end{proof}

In \cite{bjorklund2026kronecker} the Kronecker scaling property is shown to hold for two concrete sequences of tensors, the \emph{balanced tripartitioning tensors}, and the \emph{matchings connectivity tensors}. We will now carefully apply and strengthen their Steinitz balancing technique to prove that a large class of sequences of tensors, which includes the balanced tripartitioning tensors, are approximately geometric.

In the following we consider block tensors that can be defined in terms of direct sum decompositions of the appearing vector spaces (essentially a coordinate-free version of the inner and outer structure \cite[Definition 8.1.]{blaser2013fast}). If $I$ is a finite set and
\begin{equation}
U=\bigoplus_{i\in I}U^{(i)}
\end{equation}
is a direct sum decomposition of the vector space $U$, then for $i\in I$ we let $E_i:U\to U$ denote the projection onto $U^{(i)}$ (followed by inclusion in $U$). If $J$ and $K$ are finite sets as well and
\begin{align}
V & = \bigoplus_{j\in J}V^{(j)}  \\
W & = \bigoplus_{k\in K}W^{(k)},
\end{align}
then a tensor $t\in U\otimes V\otimes W$ can be decomposed as
\begin{equation}
t=\sum_{\substack{i\in I \\ j\in J \\ k\in K}}t_{i,j,k}
\end{equation}
where $t_{i,j,k}=(E_i\otimes E_j\otimes E_k)t$. With respect to such a decomposition, we will say that $t$ is a \emph{block tensor} and refer to the tensors $t_{i,j,k}$ as the blocks of $t$. The \emph{support} of $t$ is the set $\support t=\setbuild{(i,j,k)\in I\times J\times K}{t_{i,j,k}\neq 0}$. For a subset $\Psi\subseteq I\times J\times K$ we let
\begin{equation}
t[\Psi]=\sum_{(i,j,k)\in\Psi}t_{i,j,k}.
\end{equation}
In particular, if $\Psi=X\times Y\times Z$ for subsets $X\subseteq I$, $Y\subseteq J$, $Z\subseteq K$, then
\begin{equation}
t[\Psi\cap\support t]
 = t[\Psi]
 = \left(\left(\sum_{i\in X}E_i\right)\otimes\left(\sum_{j\in Y}E_j\right)\otimes\left(\sum_{k\in Z}E_k\right)\right)t
 \le t.
\end{equation}

Let $t\in U\otimes V\otimes W$ be a block tensor with index sets $\mathcal{A}$, $\mathcal{B}$, $\mathcal{C}$ labelling the decompositions in the three factors. Kronecker powers of $t$ will also be considered as block tensors. The blocks of $t^{\otimes m}$ are labelled by $I=\mathcal{A}^m$, $J=\mathcal{B}^m$, and $K=\mathcal{C}^m$ where, e.g., the string $i=(i_1,i_2,\dots,i_m)\in \mathcal{A}^m$ corresponds to the subspace $U^{(i_1)}\otimes U^{(i_2)}\otimes\dots\otimes U^{(i_m)}\le U^{\otimes m}$.

For every $m\in\positiveintegers$ and $m$-type $P\in\distributions[m](\mathcal{A})$, we define the type class projection
\begin{equation}
\typeclassprojector{m}{P}=\sum_{(i_1,\dots,i_m)\in\typeclass{m}{P}}E_{i_1}\otimes E_{i_2}\otimes\cdots\otimes E_{i_m}\in\End(U^{\otimes m}),
\end{equation}
and similarly for elements of $\distributions[m](\mathcal{B})$ and $\distributions[m](\mathcal{C})$ we define projections acting on $V^{\otimes m}$ and $W^{\otimes m}$.

In addition, if $P\in\distributions[m](\mathcal{A}\times\mathcal{B}\times\mathcal{C})$ is a \emph{joint} $m$-type for some $m\in\positiveintegers$, then we will consider the joint type class $\typeclass{m}{P}\subseteq \mathcal{A}^m\times \mathcal{B}^m\times \mathcal{C}^m\simeq(\mathcal{A}\times\mathcal{B}\times\mathcal{C})^m$ consisting of strings $((i_1,j_1,k_1),(i_2,j_2,k_2),\dots,(i_n,j_n,k_n))$ in which each triple $(i,j,k)$ occurs $mP(i,j,k)$ times. The marginals of $P$ on the three factors will be denoted by $P_A$, $P_B$, and $P_C$. The type classes satisfy $\typeclass{m}{P}\subseteq\typeclass{m}{P_A}\times\typeclass{m}{P_B}\times\typeclass{m}{P_C}$.

We will use the following lemma.
\begin{lemma}[Steinitz concentration, {\cite[Lemma 3.1]{bjorklund2026kronecker}}]\label{lem:Steinitz}
Let $\norm{\cdot}$ be a norm on $\reals^d$ and $v_1,v_2,\ldots,v_r\in\reals^d$ with $\norm{v_i}\le 1$ for all $i=1,2,\dots,r$. Let $g_1,\dots,g_s\in\positiveintegers$, $g_1+\dots+g_s=r$. Then there exists a set partition $G_1\cup G_2\cup\dots\cup G_s=[r]$ such that for all $j=1,2,\dots,s$ we have $\lvert G_j\rvert=g_j$ and
\begin{equation}
\norm{\frac{1}{g_j}\sum_{i\in G_j}v_i-\frac{1}{r}\sum_{i=1}^rv_i}\le\frac{4d}{g_j}.
\end{equation}
\end{lemma}

\begin{lemma}\label{lem:marginaluniqueproperties}
Let $\Phi=\support t$, and suppose that the map $\distributions(\Phi)\to\distributions(\mathcal{A})\times\distributions(\mathcal{B})\times\distributions(\mathcal{C})$ that maps a distribution to the collection of its marginals is injective.
\begin{enumerate}
\item\label{it:typetensorbounds} For all $m\in\positiveintegers$ and $P\in\distributions[m](\Phi)$ the inequalities
\begin{equation}
\unittensor{1}\le t^{\otimes m}[\typeclass{m}{P}]\le t^{\otimes m}
\end{equation}
hold.
\item\label{it:typetensorsupermultiplicative} For all $m,n\in\positiveintegers$, $P\in\distributions[m]$, and $Q\in\distributions[n]$, we have
\begin{equation}
t^{\otimes m}[\typeclass{m}{P}]\otimes t^{\otimes n}[\typeclass{n}{Q}]
 = t^{\otimes(m+n)}[\typeclass{m}{P}\times\typeclass{n}{Q}]
 \le t^{\otimes(m+n)}[\typeclass{m+n}{\frac{mP+nQ}{m+n}}].
\end{equation}
\item\label{it:typetensorpowerbound} Let $P\in\distributions[k](\Phi)$ for some $k\in\positiveintegers$. For all $b,g,s\in\naturals$ the inequality
\begin{equation}
t^{\otimes bgsk}[\typeclass{bgsk}{P}]\le\unittensor{(bk+1)^{\lvert\Phi\rvert gs}}\otimes \left(t^{\otimes (bgk+4\lvert\Phi\rvert bk^2)}\left[\typeclass{bgk+4\lvert\Phi\rvert bk^2}{P}\right]\right)^{\otimes s}
\end{equation}
holds.
\end{enumerate}
\end{lemma}
\begin{proof}
\ref{it:typetensorbounds}:
The set $\typeclass{m}{P}$ is not empty since $P\in\distributions[m](\Phi)$. If $x\in\typeclass{m}{P}$ then $t^{\otimes m}[x]\le t^{\otimes m}[\typeclass{m}{P}]$ and the corresponding block is a tensor product of blocks of $t$ with indices from $\Phi=\support t$. This implies that $t^{\otimes m}[x]\neq 0$ and, since every nonzero tensor restricts to $\unittensor{1}$, that $\unittensor{1}\le t^{\otimes m}[x]\le t^{\otimes m}[\typeclass{m}{P}]$.

For the upper bound, we use the decomposition
\begin{equation}
t^{\otimes m}
 = \sum_{Q\in\distributions[m](\Phi)}t^{\otimes m}[\typeclass{m}{Q}]
\end{equation}
into sums of blocks over joint type classes. The set $\typeclass{m}{Q}$ is either disjoint from or contained in $\typeclass{m}{P_A}\times\typeclass{m}{P_B}\times\typeclass{m}{P_C}$ depending on whether the marginals of $Q$ agree with those of $P$, and by assumption they agree if and only if $Q=P$. This implies that
\begin{equation}
\begin{split}
\left(\typeclassprojector{m}{P_A}\otimes\typeclassprojector{m}{P_B}\otimes\typeclassprojector{m}{P_C}\right)t^{\otimes m}
 & = t^{\otimes m}[\typeclass{m}{P_A}\times\typeclass{m}{P_B}\times\typeclass{m}{P_C}]  \\
 & = \sum_{Q\in\distributions[m](\Phi)}t^{\otimes m}[\typeclass{m}{Q}\cap(\typeclass{m}{P_A}\times\typeclass{m}{P_B}\times\typeclass{m}{P_C})]  \\
 & = \sum_{\substack{Q\in\distributions[m](\Phi)  \\  \typeclass{m}{Q}\subseteq \typeclass{m}{P_A}\times\typeclass{m}{P_B}\times\typeclass{m}{P_C}}}t^{\otimes m}[\typeclass{m}{Q}]  \\
 & = t^{\otimes m}[\typeclass{m}{P}].
\end{split}
\end{equation}

\ref{it:typetensorsupermultiplicative}:
Since $\typeclass{m}{P}\times\typeclass{n}{Q}=\typeclass{m+n}{\frac{mP+nQ}{m+n}}\cap\left((\typeclass{m}{P_A}\times\typeclass{n}{Q_A})\times(\typeclass{m}{P_B}\times\typeclass{n}{Q_B})\times(\typeclass{m}{P_C}\times\typeclass{n}{Q_C})\right)$, we have
\begin{equation}
\begin{split}
t^{\otimes m}[\typeclass{m}{P}]\otimes t^{\otimes n}[\typeclass{n}{Q}]
 & = t^{\otimes (m+n)}[\typeclass{m}{P}\times\typeclass{n}{Q}]  \\
 & = t^{\otimes (m+n)}[\typeclass{m+n}{\frac{mP+nQ}{m+n}}\cap\left((\typeclass{m}{P_A}\times\typeclass{n}{Q_A})\times(\typeclass{m}{P_B}\times\typeclass{n}{Q_B})\times(\typeclass{m}{P_C}\times\typeclass{n}{Q_C})\right)]  \\
 & = \left((\typeclassprojector{m}{P_A}\otimes\typeclassprojector{n}{Q_A})\otimes(\typeclassprojector{m}{P_B}\otimes\typeclassprojector{n}{Q_B})\otimes(\typeclassprojector{m}{P_C}\otimes\typeclassprojector{n}{Q_C})\right)t^{\otimes (m+n)}[\typeclass{m+n}{\frac{mP+nQ}{m+n}}]  \\
 & \le t^{\otimes(m+n)}[\typeclass{m+n}{\frac{mP+nQ}{m+n}}].
\end{split}
\end{equation}

\ref{it:typetensorpowerbound}:
We group the $bgsk$ factors of $t^{\otimes bgsk}$ into $gs$ blocks of size $bk$, and partition $\typeclass{bgsk}{P}$ according to the type within each block, encoded in $Q=(Q_1,Q_2,\dots,Q_{gs})\in\distributions[bk](\Phi)^{gs}$:
\begin{equation}\label{eq:blocktypedecomposition}
\begin{split}
t^{\otimes bgsk}[\typeclass{bgsk}{P}]
 & = \sum_{Q\in\distributions[bk](\Phi)^{gs}}t^{\otimes bgsk}[\typeclass{bgsk}{P}\cap(\typeclass{bk}{Q_1}\times\cdots\times\typeclass{bk}{Q_{gs}})]  \\
 & = \sum_{\substack{Q\in\distributions[bk](\Phi)^{gs}  \\  Q_1+\dots+Q_{gs}=gsP}}t^{\otimes bgsk}[\typeclass{bk}{Q_1}\times\cdots\times\typeclass{bk}{Q_{gs}}]  \\
 & \le \bigoplus_{\substack{Q\in\distributions[bk](\Phi)^{gs}  \\  Q_1+\dots+Q_{gs}=gsP}}t^{\otimes bgsk}[\typeclass{bk}{Q_1}\times\cdots\times\typeclass{bk}{Q_{gs}}].
\end{split}
\end{equation}

Consider the term corresponding to some $Q$. We apply \cref{lem:Steinitz} to the vectors $Q_1,\dots,Q_{gs}\in\reals^\Phi$ ($d=\lvert\Phi\rvert$) of $l^\infty$-norm at most $1$, $r=gs$, and $g_1=g_2=\dots=g_s=g$. It asserts the existence of a partition of $[gs]$ into $s$ subsets $G_1,\dots,G_s$ of size $g$ each such that
\begin{equation}
\norm[\infty]{\frac{1}{g}\sum_{i\in G_j}Q_i-\frac{1}{gs}\sum_{i=1}^{gs}Q_i}\le\frac{4\lvert\Phi\rvert}{g},
\end{equation}
i.e.,
\begin{equation}
\norm[\infty]{\sum_{i\in G_j}bkQ_i-gbkP}\le 4\lvert\Phi\rvert bk.
\end{equation}
Since $\support(Q_i)\subseteq\support(P)$ and $P\in\distributions[k](\Phi)$, therefore $1_{\support(P)}\le kP$, this implies the entrywise inequality
\begin{equation}
\sum_{i\in G_j}bkQ_i-gbkP
 \le 4\lvert\Phi\rvert bk 1_{\support(P)}
 \le 4\lvert\Phi\rvert bk^2P,
\end{equation}
i.e.,
\begin{equation}
R:=\frac{1}{4\lvert\Phi\rvert bk^2}\left((gbk+4\lvert\Phi\rvert bk^2)P-\sum_{i\in G_j}bkQ_i\right)\ge 0.
\end{equation}
Since $\norm[1]{R}=1$, we have $R\in\distributions[4\lvert\Phi\rvert bk^2](\Phi)$. It follows that for each $j$,
\begin{equation}
\begin{split}
t^{\otimes bgk}\left[\prod_{i\in G_j}\typeclass{bk}{Q_i}\right]
 & \le t^{\otimes (bgk+4\lvert\Phi\rvert bk^2)}\left[\typeclass{4\lvert\Phi\rvert bk^2}{R}\times\prod_{i\in G_j}\typeclass{bk}{Q_i}\right]  \\
 & \le t^{\otimes (bgk+4\lvert\Phi\rvert bk^2)}\left[\typeclass{4\lvert\Phi\rvert bk^2}{R}\times\typeclass{gbk}{\frac{1}{g}\sum_{i\in G_j}Q_i}\right]  \\
 & \le t^{\otimes (bgk+4\lvert\Phi\rvert bk^2)}\left[\typeclass{bgk+4\lvert\Phi\rvert bk^2}{P}\right].
\end{split}
\end{equation}

Multiplying the inequalities for $j=1,2,\ldots,s$, we obtain $t^{\otimes bgsk}[\typeclass{bk}{Q_1}\times\cdots\times\typeclass{bk}{Q_{gs}}] \le \left(t^{\otimes (bgk+4\lvert\Phi\rvert bk^2)}\left[\typeclass{bgk+4\lvert\Phi\rvert bk^2}{P}\right]\right)^{\otimes s}$ for each term in the direct sum in \eqref{eq:blocktypedecomposition}. The number of terms is bounded by $\lvert\distributions[bk](\Phi)^{gs}\rvert\le(bk+1)^{\lvert\Phi\rvert gs}$, therefore
\begin{equation}
t^{\otimes bgsk}[\typeclass{bgsk}{P}]\le\unittensor{(bk+1)^{\lvert\Phi\rvert gs}}\otimes \left(t^{\otimes (bgk+4\lvert\Phi\rvert bk^2)}\left[\typeclass{bgk+4\lvert\Phi\rvert bk^2}{P}\right]\right)^{\otimes s}.
\end{equation}
\end{proof}

\begin{theorem}\label{thm:typetensorapproximatelygeometric}
Let $\Phi=\support t$, and suppose that the map $\distributions(\Phi)\to\distributions(\mathcal{A})\times\distributions(\mathcal{B})\times\distributions(\mathcal{C})$ that maps a distribution to the collection of its marginals is injective. Let $k\in\positiveintegers$ and $P\in\distributions[k](\Phi)$. Then the sequence $a=(a_n)_{n\in\naturals}$ given by
\begin{equation}
a_n=t^{\otimes nk}[\typeclass{nk}{P}]
\end{equation}
is approximately geometric.
\end{theorem}
\begin{proof}
Let $m=bg+4\lvert\Phi\rvert bk$ with $b,g\in\positiveintegers$ to be specified later. For all $q,r\in\naturals$ satisfying $0\le r<m$, by \cref{lem:marginaluniqueproperties} we have
\begin{equation}
a_m^q\le a_m^qa_r\le a_{qm+r}
\end{equation}
and
\begin{equation}
\begin{split}
a_{qm+r}
 & \le a_{qm+r}a_{bgs-(qm+r)}  \\
 & \le a_{bgs}  \\
 & \le\unittensor{(bk+1)^{\lvert\Phi\rvert gs}}\otimes a_m^{\otimes s}  \\
 & \le\unittensor{(bk+1)^{\lvert\Phi\rvert gs}}\otimes\unittensor{\tensorrank(a_m^{\otimes(s-q)})}\otimes a_m^{\otimes q}  \\
 & \le\unittensor{(bk+1)^{\lvert\Phi\rvert gs}\tensorrank(t)^{mk(s-q)}}\otimes a_m^{\otimes q}  \\
\end{split}
\end{equation}
provided that $qm+r\le bgs$. To ensure this condition, we set
\begin{equation}
s=\left\lceil\frac{(q+1)m}{bg}\right\rceil.
\end{equation}
Then
\begin{equation}
\log\left[(bk+1)^{\lvert\Phi\rvert gs}\tensorrank(t)^{mk(s-q)}\right]
 = \lvert\Phi\rvert gs\log(bk+1)+mk(s-q)\log\tensorrank(t)
\end{equation}
is finite, and the growth of the terms is bounded as
\begin{equation}
\begin{split}
\limsup_{q\to\infty}\frac{\lvert\Phi\rvert gs\log(bk+1)}{qm+r}
 & = \limsup_{q\to\infty}\frac{\lvert\Phi\rvert g\left\lceil\frac{(q+1)m}{bg}\right\rceil\log(bk+1)}{qm+r}  \\
 & = \lvert\Phi\rvert\frac{1}{b}\log(bk+1)
\end{split}
\end{equation}
and
\begin{equation}
\begin{split}
\limsup_{q\to\infty}\frac{mk(s-q)\log\tensorrank(t)}{qm+r}
 & = \limsup_{q\to\infty}\frac{mk\left(\left\lceil\frac{(q+1)m}{bg}\right\rceil-q\right)\log\tensorrank(t)}{qm+r}  \\
 & \le \limsup_{q\to\infty}\frac{mk\left(\frac{(q+1)m}{bg}+1-q\right)\log\tensorrank(t)}{qm+r}  \\
 & = k\left(\frac{m}{bg}-1\right)\log\tensorrank(t)  \\
 & = \frac{4\lvert\Phi\rvert k^2}{g}\log\tensorrank(t).
\end{split}
\end{equation}

For $\epsilon>0$, choose $b,g\in\positiveintegers$ sufficiently large so that
\begin{equation}
\lvert\Phi\rvert\frac{1}{b}\log(bk+1)+\frac{4\lvert\Phi\rvert k^2}{g}\log\tensorrank(t)<\epsilon.
\end{equation}
Then there is some $c\in\naturals$ such that for all $n=qm+r$ the inequality
\begin{equation}
\log\left[(bk+1)^{\lvert\Phi\rvert gs}\tensorrank(t)^{mk(s-q)}\right]\le c+\lfloor\epsilon n\rfloor
\end{equation}
holds, therefore $a_m^q\le a_n\le 2^{c+\lfloor\epsilon n\rfloor}a_m^q$.
\end{proof}

\begin{definition}\label{def:typetensorelement}
With $t\in U\otimes V\otimes W$, $\Phi=\support t$, $k\in\positiveintegers$, $P\in\distributions[k](\Phi)$, and $a$ as in \cref{thm:typetensorapproximatelygeometric}, we define $t[P]$ to be the element $a^{1/k}$ of $\completion{S}$.
\end{definition}
It follows from \cref{lem:marginaluniqueproperties} that these elements satisfy $t[P]\asymptoticle t$ and the following log-concavity property: if $P$ and $Q$ are rational distributions on $\Phi$ and $\lambda\in[0,1]\cap\rationals$, then
\begin{equation}
t[\lambda P+(1-\lambda)Q]\asymptoticge t[P]^{\otimes\lambda}\otimes t[Q]^{\otimes(1-\lambda)}.
\end{equation}

\begin{example}
Let $U=V=W=\complexes\oplus\complexes$ with basis $(e_0,e_1)$ in all three factors, and let $t=e_1\otimes e_0\otimes e_0+e_0\otimes e_1\otimes e_0+e_0\otimes e_0\otimes e_1$. Then $\Phi:=\support t=\{(1,0,0),(0,1,0),(0,0,1)\}$. Let $P\in\distributions[3](\Phi)$ be the uniform distribution. Then $t[P]^{\otimes 3}$ is the element of $\completion{S}$ represented by the sequence of balanced tripartitioning tensors.
\end{example}

Next we consider two special cases of the construction, where $t$ has even more structure. Recall that a subset $\Phi\in\mathcal{A}\times\mathcal{B}\times\mathcal{C}$ is \emph{free} if any two elements differ in at least two positions. Suppose that the decomposition of $t$ is into $1\times 1\times 1$ blocks and $\Phi$ is free in addition to the property that the marginals on singletons uniquely determine every probability distribution on $\Phi$. By \cite[Theorem 4.18.]{christandl2023universal}, the quantum functionals at $t$ can be evaluated as
\begin{equation}
\log F^\theta(t)
 = E^\theta(t)
 = \max_{Q\in\distributions(\Phi)}\theta(A)\entropy(Q_A)+\theta(B)\entropy(Q_B)+\theta(C)\entropy(Q_C)
\end{equation}
for all $\theta\in\distributions(\{A,B,C\})$.

For all $m$, any subset of the support of $t^{\otimes m}$ is also free. In particular, $E^\theta$ can be evaluated in a similar way at $t^{\otimes m}[\typeclass{m}{P}]$. Since the symmetric group $S_m$ acts transitively on $\typeclass{m}{P}$ and all three projections $\typeclass{m}{P_A}$, $\typeclass{m}{P_B}$, $\typeclass{m}{P_C}$, with equivariant projection maps, and the entropy functional is symmetric and concave, all three terms are maximized when the distribution $Q$ is uniform on $\typeclass{m}{P}$. This implies
\begin{equation}
\begin{split}
\log F^\theta(t^{\otimes m}[\typeclass{m}{P}])
 & = E^\theta(t^{\otimes m}[\typeclass{m}{P}])  \\
 & = \theta(A)\log\lvert\typeclass{m}{P_A}\rvert + \theta(B)\log\lvert\typeclass{m}{P_B}\rvert + \theta(C)\log\lvert\typeclass{m}{P_C}\rvert,
\end{split}
\end{equation}
therefore
\begin{equation}
\begin{split}
\log\completion{F^\theta}(t[P])
 & = \frac{1}{k}\lim_{n\to\infty}\frac{1}{n}\log F^\theta(t^{\otimes nk}[\typeclass{nk}{P}])  \\
 & = \theta(A)\entropy(P_A)+\theta(B)\entropy(P_B)+\theta(C)\entropy(P_C).
\end{split}
\end{equation}

In the second special case we allow blocks of arbitrary sizes, but require a stronger property for $\Phi$, called tightness. Recall that a subset $\Phi\subseteq\mathcal{A}\times\mathcal{B}\times\mathcal{C}$ is \emph{tight} if there exist injective maps $\alpha:\mathcal{A}\to\integers$, $\beta:\mathcal{B}\to\integers$, and $\gamma:\mathcal{C}\to\integers$ such that for all $(i,j,k)\in\Phi$ the equality $\alpha(i)+\beta(j)+\gamma(k)=0$ holds. For $\Phi$ tight and $P\in\distributions[k](\Phi)$, Strassen proves in \cite{strassen1991degeneration} using the method of Coppersmith and Winograd that there exist $X_n\subseteq\typeclass{nk}{P_A}$, $Y_n\subseteq\typeclass{nk}{P_B}$, and $Z_n\subseteq\typeclass{nk}{P_C}$ such that $(X_n\times Y_n\times Z_n)\cap\Phi^{nk}$ is a diagonal of size $2^{nk\min\{\entropy(P_A),\entropy(P_B),\entropy(P_C)\}-o(n)}$.

Combining Strassen's result with \cref{thm:typetensorapproximatelygeometric}, and using the notation from \cref{def:typetensorelement}, if $\Phi=\support t$ is tight and the map from $\distributions(\Phi)$ to the triples of marginals is injective, then for every rational distribution $P$ on $\Phi$ we have the inequality
\begin{equation}\label{eq:CWS}
t[P]\ge \unittensor{2^{\min\{\entropy(P_A),\entropy(P_B),\entropy(P_C)\}}}\otimes\prod_{(i,j,k)\in\Phi} t_{i,j,k}^{\otimes P(i,j,k)}.
\end{equation}

\begin{example}
As an illustration, we rephrase the upper bound on $\omega$ by Coppermith and Winograd using our notations. For $q\in\positiveintegers$, we view the tensors
\begin{multline}
\CW{q}
 = e_0\otimes e_0\otimes e_{q+1} + 
   e_0\otimes e_{q+1}\otimes e_0 + 
   e_{q+1}\otimes e_0\otimes e_0  \\  + 
   \sum_{i=1}^q\left(
   e_0\otimes e_i\otimes e_i + 
   e_i\otimes e_0\otimes e_i + 
   e_i\otimes e_i\otimes e_0 
\right)
\end{multline}
as block tensors with respect to the partition of the bases $(\{e_0\},\{e_1,e_2,\dots,e_q\},\{e_{q+1}\})$, for which we use the labels $0,1,2$. The support is $\Phi=\{(0,0,2),(0,2,0),(2,0,0),(1,1,0),(1,0,1),(0,1,1)\}$, which is tight, and probability distributions on $\Phi$ are uniquely determined by their marginals. There are four types of blocks, isomorphic to $\langle 1,1,1\rangle$, $\langle q,1,1\rangle$, $\langle 1,q,1\rangle$, $\langle 1,1,q\rangle$.

For a rational distribution $P$ on $\Phi$, we have the chain of inequalities
\begin{equation}\label{eq:CWmethod}
\unittensor{q+2}
 \asymptoticge \CW{q}
 \asymptoticge \CW{q}[P]
 \asymptoticge \unittensor{2^{\min\{\entropy(P_A),\entropy(P_B),\entropy(P_C)\}}}\otimes\langle q^{P(1,0,1)},q^{P(1,1,0)},q^{P(0,1,1)}\rangle.
\end{equation}
Here the first inequality is due to an explicit border rank decomposition of $\CW{q}$, the second inequality is a consequence of \cref{lem:marginaluniqueproperties}, and the third inequality is from \eqref{eq:CWS}.

Applying $\asymptoticrank$ to the inequality, we obtain
\begin{equation}
q+2\ge 2^{\min\{\entropy(P_A),\entropy(P_B),\entropy(P_C)\}}\asymptoticrank(\langle q^{P(1,0,1)},q^{P(1,1,0)},q^{P(0,1,1)}\rangle).
\end{equation}

Assuming a symmetric distribution $P(0,0,2)=P(0,2,0)=P(2,0,0)=\frac{\beta}{3}$ and $P(1,1,0)=P(1,0,1)=P(0,1,1)=\frac{1-\beta}{3}$, the inequality \eqref{eq:CWmethod} becomes
\begin{equation}
\unittensor{q+2}
 \asymptoticge \unittensor{2^{\entropy\left(\frac{1+\beta}{3},2\frac{1-\beta}{3},\frac{\beta}{3}\right)}}\otimes\langle 2,2,2\rangle^{\otimes\frac{1-\beta}{3}\log q},
\end{equation}
and after applying $\asymptoticrank$ and taking logarithms, we obtain
\begin{equation}
\frac{1-\beta}{3}\log(q)\omega\le\log(q+2)-\entropy\left(\frac{1+\beta}{3},2\frac{1-\beta}{3},\frac{\beta}{3}\right).
\end{equation}
We set $q=6$ and $\beta=0.048$ to get $\omega\le 2.39$.
\end{example}

\subsection{Graphs}\label{sec:graphs}

Following \cite{zuiddam2019asymptotic}, we let $\graphs$ be the semiring of isomorphism classes of finite simple undirected graphs with disjoint union $\disjointunion$ as addition, the strong product $\strongproduct$ as multiplication, equipped with the cohomomorphism preorder $\le$. This is a semiring with a Strassen preorder and, the Shannon capacity of any graph $G$ is equal to its asymptotic subrank in this semiring, which provides a dual characterization as a minimum over the asymptotic spectrum \cite[Theorem 1.1]{zuiddam2019asymptotic}. We continue using the convention that the canonical map $i:\graphs\to\completion{\graphs}$ will be omitted, but the preorder on $\graphs$ is $\le$ and the preorder on $\completion{\graphs}$ is $\asymptoticle$. The natural numbers are embedded in $\graphs$ as the complements of complete graphs $\complement{K_n}$.

Some known graph parameters are elements of the asymptotic spectrum: the Lovász number $\vartheta$ \cite{lovasz1979shannon}, the fractional clique cover number $\complement{\chi}_f$, the fractional Haemers bounds $\mathcal{H}_{\mathbb{F}}$ over arbitrary fields \cite{haemers1978upper,haemers1979some,blasiak2013graph,bukh2018fractional}, and the fractional orthogonal rank $\xi_f$ \cite{cubitt2014bounds}. Of these, the fractional clique cover number is rational for every graph \cite{scheinerman2011fractional}, therefore, by \cref{rem:asymptoticallycompletecriterion}, the semiring of graphs is not asymptotically complete.

In \cite{vrana2021probabilistic} it is shown that the asymptotic spectrum has a parametrization by a convex set. The main idea of the proof is to evaluate elements of the spectrum on sequences of subgraphs of graph powers induced on type classes, and study their appropriately normalized limits. In the following we will show that such sequences are in fact approximately geometric, and that the properties of the limiting parameters correspond to relations between elements of the asymptotic completion $\completion{\graphs}$. We take a route different from \cite{vrana2021probabilistic} in order to maximize the flexibility of the result. Concretely, we will allow arbitrary converging types and linearly growing powers as opposed to a fixed rational distribution and powers in an arithmetic sequence.

The following relation holds in $\graphs$ between any vertex-transitive graph and an arbitrary induced subgraph. In general, for a graph $G$ with vertex set $V(G)$ and $S\subseteq V(G)$, the induced subgraph $G[S]$ has vertex set $S$ and its edges are the edges of $G$ that have both endpoints in $S$.
\begin{lemma}[{\cite[Lemma 3.1.]{vrana2021probabilistic}}]\label{lem:transitiveinducedsubgraph}
Let $H$ be a vertex-transitive graph and $S\subseteq V(H)$. Then $H[S]\le H\le\complement{K_N}\strongproduct H[S]$ with
\begin{equation}
N=\left\lfloor\frac{\lvert V(H)\rvert}{\lvert S\rvert}\ln\lvert V(H)\rvert\right\rfloor+1.
\end{equation}
\end{lemma}
This implies, in particular, that in $\completion{\graphs}$, we have $H[S]\asymptoticle H\asymptoticle \frac{\lvert V(H)\rvert}{\lvert S\rvert}H[S]$. However, we will use directly \cref{lem:transitiveinducedsubgraph} instead of this simpler form, so that we get statements about $\le$ that hold in $\graphs$.

The next two lemmas are used for bounding the multiplicative distance of subgraphs induced on different type classes. The idea is to embed both in a larger type graph corresponding to the maximum of the distributions rescaled by exponent.
\begin{lemma}\label{lem:maxmeasure}
Let $\mathcal{X}$ be a finite set, $P_1,P_2\in\distributions(\mathcal{X})$, $w_1,w_2\in\nonnegativereals$, $M(x)=\max\{w_1P_1(x),w_2P_2(x)\}$ for all $x\in\mathcal{X}$, and $Q=\frac{M}{\norm[1]{M}}$. Let $\bar{w}=\max\{w_1,w_2\}$ and $\ubar{w}=\min\{w_1,w_2\}$. Then
\begin{align}
\bar{w} \le \norm[1]{M} & \le \bar{w}+\ubar{w}\norm[1]{P_1-P_2}  \\
\norm[1]{Q-P_1} & \le 2\frac{\lvert w_1-w_2\rvert}{\bar{w}}+2\norm[1]{P_1-P_2}  \\
\norm[1]{Q-P_2} & \le 2\frac{\lvert w_1-w_2\rvert}{\bar{w}}+2\norm[1]{P_1-P_2}.
\end{align}
\end{lemma}
\begin{proof}
Since $0\le w_1P_1\le M$ and $0\le w_2P_2\le M$ by construction, the norms satisfy $w_1=\norm[1]{w_1P_1}\le\norm[1]{M}$ and $w_2=\norm[1]{w_2P_2}\le\norm[1]{M}$, and therefore $\bar{w}\le\norm[1]{M}$.

We have
\begin{equation}
\begin{split}
\norm[1]{M-w_1P_1}
 & = \sum_{x\in\mathcal{X}}\max\{0,w_2P_2(x)-w_1P_1(x)\}  \\
 & \le \norm[1]{w_1P_1-w_2P_2}  \\
 & \le \norm[1]{w_1P_1-\ubar{w}P_1}+\norm[1]{\ubar{w}P_1-\ubar{w}P_2}+\norm[1]{\ubar{w}P_2-w_2P_2}  \\
 & = \lvert w_1-\ubar{w}\rvert+\ubar{w}\norm[1]{P_1-P_2}+\lvert\ubar{w}-w_2\rvert  \\
 & = \ubar{w}\norm[1]{P_1-P_2}+\lvert w_1-w_2\rvert,
\end{split}
\end{equation}
and similarly
\begin{equation}
\norm[1]{M-w_2P_2} \le \ubar{w}\norm[1]{P_1-P_2}+\lvert w_1-w_2\rvert,
\end{equation}
therefore
\begin{align}
\norm[1]{M}
 & \le \norm[1]{M-w_1P_1}+\norm[1]{w_1P_1} \le \ubar{w}\norm[1]{P_1-P_2}+\lvert w_1-w_2\rvert+w_1
\intertext{and}
\norm[1]{M}
 & \le \norm[1]{M-w_2P_2}+\norm[1]{w_2P_2} \le \ubar{w}\norm[1]{P_1-P_2}+\lvert w_1-w_2\rvert+w_2.
\end{align}
These minimum of the upper bounds gives $\norm[1]{M} \le \bar{w}+\ubar{w}\norm[1]{P_1-P_2}$.

Finally,
\begin{equation}
\begin{split}
\norm[1]{Q-P_1}
 & = \frac{1}{\norm[1]{M}}\norm[1]{M-\norm[1]{M}P_1}  \\
 & = \frac{1}{\norm[1]{M}}\norm[1]{M-w_1P_1+w_1P_1-\norm[1]{M}P_1}  \\
 & \le \frac{\ubar{w}\norm[1]{P_1-P_2}+\lvert w_1-w_2\rvert}{\norm[1]{M}}+\frac{\norm[1]{M}-w_1}{\norm[1]{M}}  \\
 & \le \frac{\ubar{w}\norm[1]{P_1-P_2}+\lvert w_1-w_2\rvert}{\norm[1]{M}}+\frac{\bar{w}+\ubar{w}\norm[1]{P_1-P_2}-w_1}{\norm[1]{M}}  \\
 & \le 2\frac{\lvert w_1-w_2\rvert}{\bar{w}}+2\norm[1]{P_1-P_2},
\end{split}
\end{equation}
while the bound on $\norm[1]{Q-P_2}$ follows by reversing the roles of $P_1$ and $P_2$.
\end{proof}

\begin{lemma}\label{lem:typegraphcontinuity}
Let $G$ be a graph, $m_1,m_2\in\positiveintegers$, $P_1\in\distributions[m_1](V(G))$, $P_2\in\distributions[m_2](V(G))$. Then $\typegraph{G}{m_1}{P_1}\le\complement{K_2}^{\strongproduct k}\strongproduct\typegraph{G}{m_2}{P_2}$ where
\begin{multline}
k\le 1+\log(\ln\lvert V(G)\rvert)+(\lvert V(G)\rvert+1)\log(m_1+m_2)  \\  +2\left(\lvert m_1-m_2\rvert+m_2\norm[1]{P_1-P_2}\right)\log\lvert V(G)\rvert+m_2h(\epsilon)
\end{multline}
with
\begin{equation}
\epsilon=\min\left\{\frac{1}{2},\frac{\lvert m_1-m_2\rvert}{\max\{m_1,m_2\}}+\norm[1]{P_1-P_2}\right\}.
\end{equation}
\end{lemma}
\begin{proof}
Let $M(v)=\max\{m_1P_1(v),m_2P_2(v)\}$ for all $v\in V(G)$, $n=\norm[1]{M}$, and $Q=\frac{M}{\norm[1]{M}}$. Both $\typegraph{G}{m_1}{P_1}$ and $\typegraph{G}{m_2}{P_2}$ appear in $\typegraph{G}{n}{Q}$ as induced subgraphs (e.g., by appending to each string in $\typeclass{m_1}{P_1}$ the missing symbols in some fixed order so that the frequencies match $Q$). By \cref{lem:transitiveinducedsubgraph}, we have
\begin{equation}
\typegraph{G}{m_1}{P_1}
 \le \typegraph{G}{n}{Q}
 \le \complement{K_N}\strongproduct\typegraph{G}{m_2}{P_2},
\end{equation}
where
\begin{equation}
N
 = \left\lfloor\frac{\lvert\typeclass{n}{Q}\rvert}{\lvert\typeclass{m_2}{P_2}\rvert}\ln\lvert\typeclass{n}{Q}\rvert\right\rfloor+1  \\
\end{equation}
Using $\lceil\log(\lfloor x\rfloor+1)\rceil\le 1+\log x$ for $x\ge 1$ and $\complement{K_N}\le\complement{K_2}^{\strongproduct \lceil\log N\rceil}$, the claimed inequality holds with
\begin{equation}
\begin{split}
k
 & = \lceil\log N\rceil  \\
 & \le 1+\log\left(\frac{\lvert\typeclass{n}{Q}\rvert}{\lvert\typeclass{m_2}{P_2}\rvert}\ln\lvert\typeclass{n}{Q}\rvert\right)  \\
 & \le 1+\log\left(n\ln(\lvert V(G)\rvert)(m_2+1)^{\lvert V(G)\rvert}2^{n\entropy(Q)-m_2\entropy(P_2)}\right)  \\
 & = 1+\log(\ln\lvert V(G)\rvert)+\log n +\lvert V(G)\rvert\log(m_2+1)+n\entropy(Q)-m_2\entropy(P_2)  \\
 & \le 1+\log(\ln\lvert V(G)\rvert)+(\lvert V(G)\rvert+1)\log(m_1+m_2)+n\entropy(Q)-m_2\entropy(P_2),
\end{split}
\end{equation}
where in the last step we used the simple bound $n\le m_1+m_2$. We split the entropy difference as
\begin{equation}
n\entropy(Q)-m_2\entropy(P_2)
  = \left(\norm[1]{M}-m_2\right)\entropy(Q)+m_2\left(\entropy(Q)-\entropy(P_2)\right),
\end{equation}
and bound the two terms separately using \cref{lem:maxmeasure,lem:entropycontinuity} as
\begin{equation}
\begin{split}
\left(\norm[1]{M}-m_2\right)\entropy(Q)
 & \le (\lvert m_1-m_2\rvert+\min\{m_1,m_2\}\norm[1]{P_1-P_2})\entropy(Q)  \\
 & \le (\lvert m_1-m_2\rvert+m_2\norm[1]{P_1-P_2})\log\lvert V(G)\rvert
\end{split}
\end{equation}
and
\begin{equation}
\begin{split}
m_2\left(\entropy(Q)-\entropy(P_2)\right)
 & \le m_2\left(\entropy(Q)-\entropy(P_2)\right)  \\
 & \le m_2\frac{1}{2}\norm[1]{Q-P_2}\log(\lvert V(G)\rvert-1)+m_2h(\frac{1}{2}\norm[1]{Q-P_2})  \\
 & \le m_2\left(\frac{\lvert m_1-m_2\rvert}{\max\{m_1,m_2\}}+\norm[1]{P_1-P_2}\right)\log\lvert V(G)\rvert+m_2h(\epsilon)  \\
 & \le \left(\lvert m_1-m_2\rvert+m_2\norm[1]{P_1-P_2}\right)\log\lvert V(G)\rvert+m_2h(\epsilon).
\end{split}
\end{equation}
We obtain the bound on $k$ by combining the preceding inequalities.
\end{proof}

\begin{theorem}\label{thm:typegraphsequenceapproximatelygeometric}
Let $G$ be a graph, $(m_n)_{n\in\naturals}\in\naturals^\naturals$ such that the limit $\lim_{n\to\infty}\frac{m_n}{n}=:\mu$ exists in $\nonnegativereals$, and for each $n$ let $P_n\in\distributions[m_n](V(G))$ such that $P_n$ converges to some $P\in\distributions(V(G))$. Then $(\typegraph{G}{m_n}{P_n})_{n\in\naturals}$ is an approximately geometric sequence, and its asymptotic equivalence class only depends on $P$ and $\mu$.
\end{theorem}
\begin{proof}
If $\mu=0$, then $1\le \typegraph{G}{m_n}{P_n}\le G^{\strongproduct m_n}\le 2^{\lceil m_n\log\lvert V(G)\rvert\rceil}$ implies that the sequence is asymptotically equivalent to $(1,1,1,\dots)$, therefore it is approximately geometric, and all such sequences with $\mu=0$ are asymptotically equivalent.

Let $(m'_n)_{n\in\naturals}\in\naturals^\naturals$ and $P'_n\in\distributions[m'_n](V(G))$ sequences such that $\lim_{n\to\infty}\frac{m'_n}{n}=\lim_{n\to\infty}\frac{m_n}{n}=\mu$ and $\lim_{n\to\infty}P'_n=\lim_{n\to\infty}P_n=P$ as well. By \cref{lem:typegraphcontinuity}, we have
\begin{equation}
\typegraph{G}{m_n}{P_n}\le\complement{K_2}^{\strongproduct k_n}\strongproduct\typegraph{G}{m'_n}{P'_n},
\end{equation}
where
\begin{multline}
k_n\le 1+\log(\ln\lvert V(G)\rvert)+(\lvert V(G)\rvert+1)\log(m_n+m'_n)  \\  +2\left(\lvert m_n-m'_n\rvert+m'_n\norm[1]{P_n-P'_n}\right)\log\lvert V(G)\rvert+m'_nh(\epsilon_n)
\end{multline}
with
\begin{equation}
\epsilon_n=\min\left\{\frac{1}{2},\frac{\lvert m_n-m'_n\rvert}{\max\{m_n,m'_n\}}+\norm[1]{P_n-P'_n}\right\}.
\end{equation}
These satisfy $\epsilon_n\ge 0$,
\begin{equation}
\begin{split}
\limsup_{n\to\infty}\epsilon_n
 & \le \limsup_{n\to\infty}\left(\frac{\lvert m_n-m'_n\rvert}{\max\{m_n,m'_n\}}+\norm[1]{P_n-P'_n}\right)  \\
 & \le \limsup_{n\to\infty}\left(\frac{\lvert \frac{m_n}{n}-\frac{m'_n}{n}\rvert}{\max\{\frac{m_n}{n},\frac{m'_n}{n}\}}+\norm[1]{P_n-P}+\norm[1]{P-P'_n}\right)
 = 0,
\end{split}
\end{equation}
and
\begin{equation}
\begin{split}
\limsup_{n\to\infty}\frac{k_n}{n}
 & = \limsup_{n\to\infty}2\left(\left\lvert \frac{m_n}{n}-\frac{m'_n}{n}\right\rvert+\frac{m'_n}{n}\norm[1]{P_n-P'_n}\right)\log\lvert V(G)\rvert+\frac{m'_n}{n}h(\epsilon_n)  \\
 & = 0.
\end{split}
\end{equation}
Since the roles of the two sequences can be exchanged, we conclude that $(\typegraph{G}{m_n}{P_n})_{n\in\naturals}$ and $(\typegraph{G}{m'_n}{P'_n})_{n\in\naturals}$ are asymptotically equivalent.

We show that the sequence $a_n=\typegraph{G}{m_n}{P_n}$ is approximately geometric. Let $l\in\positiveintegers$ and $q,r\in\naturals$ satisfy $0\le r<q$ and $n=ql+r$. Then
\begin{equation}
a_l^q
 = \typegraph{G}{m_l}{P_l}^{\strongproduct q}
 \le \typegraph{G}{qm_l}{P_l}
 \le \complement{K_2}^{\strongproduct k_q}\strongproduct\typegraph{G}{m_n}{P_n}
 = \complement{K_2}^{\strongproduct k_q}\strongproduct a_n
\end{equation}
where
\begin{multline}
k_q\le 1+\log(\ln\lvert V(G)\rvert)+(\lvert V(G)\rvert+1)\log(qm_l+m_n)  \\  +2(\lvert qm_l-m_n\rvert+m_n\norm[1]{P_l-P_n})\log\lvert V(G)\rvert+m_nh(\epsilon_q)
\end{multline}
with
\begin{equation}
\epsilon_q=\min\left\{\frac{1}{2},\frac{\lvert qm_l-m_n\rvert}{\max\{qm_l,m_n\}}+\norm[1]{P_l-P_n}\right\}.
\end{equation}
They satisfy
\begin{equation}
\begin{split}
\lim_{q\to\infty}\epsilon_q
 & = \lim_{q\to\infty}\min\left\{\frac{1}{2},\frac{\lvert \frac{qm_l}{n}-\frac{m_n}{n}\rvert}{\max\{\frac{qm_l}{n},\frac{m_n}{n}\}}+\norm[1]{P_l-P_n}\right\}  \\
 & = \lim_{q\to\infty}\min\left\{\frac{1}{2},\frac{\lvert \frac{m_l}{l}-\mu\rvert}{\max\{\frac{m_l}{l},\mu\}}+\norm[1]{P_l-P}\right\}=:\delta_l
\end{split}
\end{equation}
and
\begin{equation}
\begin{split}
\limsup_{q\to\infty}\frac{k_q}{n}
 & \le \limsup_{q\to\infty}2\left\lvert \frac{qm_l}{n}-\frac{m_n}{n}\right\rvert+2\frac{m_n}{n}\norm[1]{P_l-P_n}\log\lvert V(G)\rvert+\frac{m_n}{n}h(\epsilon_q)  \\
 & = 2\left\lvert \frac{m_l}{l}-\mu\right\rvert+2\mu\norm[1]{P_l-P}\log\lvert V(G)\rvert+\mu h(\delta_l),
\end{split}
\end{equation}
which can be made arbitrarily close to $0$ by choosing $l$ sufficiently large.

Similarly,
\begin{equation}
\begin{split}
a_n
 & = \typegraph{G}{m_n}{P_n}
 \le \complement{K_2}^{\strongproduct k'_q}\strongproduct\typegraph{G}{qm_l}{P_l}  \\
 & \le \complement{K_2}^{\strongproduct k'_q}\strongproduct\complement{K_{N_q}}\strongproduct\typegraph{G}{m_l}{P_l}^{\strongproduct q}
 \le \complement{K_2}^{\strongproduct(k'_q+\lceil\log N_q\rceil)}\strongproduct a_l^{\strongproduct q}
\end{split}
\end{equation}
where
\begin{multline}
k'_q\le 1+\log(\ln\lvert V(G)\rvert)+(\lvert V(G)\rvert+1)\log(qm_l+m_n)  \\  +2(\lvert qm_l-m_n\rvert+qm_l\norm[1]{P_l-P_n})\log\lvert V(G)\rvert+qm_lh(\epsilon_q)
\end{multline}
and
\begin{equation}
\begin{split}
N_q
 & = \left\lfloor\frac{\lvert\typeclass{qm_l}{P_l}\rvert}{\lvert\typeclass{m_l}{P_l}\rvert^q}\ln\lvert\typeclass{qm_l}{P_l}\rvert\right\rfloor+1  \\
 & \le \left\lfloor\frac{2^{qm_l\entropy(P_l)}}{\left((m_l+1)^{-\lvert V(G)\rvert}2^{m_l\entropy(P_l)}\right)^q}qm_l\ln\lvert V(G)\rvert\right\rfloor+1  \\
 & = \left\lfloor(m_l+1)^{q\lvert V(G)\rvert}qm_l\ln\lvert V(G)\rvert\right\rfloor+1  \\
\end{split}
\end{equation}

We find in the same way as before that
\begin{equation}
\limsup_{q\to\infty}\frac{k'_q}{n}
 \le 2\left\lvert \frac{m_l}{l}-\mu\right\rvert+2\frac{m_l}{l}\norm[1]{P_l-P}\log\lvert V(G)\rvert+\frac{m_l}{l} h(\delta_l)
\end{equation}
and
\begin{equation}
\limsup_{q\to\infty}\frac{\lceil\log N_q\rceil}{n}
 \le \frac{1}{l}\lvert V(G)\rvert\log(m_l+1),
\end{equation}
and both can be made arbitrarily small by choosing $l$ sufficiently large.
\end{proof}

\begin{definition}\label{def:typegraphlimit}
For a graph $G$ and $P\in\distributions(V(G))$, let $m_n=\sum_{v\in V(G)}\lceil nP(v)\rceil$ and $P_n(v)=\frac{\lceil nP(v)\rceil}{m_n}$ for all $n\in\positiveintegers$ (and $m_0\in\positiveintegers$, $P_0\in\distributions[m_0](V(G))$ arbitrary). We define $G[P]\in\completion{\graphs}$ as the element that generates the geometric sequence which is asymptotically equivalent to the sequence $n\mapsto \typegraph{G}{m_n}{P_n}$.
\end{definition}
\begin{corollary}
In the setting of \cref{thm:typegraphsequenceapproximatelygeometric}, the resulting sequence represents $G[P]^{\strongproduct\mu}$.
\end{corollary}

In \cite{vrana2021probabilistic}, to every element $f\in\Delta(\graphs)$ a probabilistic refinement was introduced that takes a pair $(G,P)$ where $G$ is a finite simple undirected graph and $P\in\distributions(V(G))$, and returns a number $f(G,P)\in\nonnegativereals$. This functional can be characterized as follows: if $P_n\in\distributions[n](V(G))$ such that $P_n\to P$, then $f(G,P)=\lim_{n\to\infty}\sqrt[n]{f(\typegraph{G}{n}{P_n})}$ (the existence of the limit and its independence of the sequence follows from \cref{thm:typegraphsequenceapproximatelygeometric} as well). With the notation of \cref{def:typegraphlimit}, we have $f(G,P)=\completion{f}(G[P])$. The main properties of the map $(G,P)\mapsto f(G,P)$ can be reinterpreted as inequalities and equations in $\completion{G}$, with essentially the same proofs as in \cite{vrana2021probabilistic}:
\begin{proposition}\leavevmode
\begin{enumerate}
\item For any graph $G$, $P_1,P_2\in\distributions(V(G))$, and $p\in[0,1]$, the inequality $G[P_1]^{\strongproduct p}\strongproduct G[P_2]^{\strongproduct(1-p)}\asymptoticle G[pP_1+(1-p)P_2]$ holds
\item If $G$, $H$ are graphs, $P\in\distributions(V(G)\times V(H))$ with marginals $P_G$ and $P_H$, then $(G\strongproduct H)[P]\asymptoticle G[P_G]\strongproduct H[P_H]\asymptoticle \complement{K_2}^{\strongproduct \mutualinformation(G:H)_P}\strongproduct(G\strongproduct H)[P]$.
\item If $G$, $H$ are graphs, $P_G\in\distributions(V(G))$, $P_H\in\distributions(V(H))$, and $p\in[0,1]$, then $(G\disjointunion H)[pP_G\oplus(1-p)P_H]=\complement{K_2}^{\strongproduct h(p)}\strongproduct G[P_G]^{\strongproduct p}H[P_H]^{\strongproduct(1-p)}$
\item If $\varphi:\complement{H}\to\complement{G}$ is a homomorphism and $P\in\distributions(V(H))$, then $H[P]\asymptoticle G[\varphi_*(P)]$.
\end{enumerate}
\end{proposition}
Note, however, that while for every $f\in\Delta(\graphs)$ and graph $G$ there is a distribution $P\in\distributions(V(G))$ such that $f(G)=f(G,P)$, in general there is no $P$ such that $G=G[P]$ in $\completion{\graphs}$. On the other hand, if $G$ is vertex-transitive, then $G=G[U]$, where $U$ is the uniform distribution on $V(G)$.

In the remainder of this section, we connect a special class of approximately geometric sequences of graphs to another notion of approximation introduced in \cite{de2024asymptotic}. In that paper, de Boer, Buys, and Zuiddam consider a pseudometric on $\graphs$ (or a metric on $\graphs/\asymptoticeq$) defined as
\begin{equation}
d(G,H)=\max_{f\in\Hom(\graphs,\nonnegativereals)}\left\lvert f(G)-f(H)\right\rvert,
\end{equation}
where the maximum is over monotone semiring homomorphisms, which they call the \emph{asymptotic spectrum distance}. By the asymptotic spectrum duality, $d(G,H)\le\frac{a}{b}$ if and only if both $\complement{K_b}\strongproduct G\asymptoticle\complement{K_b}\strongproduct G\disjointunion\complement{K_a}$ and $\complement{K_b}\strongproduct H\asymptoticle\complement{K_b}\strongproduct H\disjointunion\complement{K_a}$ hold.

Informally speaking, the distance $d$ measures how far the graphs are in an additive sense (after scaling up by multiplication with natural numbers), whereas our notion of approximation is multiplicative (after taking powers). However, since every nonzero element in $\graphs$ is at least $1$ and thus nonzero graphs are bounded away from zero, the two notions become equivalent in a particular limit.
\begin{proposition}\label{prop:spectrumdistanceconvergence}
Let $G_0,G_1,G_2,\dots$ be nonempty graphs. Then $(G_n)_{n\in\naturals}$ is a Cauchy sequence with respect to the asymptotic spectrum distance if and only if the sequence $(G_n^{\strongproduct n})_{n\in\naturals}$ is approximately geometric with respect to the asymptotic cohomomorphism preorder.
\end{proposition}
\begin{proof}
Suppose that $(G_n^{\strongproduct n})_{n\in\naturals}$ is approximately geometric with respect to $\asymptoticle$. Let $p\in\naturals$ such that $G_n^{\strongproduct n}\asymptoticle 2^{pn}$ for all $n\in\positiveintegers$. Let $\epsilon>0$ and $m\in\positiveintegers$, $c\in\naturals$ such that $G_n^{\strongproduct n}\asymptoticle 2^{c+\lfloor\epsilon n\rfloor}(G_m^{\strongproduct m})^{\strongproduct\left\lfloor\frac{n}{m}\right\rfloor}$ and $(G_m^{\strongproduct m})^{\strongproduct\left\lfloor\frac{n}{m}\right\rfloor}\asymptoticle 2^{c+\lfloor\epsilon n\rfloor}G_n^{\strongproduct n}$ hold for all $n\in\naturals$.

Applying an arbitrary $f\in\Hom(\graphs,\nonnegativereals)$ to the inequalities, and taking $n$th roots, we obtain
\begin{equation}
f(G_n)
 \le 2^{\frac{c+\lfloor\epsilon n\rfloor}{n}}f(G_m)^{\frac{m}{n}\left\lfloor\frac{n}{m}\right\rfloor}
 \le 2^{\frac{c+\lfloor\epsilon n\rfloor}{n}}f(G_m)
\end{equation}
and
\begin{equation}
f(G_n)
 \ge 2^{-\frac{c+\lfloor\epsilon n\rfloor}{n}}f(G_m)^{\frac{m}{n}\left\lfloor\frac{n}{m}\right\rfloor}
 \ge 2^{-\frac{c+\lfloor\epsilon n\rfloor}{n}}f(G_m)^{\frac{m}{n}\left(\frac{n}{m}-1\right)}
 \ge 2^{-\frac{pm+c+\lfloor\epsilon n\rfloor}{n}}f(G_m),
\end{equation}
i.e.,
\begin{equation}\label{eq:uniformfGbound}
\left\lvert f(G_n)-f(G_m)\right\rvert
 \le \left(2^{\frac{pm+c+\lfloor\epsilon n\rfloor}{n}}-1\right)f(G_m)
 \le \left(2^{\frac{pm+c+\lfloor\epsilon n\rfloor}{n}}-1\right)2^p.
\end{equation}

We know that for any $f\in\Hom(\graphs,\nonnegativereals)$ the sequence $f(G_n^{\strongproduct n})$ is approximately geometric, and therefore the limit
\begin{equation}
\lim_{n\to\infty}\sqrt[n]{f(G_n^{\strongproduct n})}=\lim_{n\to\infty}f(G_n)
\end{equation}
exists. The limit of \eqref{eq:uniformfGbound} gives
\begin{equation}
\left\lvert \lim_{n\to\infty}f(G_n)-f(G_m)\right\rvert\le (2^{\epsilon}-1)2^p,
\end{equation}
therefore the convergence of $n\mapsto f(G_n)$ is uniform in $f\in\Hom(\graphs,\nonnegativereals)$, i.e., the sequence is a Cauchy sequence.

Conversely, suppose that $(G_n)_{n\in\naturals}$ is a Cauchy sequence with respect to the asymptotic spectrum distance, and let
\begin{equation}
M
 = \sup_{n\in\naturals}\max_{f\in\Hom(\graphs,\nonnegativereals)}f(G_n)
 = \sup_{n\in\naturals}\complement{\chi}_f(G_n).
\end{equation}
Let $\epsilon>0$ and choose $m\in\naturals$ such that for all $n\ge m$ and all $f\in\Hom(\graphs,\nonnegativereals)$ the value $f(G_n)$ differs from its limit by at most $\epsilon$. Then for all $n\ge m$ we have
\begin{equation}
\begin{split}
\frac{f\left(G_m^{\strongproduct m\left\lceil\frac{n}{m}\right\rceil}\right)}{f(G_n^{\strongproduct n})}
 & = \frac{f(G_m)^{m\left\lceil\frac{n}{m}\right\rceil}}{f(G_n)^n}
 \le \left(\frac{f(G_m)}{f(G_n)}\right)^n  \\
 & \le \left(\frac{f(G_n)+2\epsilon}{f(G_n)}\right)^n
 = \left(1+\frac{2\epsilon}{f(G_n)}\right)^n  \\
 & \le \left(1+2\epsilon\right)^n
 \le e^{2\epsilon n},
\end{split}
\end{equation}
since $G_n\neq 0$ implies $f(G_n)\ge 1$. This is true for all $f$ therefore, by the asymptotic spectrum duality,
\begin{equation}
G_m^{\strongproduct m\left\lceil\frac{n}{m}\right\rceil}
 \asymptoticle 2^{\left\lceil\frac{2\epsilon}{\ln 2} n\right\rceil}G_n^{\strongproduct n}.
\end{equation}
Similarly,
\begin{equation}
\begin{split}
\frac{f(G_n^{\strongproduct n})}{f\left(G_m^{\strongproduct m\left\lceil\frac{n}{m}\right\rceil}\right)}
 & = \frac{f(G_n)^n}{f(G_m)^{m\left\lceil\frac{n}{m}\right\rceil}}
 \le \frac{f(G_n)^n}{f(G_m)^{n-m}}
 \le M^m\left(\frac{f(G_n)}{f(G_m)}\right)^n
 \le M^m e^{2\epsilon n},
\end{split}
\end{equation}
therefore
\begin{equation}
G_n^{\strongproduct n}
 \asymptoticle 2^{\left\lceil m\log M+\frac{2\epsilon}{\ln 2} n\right\rceil}G_m^{\strongproduct m\left\lceil\frac{n}{m}\right\rceil}.
\end{equation}
The coefficient of $n$ in the exponent of $u$ can be made arbitrarily small by choosing $m$ large, and the inequalities for $n<m$ can be ensured by an additional constant factor, therefore the sequence $n\mapsto G_n^{\strongproduct n}$ is approximately geometric with respect to $\asymptoticle$.
\end{proof}
As explained in \cite{de2024asymptotic}, the space of continuous functions on $\Hom(\graphs,\nonnegativereals)$ provides an abstract way to realize a completion of $\graphs$ with repsect to the asymptotic spectrum distance. They show that, some Cauchy sequences admit infinite graphs as concrete models for their limit points. \Cref{prop:spectrumdistanceconvergence} can be seen as another (abstract) way to realize all limits as certain elements in $\completion{\graphs}$: the limit of the sequence $(G_n)_{n\in\naturals}$ (which we may assume to be either all nonzero, or zero precisely when $n\ge 1$) corresponds to the element $g\in\completion{\graphs}$ that generates the geometric sequence $(1,g,g^2,g^3,\dots)$ which is asymptotically equivalent to $(1,G_1,G_2^{\strongproduct 2},G_3^{\strongproduct 3},\dots)$. We can see from the proof that this element satisfies
\begin{equation}
\completion{f}(g)=\lim_{n\to\infty}f(G_n)
\end{equation}
for all $f\in\Hom(\graphs,\nonnegativereals)$.

The set of elements of $\completion{\graphs}$ arising in this way is a proper subsemiring. To see this, note that every nonempty graph $G$ is either a complete graph (which is cohomomorphically equivalent to $\complement{K_1}$), or it has at least one non-edge, which implies that $\complement{K_2}\le G$. Consequently, given an element $f$ of the asymptotic spectrum, a Cauchy sequence $(G_n)_{n\in\naturals}$ of nonempty graphs either satisfies $\lim_{n\to\infty} f(G_n)=1$ or $\lim_{n\to\infty} f(G_n)\ge 2$, which implies that the value of $\completion{f}$ on the corresponding element of $\completion{\graphs}$ is in $\{1\}\cup[2,\infty)$. However, the image of $\completion{f}$ is $\{0\}\cup[1,\infty)$, therefore there must be elements outside this subsemiring. A simple concrete example is $\complement{K_2}^{\strongproduct\frac{1}{2}}$.

\subsection{Majorization semirings}\label{sec:majorization}

We consider semirings of families of measures on finite sets and preorders induced by various versions of majorization. Let $I$ be an index set (finite, for simplicity), $\mathcal{X},\mathcal{X}'$ finite sets, and $P:I\to\nonnegativereals^{\mathcal{X}}$, $P':I\to\nonnegativereals^{\mathcal{X}'}$. We think of the $P_i$ and $P'_i$ as column vectors and $P,P'$ as matrices, with columns indexed by $I$. We say that $P$ majorizes $P'$ if there is a column-stochastic matrix $T$ such that $P'=TP$. $P$ and $P'$ are equivalent if there is a bijection $\phi:\mathcal{X}\to\mathcal{X}'$ such that $P_i(x)=P'_i(\phi(x))$ holds for all $x\in\mathcal{X}$. On the set of equivalent classes, we define addition and multiplication as columnwise direct sum $\oplus$ and columnwise Kronecker product $\otimes$. With the majorization preorder $\le$, this set becomes a preordered semiring.

In the case when the measures are probability measures, the elements can be interpreted as statistical experiments. The outcome of the experiment is a random sample drawn from a probability distribution that depends on the unknown parameter $i\in I$. Multiplication corresponds to performing experiments independently of each other, but with the same parameter value, while the majorization preorder formalizes the relation that one experiment is more informative than another.

The semiring as described above is not of polynomial growth, but it has subsemirings of polynomial growth that are formed by imposing conditions on the supports of the measures. In the following we restrict to the simplest case, where the supports are all assumed to be equal to $\mathcal{X}$. In this case $U:I\to\nonnegativereals^{\mathcal{X}}$ is power universal if and only if $\norm[1]{U_i}=1$ for all $i\in I$, and $U_i\neq U_{i'}$ for all $i,i'\in I$, $i\neq i'$. Semirings of families with more general support conditions have been studied in \cite{verhagen2025matrix}. In \cite{farooq2024matrix}, Farooq, Fritz, Haapasalo, and Tomamichel classify monotone homomorphisms into $\nonnegativereals$, $\nonnegativereals\opposite$, $\tropicalreals$, and $\tropicalreals\opposite$, and monotone derivations at degenerate homomorphisms to $\nonnegativereals$ (there are finitely many degenerate homomorphisms). The monotone homomorphisms into $\nonnegativereals$ and $\nonnegativereals\opposite$ are of the form
\begin{equation}
f_{\underline{\alpha}}(P)=\sum_{x\in\mathcal{X}}\prod_{i\in I}P_i(x)^{\alpha_i},
\end{equation}
where $\underline{\alpha}=(\alpha_i)_{i\in I}$ is a tuple of real numbers subject to the following conditions: all $\alpha_i\ge 0$ and $\sum_{i\in I}\alpha_i=1$ for homomorphisms into $\nonnegativereals\opposite$; $\alpha_i>0$ holds for exactly one $i\in I$, and $\sum_{i\in I}\alpha_i=1$ for homomorphisms into $\nonnegativereals$. The non-constant monotone homomorphisms into $\tropicalreals$ are of the form
\begin{equation}
f_{\underline{\beta}}(P)=\max_{x\in\mathcal{X}}\prod_{i\in I}P_i(x)^{\beta_i},
\end{equation}
where $\beta_i>0$ holds for exactly one $i\in I$ and $\sum_{i\in I}\beta_i=0$. There are no monotone homomorphisms into $\tropicalreals\opposite$. The degenerate homomorphisms into $\nonnegativereals$ are the norms $f_{e_i}(P)=\norm[1]{P_i}$. Up to interchangeability, the monotone derivations at $f_{e_i}$ are exactly the elements of the convex cone generated by the maps $P\mapsto\relativeentropy{P_i}{P_{i'}}$, $i'\in I\setminus\{i\}$.

The monotone homomorphisms and derivations can be expressed in a unified way in terms of the multivariate Rényi divergences
\begin{equation}
D_{\underline{\alpha}}(P)=\frac{1}{\max_{i\in I}\alpha_i-1}\log\sum_{x\in\mathcal{X}}\prod_{i\in I}P_i(x)^{\alpha_i},
\end{equation}
with $\sum_{i\in I}\alpha_i=1$ and either $0\le\alpha_i<1$ for all $i\in I$ or $\alpha_i>1$ and $\alpha_{i'}\le 0$ for some $i\in I$ and all $i'\in I\setminus\{i\}$. These quantities are monotone decreasing, and together with their pointwise limits, give the same functionals as above, up to logarithm and rescaling.

By applying \cite[1.2. Theorem]{fritz2021abstract2}, this classification of homomorphisms into the semirings $\mathcal{K}=\{\nonnegativereals,\nonnegativereals\opposite,\tropicalreals,\tropicalreals\opposite,\dualnumbers\}$ results in a concrete necessary and sufficient condition for matrix majorization of powers up to a sublinear number of power universal elements. \Cref{cor:dualityderivations} directly allows us to extend this to approximately geometric sequences:
\begin{proposition}
Let $(P_n)_{n\in\naturals}$ and $(Q_n)_{n\in\naturals}$ be approximately geometric sequences of $I$-indexed families of probability distributions, and let $U$ be a power universal. Then the multivariate Rényi divergence rates
\begin{equation}
\completion{D}_{\underline{\alpha}}(P):=\lim_{n\to\infty}\frac{1}{n}D_{\underline{\alpha}}(P_n)
\end{equation}
exist (similarly defined for $Q$, and including the pointwise limits), and the following are equivalent:
\begin{enumerate}
\item there is a sublinear sequence $(k_n)_{n\in\naturals}$ such that $U^{\otimes k_n}\otimes P_n\ge Q_n$ for all $n\in\naturals$
\item $\completion{D}_{\underline{\alpha}}(P)\ge \completion{D}_{\underline{\alpha}}(Q)$ for all $\underline{\alpha}$ in the above range.
\end{enumerate}
\end{proposition}
Similar extensions can be derived for the characterizations of asymptotic majorization between probability distributions with weaker support conditions \cite{verhagen2025matrix}, distributions on standard Borel spaces \cite{haapasalo2026multivariate}, and a special set of pairs of quantum states \cite{verhagen2026conditions}.

In \cite{renes2016relative}, Renes introduced a modification of majorization (with two columns, known as relative majorization) in connection with applications to quantum thermodynamics and hypothesis testing, called relative submajorization, an asymptotic version of which (as well as its quantum generalization) was later studied using preordered semirings \cite{perry2022semiring}, and subsequently extended to more general index sets \cite{bunth2021asymptotic,bunth2023equivariant}. For finite index sets $I,J$, the semiring consists of equivalence classes of pairs $(P,Q)$ where $P:I\to\nonnegativereals^\mathcal{X}$ and $Q:J\to\nonnegativereals^\mathcal{X}$, operations defined similarly as above, and the \emph{relative submajorization preorder}: $(P,Q)\ge(P',Q')$ if there exists a column-substochastic matrix $T$ such that the entrywise inequalities $P'\le TP$ and $Q'\ge TQ$ hold. Assuming full support for simplicity, the resulting preordered semiring satisfies the polynomial growth property. An example of a power universal element is $u=(p,q)$ with $\lvert\mathcal{X}\rvert=1$, $p=(2,2,\dots,2)$ and $q=(1,1,\dots,1)$. In this semiring, the zero element is the unique tuple with $\mathcal{X}=\emptyset$, and the unit is $((1,1,\dots,1),(1,\dots,1))$. We allow for $T$ the zero matrices with $0$ rows and columns indexed by an arbitrary finite set, which in particular implies $0\le 1$.

In the context of asymptotic hypothesis testing, relative submajorization is well suited for the study of (strong converse) error exponents. Given a pair $(P,Q)$ where $P:I\to\nonnegativereals^\mathcal{X}$ and $Q:J\to\nonnegativereals^\mathcal{X}$ map to probability distributions, a (deterministic) test is a subset $A\subseteq\mathcal{X}$, which is interpreted as follows. After drawing a sample $x\in\mathcal{X}$ from an unknown distribution that is either some $P_i$ or $Q_j$, we decide that the distribution was one of the $P_i$ if $x\in A$ (null hypothesis), and one of $Q_j$ otherwise (alternative hypothesis). Note that $P$ and $Q$ are considered as two composite hypotheses, and the tester is not required to find out which $i\in I$ or $j\in J$ corresponds to the true distribution. The test $A$ can be encoded in a single-row matrix $T$ with columns indexed by $\mathcal{X}$, and a $1$ entry for each $x\in A$ and $0$ entries for $x\notin A$.

More generally, we allow the test to be a randomized procedure, which can be described by a column-substochastic matrix $T$ with a single row containing the probabilities of accepting the null hypothesis for each outcome in $\mathcal{X}$. In this case, the elements of the row vectors $TP$ and $TQ$ are probabilities that the null hypothesis is accepted, for every possible distribution. We distinguish between two cases of incorrect decision. A type I error occurs if the tester reports $Q$ but the distribution was drawn from one of the $P_i$, while the opposite case is the type II error. Given a test $T$ and hypotheses $(P,Q)$, the worst-case probability of a type I error is
\begin{equation}
\alpha(T)=\max_{i\in I}(1-TP_i),
\end{equation}
and the worst-case probability of a type II error is
\begin{equation}
\beta(T)=\max_{j\in J}TQ_j.
\end{equation}
Note that since $T$ is substochastic, by the definition of relative submajorization, we have
\begin{equation}
(P,Q)\ge((1-\alpha(T),1-\alpha(T),\dots,1-\alpha(T)),(\beta(T),\dots,\beta(T))).
\end{equation}
Conversely, if $(P,Q)\ge((a,a,\dots,a),(b,\dots,b))$, then there exists a (randomized) test $T$ such that $\alpha(T)\le 1-a$ and $\beta(T)\le b$.

In general, the two error probabilities cannot be both arbitrarily small. In asymptotic hypothesis testing, we are interested in the scaling of the error probabilities as we draw more samples (i.i.d.\ or more general, described by the joint distributions $(P_n,Q_n)$). We will require an exponentially low $\beta_n(T_n)$ in the number of samples and tune the exponent to a desired value $r$, i.e., $\beta_n(T_n)=2^{-rn+o(n)}$, and study the asymptotics of smallest $\alpha_n(T_n)$. In the simple i.i.d.\ case ($\lvert I\rvert=\lvert J\rvert=1$), the Stein lemma gives a threshold $r_0=\relativeentropy{P}{Q}$ such that we can set $r=r_0$ and have $\alpha_n(T_n)\to 0$, but, if $r>r_0$, then $\alpha_n(T_n)\to 1$ exponentially (strong converse). In this case the optimal scaling is given by the lowest $R$ such that $\alpha_n(T_n)=1-2^{-nR+o(n)}$ for some test sequence.

We define the exponents in an analogous way with the worst-case error for composite and correlated hypotheses. By the preceding discussion, if the hypotheses are described by an approximately geometric sequence $(P,Q)=((P_n,Q_n))_{n\in\naturals}$, then the condition that $(R,r)$ is a valid exponent pair for some test sequence can be expressed in the relative submajorization semiring as $(P,Q)\asymptoticge(2^{-R},2^{-r})$.

The asymptotic relation can be characterized using \cref{cor:dualitytropicalhomomorphisms} (and if $\lvert I\rvert=\lvert J\rvert$, also \cref{cor:dualityrealhomomorphisms}). The relevant spectrum is a subset of the homomorphisms $f_{\underline{\alpha}}$ and $f_{\underline{\beta}}$ given above. This was determined in \cite{perry2022semiring,bunth2021asymptotic,bunth2023equivariant}, and can be parametrized as follows. For a probability vector $\gamma\in\distributions(J)$, $i\in I$, and $\alpha\ge 1$, we set
\begin{equation}
f_{\alpha,i,\gamma}(P,Q)=\sum_{x\in\mathcal{X}}P_i(x)^\alpha\prod_{j\in J}Q_j(x)^{(1-\alpha)\gamma(j)}
\end{equation}
and
\begin{equation}
f_{i,\gamma}(P,Q)=\max_{x\in\mathcal{X}}P_i(x)\prod_{j\in J}Q_j(x)^{\gamma(j)}.
\end{equation}
The set of all such $f_{\alpha,i,\gamma}$ are precisely the homomorphisms into $\nonnegativereals$, while the $f_{i,\gamma}(P,Q)$ are the homomorphisms into $\tropicalreals$, up to normalization. Again, this can be written in a unified way as
\begin{equation}
D_{\alpha,i,\gamma}(P,Q)=\relativeentropy[\alpha]{P_i}{\textstyle\prod_{j\in J}Q_j(x)^{\gamma(j)}}=\frac{1}{\alpha-1}\log f_{\alpha,i,\gamma}(P,Q),
\end{equation}
together with the limits
\begin{equation}
D_{\infty,i,\gamma}(P,Q):=\lim_{\alpha\to\infty}D_{\alpha,i,\gamma}(P,Q)=\relativeentropy[\infty]{P_i}{\textstyle\prod_{j\in J}Q_j(x)^{\gamma(j)}}=\log f_{i,\gamma}(P,Q).
\end{equation}

The classification together with \cref{cor:dualitytropicalhomomorphisms} implies the following result:
\begin{proposition}\label{prop:approximatelygeometricstrongconverse}
Let $(P,Q)=((P_n,Q_n))_{n\in\naturals}$ be an approximately geometric sequence in the semiring with the relative submajorization preorder. Then for all $\alpha\in[1,\infty]$, $i\in I$, and $\gamma\in\distributions(J)$ the multivariate Rényi divergence rates
\begin{equation}
\completion{D}_{\alpha,i,\gamma}(P,Q)
 := \lim_{n\to\infty}\frac{1}{n}D_{\alpha,i,\gamma}(P_n,Q_n)
  = \frac{1}{\alpha-1}\log \completion{f_{\alpha,i,\gamma}}(P,Q)
\end{equation}
exist, and the optimal strong converse exponent $R^*(r)$ for the type I error as a function of the decay rate $r$ of the type II error probability is
\begin{equation}
R^*(r)=\sup_{\alpha\ge 0}\max_{\substack{i\in I  \\  \gamma\in\distributions(J)}}\frac{\alpha-1}{\alpha}\left[r-\completion{D}_{\alpha,i,\gamma}(P,Q)\right].
\end{equation}
\end{proposition}

\begin{theorem}
Let $P_i(x)$, $Q_j(x)$ initial probability distributions with full support on some finite set $\mathcal{X}$, and let $P_i(x_2|x_1)$ and $Q_j(x_2|x_1)$ be matrices of strictly positive transition probabilities so that the joint distributions of $n$ samples are
\begin{align}
P_{i,n}(x_0,x_1,x_2,\dots,x_n) & = P_i(x_0)P_i(x_1|x_0)P_i(x_2|x_1)\cdots P_i(x_n|x_{n-1}),  \\
Q_{j,n}(x_0,x_1,x_2,\dots,x_n) & = Q_j(x_0)Q_j(x_1|x_0)Q_j(x_2|x_1)\cdots Q_j(x_n|x_{n-1}).
\end{align}
Consider the composite hypothesis testing problem with samples drawn from sets of Markov processes $P$ and $Q$. Imposing a bound of $2^{-nr+o(n)}$ on the type II error probability, the optimal strong converse exponent $R^*(r)$ for the type I error is
\begin{equation}
R^*(r)=\sup_{\alpha\ge 0}\max_{\substack{i\in I  \\  \gamma\in\distributions(J)}}\frac{\alpha-1}{\alpha}\left[r-\frac{1}{\alpha-1}\log\rho\left(\left(\textstyle P_i(x|x')^\alpha\prod_{j\in J}Q_j(x|x')^{(1-\alpha)\gamma(j)}\right)_{x,x'\in\mathcal{X}}\right)\right],
\end{equation}
where $\rho$ denotes the spectral radius.
\end{theorem}
\begin{proof}
Let $S$ be the semiring of equivalence classes of families (indexed by $I$ and $J$) of measures with full support on finite sets with the relative submajorization preorder.
Let $e\in S^{\mathcal{X}}$ be the column vector with entries $1\in S$, $A$ the matrix whose $(x,x')$ entry is the tuple $((P_i(x|x'))_{i\in I},(Q_j(x|x'))_{j\in J})$, considered as measures on a singleton set, and let $y\in S^{\mathcal{X}}$ be the column vector with entries $((P_i(x))_{i\in I},(Q_j(x))_{j\in J})_{x\in\mathcal{X}}$. Then $A^ny\in S^{\mathcal{X}}$ is a column vector whose $x$ entry contains the probabilities of all strings in $\mathcal{X}$ ending with $x$, under all the Markov processes above. Consequently, $e^TA^ny$ contains the full joint distributions of the first $n$ samples, for all $i\in I$ and $j\in J$.

By \cref{prop:linearrecursionapproximatelygeometric}, the sequence $(e^TA^ny)_{n\in\naturals}$ is approximately geometric, and it represents the element $\rho(A)$. The claim follows from \cref{prop:approximatelygeometricstrongconverse,ex:spectralradius}.
\end{proof}
In the special case when $\lvert I\rvert=\lvert J\rvert$, a different expression for the strong converse exponent is given in \cite{nakagawa1993converse}.

\section*{Acknowledgement}

I thank Péter Frenkel for the idea of considering linear recursions, and ChatGPT for valuable discussion on proof ideas for the implication \ref{it:approxgeomproductsepsilon}$\implies$\ref{it:approxgeomproducts} in \cref{prop:approxgeomcharacterizations} and encouragement. This work was supported by the János Bolyai Research Scholarship of the Hungarian Academy of Sciences, and the Ministry of Culture and Innovation of Hungary from the National Research, Development and Innovation Fund via the research grants FK~146643, K~146380, and EXCELLENCE~151342.

\bibliography{refs}{}

\end{document}